\documentclass[11pt,reqno]{amsproc}

\title{Proof of the transverse instability of Stokes waves}

\author{Ryan P. Creedon}
\address{Department of Applied Mathematics, University of Washington, Seattle, WA 98195}
\email{creedon@uw.edu}

\author{Huy Q. Nguyen}
\address{Department of Mathematics, University of Maryland, College Park, MD 20742}
\email[H. Nguyen]{hnguye90@umd.edu}

\author{Walter A. Strauss}
\address{Department of Mathematics, Brown University, Providence, RI 02912}
\email{wstrauss@math.brown.edu}

\usepackage[margin=1in]{geometry}
\usepackage{amsmath, amsthm, amssymb, mathrsfs}
\usepackage{times}
\usepackage{color}
\usepackage{hyperref}
\usepackage{multirow}
\usepackage{graphicx}
\usepackage{tikz}
\usetikzlibrary{decorations.pathmorphing}

\newcommand{\bq}{\begin{equation}}
\newcommand{\eq}{\end{equation}}
\newcommand{\bqa}{\begin{eqnarray*}}
\newcommand{\eqa}{\end{eqnarray*}}

\theoremstyle{plain}
\newtheorem{theo}{Theorem}[section]
\newtheorem{prop}[theo]{Proposition}
\newtheorem{lemm}[theo]{Lemma}
\newtheorem{coro}[theo]{Corollary}

\newtheorem{defi}[theo]{Definition}
\theoremstyle{definition}
\newtheorem{rema}[theo]{Remark} 
\newtheorem{nota}[theo]{Notation}
\DeclareMathOperator{\cnx}{div}

\DeclareMathOperator{\I}{I}

\DeclareSymbolFont{pletters}{OT1}{cmr}{m}{sl}
\DeclareMathSymbol{s}{\mathalpha}{pletters}{`s}

\def\tt{\theta}
\def\eps{\varepsilon}
\def\na{\nabla}

\def\wh{\widehat}
\def\g{\gamma}
\def\lb{\lambda}
\def\mez{\frac{1}{2}}
\def\tdm{\frac{3}{2}}

\def\Rr{\mathbb{R}}

\def\Zz{\mathbb{Z}}

\def\cG{\mathcal{G}}
\def\cJ{\mathcal{J}}
\def\cK{\mathcal{K}}

\def\cL{\mathcal{L}}
\def\L1{\mathcal{L}^{(1)}}
\def\L2{\mathcal{L}^{(2)}}
\def\L3{\mathcal{L}^{(3)}}
\def\cU{\mathcal{U}}
\def\cO{O}
\def\cH{\mathcal{H}}
\def\cV{\mathcal{V}}
\def\ld{\lambda}

\def\p{\partial}

\def\na{\nabla}

\def\ka{\kappa}

\def\ol{\overline}
\def\T{\mathbb{T}}
\def\Tt{\Theta}
\def\wt{\widetilde}

\def\ka{\kappa}
\def\Om{\Omega}
\def\ld{\lambda}

\numberwithin{equation}{section}

\begin{document}
\begin{abstract}
A Stokes wave is a traveling free-surface periodic water wave  that is constant in the direction transverse to the direction of propagation. 
In 1981 McLean discovered via numerical methods that Stokes waves at infinite depth are unstable with respect to transverse perturbations of the initial data. Even for a Stokes wave that has very small amplitude $\eps$, we prove rigorously that transverse perturbations, 
after linearization, will lead to exponential growth in time.  To observe this instability, extensive calculations are required all the way up to order $O(\eps^3)$.  
All previous rigorous results of this type were merely two-dimensional, in the sense that they only treated long-wave perturbations in the longitudinal direction. This is the first rigorous proof of three-dimensional instabilities of Stokes waves. 
\end{abstract}

\keywords{Stokes waves, gravity  waves, transverse instabilities, unstable eigenvalues, infinite depth}

\noindent\thanks{\em{ MSC Classification: 76B07, 35Q35,  35R35, 35C07, 35B35.}}

\maketitle

\section{Introduction} 
We consider classical  water waves that are irrotational, incompressible,  and inviscid.  
The water lies below an unknown free surface $S$. 
 Such waves have  been studied for over two centuries, notably by Stokes \cite{Stokes}.  
A {\it Stokes wave} is a two-dimensional steady wave traveling in a fixed horizontal direction at a fixed speed $c$. It has been known for a century that 
a curve of small-amplitude Stokes waves exists \cite{Nekrasov, Civita, Struik}.  
Several decades ago it was proven that the curve extends to large amplitudes as well \cite{Keady}.

In 1967 Benjamin and Feir \cite{BenjFeir} discovered, to the general surprise of the fluids community, that a {\it small} 
long-wave  perturbation of a {\it small} Stokes wave  in the same direction of propagation will lead to exponential instability.  
This is known as the {\it modulational} (or Benjamin-Feir or sideband) {\it instability}, a phenomenon 
 whereby deviations from a periodic wave are reinforced by the nonlinearity, leading to the eventual breakup of the wave into a train of pulses.   Rigorous proofs of the modulational instability  were discovered by Bridges and Mielke \cite{BM} in the case of finite depth, provided  the depth is larger than a critical depth $d_0$, and by  two of the current authors \cite{NguyenStrauss} for infinite depth.  
A more detailed description of the instability, including the figure-8 pattern of the 
unstable eigenvalues, was found numerically in \cite{DecOli} and asymptotically by another of the current authors \cite{CreDec}.  
This detailed description was proven rigorously by Berti \emph{et al}, first in the deep water case \cite{Berti1} 
and then in the finite depth case \cite{Berti2} when the depth is larger than $d_0$.  Recently the much more subtle 
critical depth case was treated in \cite{Berti3}. 

 A different type of instability due to perturbations in the same direction of propagation (\emph{i.e.}, the longitudinal direction) 
was  detected in the numerical work of McLean \cite{McLeanDeep,McLeanFinite}. They are called  {\it high-frequency instabilities} because they develop away from the origin of the complex plane, appearing as small isolas (bubbles)
centered on the imaginary axis.  In contrast to modulational instabilities, high-frequency instabilities occur at all values of the depth. The first plot of the high-frequency instabilities was due to Deconinck and Oliveras \cite{DecOli}, thirty years after McLean's work. Among the challenges in plotting these instabilities was to find the longitudinal wave numbers of the perturbation that correspond to each high-frequency isola, which exist in narrow intervals that drift as the amplitude of the Stokes wave increases. In \cite{CreDecTri}, a perturbation method was developed to obtain an asymptotic expansion of these intervals in addition to an asymptotic expansion of the maximum growth rates of the high-frequency instabilities.  This  revealed for the first time analytically that such instabilities can grow faster than the modulational instability at certain finite depths. These high-frequency results have since been made rigorous in the recent work \cite{HurYang}.

Both modulational and high-frequency instabilities  are created by longitudinal perturbations 
that have {\it different periods} compared to that of the Stokes waves.  
On the other hand, what was unanswered was 
whether a small Stokes wave could be unstable when perturbed in both horizontal directions but keeping the longitudinal period unperturbed. This {\it transverse instability} problem was studied numerically first by Bryant \cite{Bryant} and was 
followed by much more detailed work of McLean {\it et al} \cite {MMMSY, McLeanDeep, McLeanFinite}. 
While these remarkable papers did detect transverse instabilities, a mathematical proof has been missing ever since.  This problem is truly three-dimensional.
{\it The purpose of the current paper is to provide the first rigorous proof of transverse instabilities of small Stokes waves. }

Before discussing this paper, it is important to note that there are many other models of water waves for which the transverse instability has been studied rigorously. 
One such model includes the presence of surface tension, that is, gravity-capillary waves.  
{\it However, it should be kept in mind that the presence of surface tension 
drastically changes the mathematical problem. } 
The transverse instability for solitary (non-periodic) waves in such a model was rigorously 
discussed by a number of authors, including Bridges \cite{Br}, 
Pego and Sun \cite{PegoSun1} and  Rousset and Tzvetkov \cite{RouTzv}.  
The transverse instability for periodic waves in this model was recently studied 
by Haragus, Truong and Wahlen \cite{HTW}.

With these results in mind, we now specify the parameters of our problem. Let $x$ and $y$ denote the horizontal variables 
and $z$ the vertical one.  
For simplicity, we assume here that the depth is infinite.  
We are confident that our proof generalizes to the finite-depth case. Consider the curve of Stokes waves traveling in the $x$-direction and with a given period, say $2\pi$ without loss of generality.  This curve is parametrized by a small parameter $\eps$ which represents the wave  amplitude of the Stokes waves.  
Such a steady wave can be described in the moving $(x,z)$ plane 
(where $x-ct$ is replaced by $x$) 
by its free surface 
$S=\{(x,y,z)\ |\ z=\eta^*(x;\eps)\}$ and by its velocity potential $\psi^*(x;\eps)$ restricted to $S$.

The perturbation of $\eta^*$ takes the form $\overline{\eta}(x)e^{\lambda t + i\alpha y}$, where $\overline{\eta}$ has the {\it same period} $2\pi$ as the Stokes wave, $\lambda \in \mathbb{C}$ is the growth rate of the perturbation, and $\alpha\in \Rr$ is the transverse  wave number  of the perturbation. The problem is to find at least one value of $\alpha$  that leads to instability, that is, $\text{Re}\lambda > 0$ . 
After linearizing the nonlinear water wave system about a Stokes wave, introducing a ``good-unknown,'' and performing a conformal mapping  change of variables, we find that the exponents $\lambda$ are eigenvalues of a linear operator $\mathcal{L}_{\varepsilon,\beta}$, where $\beta = \alpha^2$. Motivated by \cite{MMMSY}, we  first determine a {\it resonant transverse wave number} $\alpha_*$ so that the unperturbed operator $\cL_{0, \beta_*}$ with $\beta_* = \alpha_*^2$ has an imaginary double eigenvalue $\lambda_0 = i\sigma$. 
 This eigenvalue corresponds to the lowest-possible resonance that generates a Type II transverse instability according to McLean \cite{MMMSY}, of which there are infinitely many higher-order resonances that have potential to generate higher-order transverse instabilities. We expect however that higher-order transverse instabilities have smaller growth rates for small Stokes waves. In order to capture the transverse instabilities we introduce a new small parameter $\delta$ for the perturbation of $\beta$ about $\beta_*$. Our main result is that the perturbed operator $\cL_{\eps, \beta_*+\delta}$ has eigenvalues $\ld_\pm$ 
with non-zero real parts that bifurcate from $\lambda_0$, stated more precisely in the following theorem.  

\begin{theo} \label{theo:main}
There exist $\varepsilon_{\textrm{max}} > 0$ and $\delta_{\text{max}}>0$  such that  
for all $ \varepsilon \in (- \varepsilon_{\textrm{max}},  \varepsilon_{\textrm{max}})$ and $\delta\in(-\delta_{\text{max}},\delta_{\text{max}})$, 
the operator $\mathcal{L}_{\varepsilon,\beta_*+\delta}$ has a pair of eigenvalues
\begin{align}
\lambda_{\pm} &=i\Big(\sigma + \frac{1}{2}T(\varepsilon,\delta) \Big) \pm \frac12 \sqrt{\Delta(\varepsilon,\delta)}, \label{lambda_exact1}
\end{align}
where $T$ and $\Delta$ are real-valued, real-analytic functions such that $T(\varepsilon,\delta) = O\left(\delta \right)$ and $\Delta(\varepsilon,\delta) = O\left(\delta^2\right)$ as $(\varepsilon,\delta) \rightarrow (0,0)$. Furthermore, there exist $\ka_0 \in \mathbb{R}$ and $\ka_1>0$ such that for 
\begin{align}
\delta = \delta(\varepsilon,\theta) = \kappa_0\varepsilon^2 + \theta\varepsilon^3 \quad \textrm{with} \quad |\theta|<\kappa_1,
\end{align}
we have $\Delta(\varepsilon,\delta(\varepsilon,\theta))>0$ for sufficiently small $\varepsilon$.
Thus, the eigenvalue $\lambda_+$  has positive real part provided $\delta = \delta(\varepsilon,\theta)$ with $|\theta|<\kappa_1$ and $\varepsilon$ is sufficiently small. Moreover, $\text{Re}\lambda_+ = O\left(\varepsilon^3\right)$ as $\varepsilon \rightarrow 0$ for each $\theta$.
This means that there exist transverse perturbations of the given Stokes wave whose amplitudes 
grow temporally like $e^{t\,\text{Re}\ld_+ }$.
\end{theo}

Substituting $\delta = \delta(\varepsilon,\theta)$ into \eqref{lambda_exact1} and dropping terms of $O\left(\varepsilon^4\right)$ and smaller, we obtain an asymptotic expansion of the unstable eigenvalues. By eliminating $\theta$ from this expansion in favor of its real and imaginary parts, denoted $\lambda_r$ and $\lambda_i$, respectively, we find that the eigenvalues lie approximately on the ellipse
\begin{equation}
\frac{4.085\lambda_r^2}{\varepsilon^6} + \frac{86.059\left(\lambda_i+0.389-0.467\varepsilon^2\right)^2}{\varepsilon^6} = 1, \label{ellipsenum}
\end{equation}
where we have numerically evaluated coefficients for ease of readability. The center of this ellipse drifts from the double eigenvalue $i\sigma  \approx -0.389i$  along the imaginary axis like $O\left(\varepsilon^2\right)$, while its semi-major and semi-minor axes scale like $O\left(\varepsilon^3\right)$.  We refer to Corollary \ref{cor:ellipse} for the precise statement and the left panel of Figure \ref{fig1} for a schematic. 

We compare \eqref{ellipsenum} with numerical computations of the unstable eigenvalues obtained from the Floquet-Fourier-Hill method applied to the Ablowitz-Fokas-Musslimani formulation of the transverse spectral problem, see \cite{DecOli15} for details. The right panel of Figure \ref{fig1} shows the results of these numerical computations on a Stokes wave with amplitude parameter $\varepsilon = 0.01$. Also plotted is the corresponding asymptotic ellipse \eqref{ellipsenum}. 
The difference between the asymptotic and numerical results is $O\left(\varepsilon^4\right)$, demonstrating agreement between the theoretical results and the numerical computations to $O\left(\varepsilon^3\right)$, as desired. Even better agreement can be found by retaining higher-order corrections of the unstable eigenvalues in a manner similar to \cite{CreDecTri}.

The isola of unstable eigenvalues found above is reminiscent of the high-frequency isolas that appear in the longitudinal stability spectrum. It is therefore natural to compare the growth rates of the transverse instability obtained in this work to the known growth rates of the longitudinal instabilities of Stokes waves, including the high-frequency and Benjamin-Feir instabilities. In the infinite depth longitudinal case, the largest high-frequency isola has semi-major and semi-minor axes that scale like $O\left(\varepsilon^4\right)$ \cite{CreDecTri}.  Thus our transverse instability grows at a {\it faster} rate $O(\eps^3)$ for sufficiently small 
amplitude waves.  
On the other hand, our instability grows {\it slower} than the Benjamin-Feir instability rate, which is
$O\left(\varepsilon^2\right)$ in both finite and infinite depth 
\cite{BM, NguyenStrauss, Berti1,Berti2, HurYang, CreDec}. Moreover, our instability grows {\it slower} than the largest high-frequency instability in finite depth, which grows like $\cO(\varepsilon^2)$ \cite{CreDecTri, HurYang}.

\begin{figure}[tb]
    \centering
    \includegraphics[width=10cm]{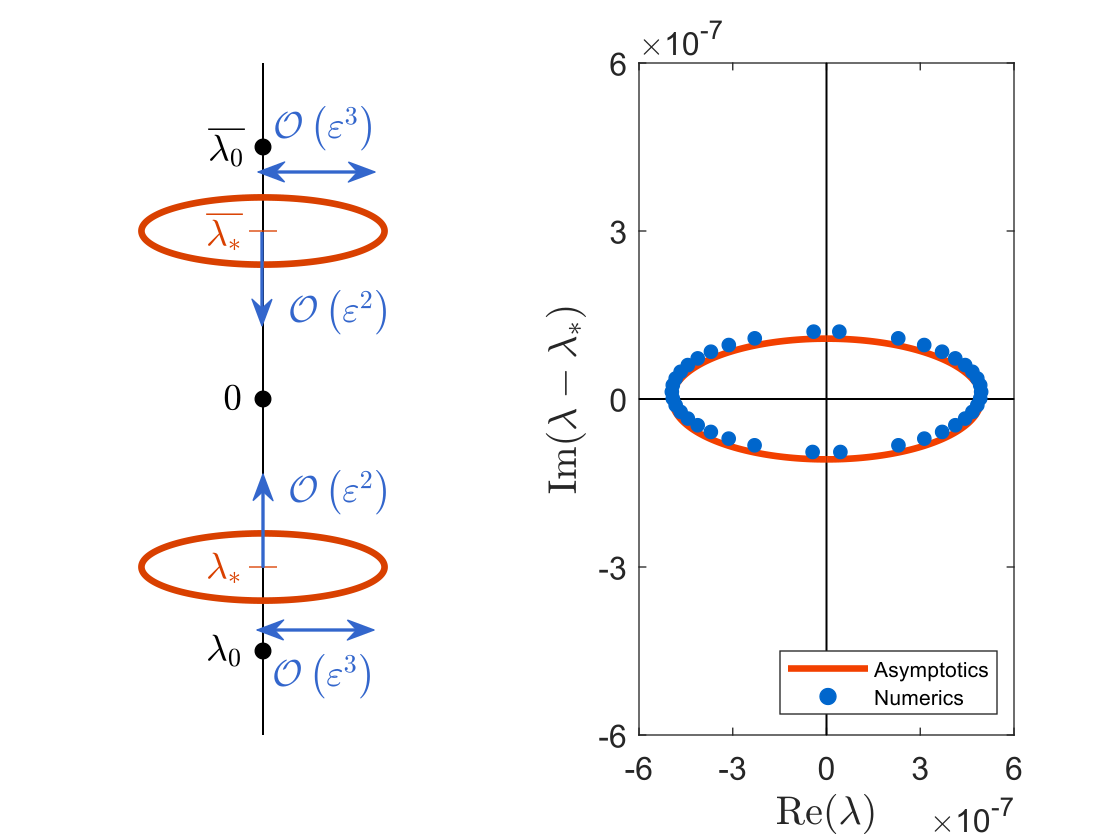}
    \caption{(Left) A schematic of the transverse instability isolas (orange curves) of width $O\left(\varepsilon^3\right)$ drifting from $\lambda_0$ like $O\left(\varepsilon^2\right)$. Here, $\lambda_*$ represents the center of the isola.  (Right) A comparison of the asymptotic approximation of the transverse instability isola \eqref{ellipsenum} (orange curve) and numerical computations of the unstable eigenvalues (blue dots) when $\varepsilon = 0.01$. The center of the isola is subtracted from the imaginary part to show a sense of scale. The difference between the numerical and asymptotic results is $O\left(\varepsilon^4\right)$. }
    \label{fig1}
\end{figure}

With the preceding discussion in mind, we now turn to the  main ideas in the proof of Theorem \ref{theo:main}.  
It  begins by finding an  expression for 
$\cL_{\eps,\beta}$ by a method analogous to that in \cite{NguyenStrauss}. However, for the present three-dimensional instability problem, $\cL_{\eps,\beta}$ involves a genuine pseudo-differential operator  as opposed to the Fourier multiplier  in the two-dimensional problem considered in \cite{NguyenStrauss, Berti1}. The proof continues by following the method of \cite{Berti1} that uses a Kato similarity transformation 
to reduce the relevant spectral data of $\cL_{\eps,\beta}$ to a $2 \times 2$ matrix $\textrm{L}_{\eps, \delta}$ with the property that $i\textrm{L}_{\eps, \delta}$ is real and skew-adjoint.  Then we show that the entries of this matrix are real analytic functions of $\varepsilon$ 
and $\delta$ and we obtain convenient functional expressions for its eigenvalues, resulting in \eqref{lambda_exact1}. In order to conclude that $\Delta(\varepsilon,\beta)>0$ for $\delta = \delta(\varepsilon,\theta)$ and sufficiently small $\varepsilon$, 
we must expand the entries of the matrix in a power series up to third order in the pair 
$(\varepsilon,\delta)$. If the expansions are terminated before third order, 
one finds $\Delta(\varepsilon,\delta) \leq 0$ for any choice of $\delta$, which is insufficient for eigenvalues with positive real part.  With the third-order expansions, however, we are able to show that $\Delta(\varepsilon,\delta) > 0$ if $\delta=O(\eps^2)$ is chosen appropriately.  

The problem solved in this paper turned out to be considerably more difficult than we had anticipated.  
Originally we began by attempting to take the transverse perturbation at a fixed period, that is, $\delta=0$.  
For the reasons stated above, that approach did not yield an instability.  
We also found that certain residue calculations related to the eigenfunction expansions of the unstable eigenvalues led to more non-zero terms than we had initially thought. This, coupled with the introduction of the small parameter $\delta$, led to extremely arduous calculations, 
so we took advantage of Mathematica to carry out the longest ones. They can be found in the companion Mathematica file \emph{CompanionToTransverseInstabilities.nb}.

In Section 2 we introduce the Stokes waves and proceed with the  linearization and then the flattening by means of a conformal mapping.  The main result here is Theorem \ref{theo:flattenG}, which is devoted to the three-dimensional Dirichlet-Neumann operator  under a two-dimensional conformal mapping, and the proof of analyticity in $\eps$ and $\delta$.  Theorem \ref{theo:main} is then reduced to studying the eigenvalues of the linearized operator $\cL_{\eps,\beta}$, which has a Hamiltonian form and is reversible.  
Section 3 is devoted to a discussion of the first resonance on the imaginary axis when $\eps=0$, the introduction of Kato's perturbed basis, 
and the reduction to the study of the eigenvalues of a $2\times 2$ matrix $\textrm{L}_{\eps,\delta}$.  
In Section 4 we perform lengthy expansions of $\cL_{\eps,\beta_*+\delta}$ out to third order in both $\eps$ and $\delta$; this is a major new  difficulty compared to the two-dimensional instabilities studied in \cite{NguyenStrauss, Berti1}. 
In Section 5 we use the expansions of $\cL_{\eps,\beta_*+\delta}$ to compute the expansions of the Kato basis vectors and of the matrix $\textrm{L}_{\eps,\delta}$. 
Finally in Section 6 we analyze the leading terms in the characteristic discriminant of $\textrm{L}_{\eps,\delta}$.  An important step is to prove that a key coefficient, which we call $b_{3,0}$, does not vanish.  This implies the existence of instabilities.  The full derivation of the ellipse \eqref{ellipsenum} is given in Corollary \ref{cor:ellipse}.
\begin{nota}
We fix the following notation throughout this paper: 
\begin{itemize}
\item $\T=\Rr/(2\pi \mathbb{Z})$.
\item  For $f,g \in \left(\textrm{L}^2(\T)\right)^2$, $(f, g)=\int_{\T} f(x) \cdot ~\overline{g}(x)dx$, the `bar' denoting complex conjugation  and the `dot' denoting the real dot product.
\item  If $\eps$ is a small parameter, we write $a=O(\eps^m)$ if for some $C>0$,  $|a|\le C|\eps^m|$ for all sufficiently small $|\eps|$. 
\end{itemize}
\end{nota}
\section{Transverse perturbations of Stokes waves}
\subsection{Stokes waves}
We consider the fluid domain 
\bq
D(t)=\{(x, y,z)\in \Rr^3\ :\ ~z<\eta(x,y, t)\}.
\eq
 below the  free surface $S=\{(x ,y, \eta(x,y, t)): (x,y)\in \Rr^3 \}$ to have infinite depth.  Assuming that the fluid is incompressible,  inviscid  and irrotational, the velocity field admits a harmonic 
potential $\phi(x, y, t): D(t)\to \Rr$. Then $\phi$ and $\eta$ satisfy the water wave system
\bq\label{ww:0}
\begin{cases}
 \Delta_{x, y,z} \phi =0\quad \text{ in } D(t), \\
\p_t\phi + \tfrac12 |\na_{x, y,z}\phi|^2 =- g\eta +P\quad \text{ on } S, \\
 \p_t\eta+\p_x\phi\p_x\eta=\p_y\phi\quad  \text{ on } S, \\
 \na_{x, y,z}\phi \to 0 \text{ as } z\to -\infty,
\end{cases}
\eq
where $P$ is the Bernoulli constant and $g>0$ is the constant acceleration due to gravity. The second equation is Bernoulli's,  which follows from the 
pressure being constant along the free surface; the third equation expresses the kinematic boundary condition that particles on the surface remain there; 
the last condition asserts that the water is quiescent at great depths.  For convenience we will take $P=0$. 

 In order to reduce the system to the free surface $S$, we introduce the Dirichlet-Neumann operator $G(\eta)$ 
 associated to $\Omega$, namely, 
\bq\label{def:Gh}
(G(\eta)f)(x,y) = \p_z\vartheta(x,y, \eta(x,y))-\na_{x,y}\vartheta(x,y, \eta(x,y)) \cdot \na_{x,y}\eta(x,y),
\eq
where $\vartheta(x, y,z)$ solves the elliptic problem 
\bq\label{elliptic:G}
\begin{cases}
\Delta_{x, y,z}\vartheta=0\quad\text{in}~D,\\
\vartheta\vert_{z=\eta(x,y)}=f(x,y),\quad \na_{x, y,z} \vartheta\in L^2(D).
\end{cases}
\eq
 We define $\psi$ as the trace of the velocity potential on the free surface, $\psi(t, x,y)=\phi(t, x,y, \eta(t, x,y))$. Then, in the moving frame with speed $c\in\Rr^2$, the gravity water wave system 
written in the Zakharov-Craig-Sulem formulation \cite{Zak, CraSul} is 
 \bq\label{ww:c}
 \begin{cases}
 \p_t \eta=c \cdot \na_{x,y}\eta+ G(\eta)\psi,\\
 \p_t\psi=c \cdot \na_{x,y}\psi-\mez|\na_{x,y}\psi|^2 
 + \mez\frac{\big( G(\eta)\psi+\na_{x,y}\psi\cdot\na_{x,y}\eta\big)^2}{1+|\na_{x,y}\eta|^2}-g\eta,
 \end{cases}
 \eq
  
 By a {\it Stokes wave} we mean a periodic steady solution of \eqref{ww:c}, i.e., 
 \bq\label{sys:Stokes}
 \begin{cases}
F_1(\eta, \psi, c):=c\cdot \na_{x, y}\eta+ G(\eta)\psi=0,\\
 F_2(\eta, \psi, c):=c\cdot\na_{x, y}\psi-\mez|\na_{x, y}\psi|^2+\mez\frac{\big( G(\eta)\psi+\na_{x, y}\psi\cdot\na_{x, y}\eta\big)^2}{1+|\na_{x, y}\eta|^2}-g\eta=0,
 \end{cases}
 \eq
 where $\eta,~\psi:\Rr^2\to \Rr$.  
 
 Consider now a Stokes wave traveling in the $x$-direction.  
 It is given by a pair of periodic functions $(\eta(x), \psi(x))$ independent of $y$ and a speed $c=(c_1, 0)$ solving \eqref{sys:Stokes}. Without loss of generality, we  henceforth consider $2\pi$-periodic Stokes waves.  As has been known for over a century, there exists a curve of small Stokes waves 
 parametrized analytically by the small amplitude $\eps$.  
 \begin{theo}[\protect{Theorem 1.3, \cite{Berti-DN}}]\label{theo:Stokes}
 For any $s>5/2$, there exists $\eps_{St}(s)>0$ and a unique family of solutions 
 \[
 \big (\eta(x; \eps), \psi(x; \eps), c(\eps)\big)\in H^s(\T)\times H^s(\T)\times (\Rr\times \{0\})
 \]
to the problem \eqref{sys:Stokes}, parametrized by $|\eps|<\eps_{St}(s)$, such that 
 \begin{itemize}
 \item(i) the mapping $(-\eps_{St}(s), \eps_{St}(s))\ni \eps\to  \big (\eta(\cdot; \eps), \psi(\cdot; \eps), c_1(\eps)\big)\in H^s(\T)\times H^s(\T)\times \Rr$ is analytic;
 \item(ii) $\eta(\cdot, \eps)$ is even and has average zero, and $\psi(\cdot, \eps)$ is odd. 
 \end{itemize}
 \end{theo}
The following fourth-order expansion, basically already known by Stokes, will be needed in our proof.  
We derive it in Appendix \ref{appendix:Stokes}.  
 \begin{prop}[\protect{Theorem 1.3, \cite{Berti-DN}}]
 The unique family of Stokes waves, given by Theorem \ref{theo:Stokes}, has the expansions 
\bq\label{expand:star}
\begin{aligned}
\eta(x;\eps)&= \eps \cos x+\mez \eps^2\cos(2x)+\eps^3\Big\{\frac{1}{8}\cos x+\frac{3}{8}\cos (3x)\Big\}\\
&\quad+\eps^4\left\{\frac56\cos(2x)+\frac13\cos(4x)\right\}+O(\eps^5),\\
\psi(x,\eps)&= \eps \sqrt{g}\sin x+\frac{\sqrt{g}}{2}\eps^2\sin (2 x)+\frac{\sqrt{g}}{4} \eps^3\big\{3\sin x\cos(2x)+\sin x\big\}\\
&\quad+\eps^4\sqrt{g}\left\{\frac{5}{12}\sin(2x)+\frac{1}{3}\sin(4x)\right\}+O(\eps^5),\\
c_1(\eps)&=\sqrt{g}+\frac{\sqrt{g}}{2} \eps^2+O(\eps^4).
\end{aligned}
\eq
\end{prop}
\subsection{Transverse perturbations, linearization, and transformations}
Let us fix a Stokes wave $(\eta^*, \psi^*, c^*)$ with amplitude $\eps$ and  perturb it by a two-dimensional perturbation:
\bq\label{perturbation:etapsi}
\eta(x, y)=\eta^*(x)+\nu \ol{\eta}(x, y),\quad \psi(x, y)=\psi^*(x)+\nu \ol{\psi}(x, y),\quad |\nu|\ll 1.
\eq
Before linearizing the dynamics of the full water wave equations using \eqref{perturbation:etapsi}, we  recall the  shape-derivative formula for the derivative of $G(\eta)\psi$ with respect to $\eta$. 
\begin{theo}[\protect{Theorem 3.20, \cite{Lannes}, and Proposition 2.11, \cite{ABZ1}}]\label{theo:shapederi}
Let $\eta: \T^2\to \Rr$ and $\psi: \T^2\to \Rr$. The shape-derivative of $G(\eta)\psi$ with respect to $\eta$ is denoted by 
\bq
G'(\eta)\overline{\eta}\psi=\lim_{h\to 0}\frac{1}{h}\left(G(\eta+h\overline{\eta})\psi-G(\eta)\psi\right).
\eq
We have
\bq\label{shapederi:form}
 G'(\eta)\ol{\eta}\psi=-G(\eta)(\ol\eta B(\eta)\psi)-\cnx(\ol\eta V(\eta)\psi),
\eq
where
\bq\label{def:BV:0}
B(\eta)\psi:=\frac{G(\eta)\psi+\na \eta\cdot \na \psi }{1+|\na \eta|^2},\quad V(\eta)\psi:=\na \psi-B\na \psi.
\eq
\end{theo}
Using the formulas \eqref{shapederi:form}, \eqref{def:BV:0}, and the fact that the Stokes wave  is independent of the transverse variable $y$, we can {\it linearize} \eqref{sys:Stokes} around it to obtain 
\begin{align}
&\p_t \ol\eta=\frac{\delta F_1(\eta^*, \psi^*, c^*)}{\delta(\eta, \psi)}(\ol\eta, \ol \psi)=\p_x\big((c_1^*-V^*)\ol\eta\big)+G(\eta^*) (\ol\psi-{B^*} \ol\eta),\label{lin:eta}\\
&\p_t\ol \psi=\frac{\delta F_2(\eta^*, \psi^*, c^*, P^*)}{\delta(\eta, \psi)}(\ol\eta, \ol \psi)=(c_1^*-{V^*})\p_x\ol\psi +{B^*}G(\eta^*)(\ol\psi-{B^*} \ol\eta)-{B^*}\p_x{V^*}\ol\eta-g\ol\eta,\label{lin:psi}
\end{align}
where
\bq\label{def:B*V*}
B^*:=B(\eta^*)\psi^*=\frac{G(\eta^*)\psi^*+\p_x\psi^*\p_x\eta^*}{1+|\p_x\eta^*|^2},\qquad V^*:=V(\eta^*)\psi^*=\p_x\psi^*-B^*\p_x\eta^*.
\eq 
A similar derivation was done in Lemma 3.2, \cite{NguyenStrauss} for perturbations $\ol{\eta}$ and $\ol{\psi}$ depending only on $x$. 
%
It is worthwhile noting that in fact  $B^*$ and $V^*$ are the vertical and horizontal 
components, respectively,  of the fluid velocity of the Stoke wave at the free surface $\{z=\eta^*(x)\}$. 
Where $G(\eta^*)$ acts on functions of $(x, y)$ we consider $\eta^*(x, y)\equiv \eta^*(x)$. 
We change variables to the so-called {\it good unknowns} (\`a la Alinhac)
\bq\label{def:v12}
v_1(x, y)=\ol\eta\quad\text{and}~ v_2(x, y)=\ol\psi-{B^*}\ol\eta
\eq
satisfying
 \begin{align}\label{eq:v1}
&\p_t v_1=\p_x\big((c_1^*-{V^*})v_1\big)+G(\eta^*)v_2,\\ \label{eq:v2}
&\p_t v_2=-\big(g+({V^*}-c_1^*)\p_x{B^*} \big)v_1+(c_1^*-{V^*})\p_xv_2.
\end{align}
Choosing a simple form for the transverse perturbations, we consider perturbations $\ol{\eta}$ and $\ol{\psi}$ that have wave number $\alpha \in \Rr$ in the transverse variable $y$: 
\bq
\ol{\eta}(x, y)=\wt\eta(x)e^{i\alpha y},\quad \ol{\psi}(x, y)=\wt\psi(x)e^{i\alpha y}.  
\eq
Consequently, the good unknowns have the form
\bq
v_1(x, y)=\wt\eta(x)e^{i\alpha y}:=\wt{v_1}(x)e^{i\alpha y},\quad v_2(x, y)=(\wt\psi-{B^*}\wt\eta)e^{i\alpha y}:=\wt{v_2}(x)e^{i\alpha y}.
\eq
Then, in the linearized system \eqref{eq:v1}-\eqref{eq:v2}, the most difficult term to analyze is $G(\eta^*)(\wt{v_2}(x)e^{i\alpha y})$. To handle this term, we shall flatten the surface 
$\{z=\eta^*(x)\}$  using an extension of the Riemann mapping  in \cite{NguyenStrauss} to three dimensions. 
For the sake of clarity, from now on we denote the independent variables in the physical space 
$D$ by $(X,y,Z)$.   The {\it Riemann mapping} is as follows. 
 \begin{prop}[\protect{Proposition 3.3, \cite{NguyenStrauss}}]  \label{prop:Riemann}
 There exists a  holomorphic bijection $X(x, z)+iZ(x, z)$ from the half-plane $\Rr^2_-=\{(x, z)\in \Rr^2: z<0\}$  onto the fluid region $\{(X,Z)\in \Rr^2: Z<\eta^*(X)\}$ with the following properties.  
 \begin{itemize}
 \item[(i)] $X(x+2\pi, z)=2\pi +X(x, z)$ and $Z(x+2\pi, z)=Z(x, z)$ for all $(x, z)\in \Rr^2_-:=\Rr\times (-\infty, 0)$; 
 \newline $X$ is odd in $x$ and $Z$ is even in $x$;
 \item[(ii)] $X+iZ$ maps  $\{(x, 0): x\in \Rr\}$ onto $\{(x, \eta^*(x)): x\in \Rr\}$;\\
 \item[(iii)] Defining the ``Riemann stretch" as 
 \bq\label{def:zeta}
 \zeta(x)=X(x, 0),
 \eq 
we have the Fourier expansion 
 \bq\label{z1}
X(x, z)=x-\frac{i}{2\pi}\sum_{k\in \Zz\setminus\{0\}}e^{ikx}\mathrm{sign}(k)e^{|k|z}\widehat{\eta^*\circ \zeta}(k)\quad\forall (x, z)\in \Rr^2_-.
\eq
Here, 
\[
\widehat{f}(k)=\int_\T f(x)e^{-ikx}dx;
\]
\item[(iv)] $\| \na_{x, z}(X-x)\|_{ L^\infty(\Rr^2_-)}+ \|\na_{x, z}(Z-z)\|_{ L^\infty(\Rr^2_-)}\le C\eps$.
 \end{itemize}
 \end{prop}
 \begin{rema}\label{rema:XZzeta}
$X$, $Z$, and $\zeta$ are  analytic in $\eps$  with values in  Sobolev spaces. Indeed, evaluating \eqref{z1} at $z=0$ yields 
\[
\zeta(x)=x-\frac{i}{2\pi}\sum_{k\ne 0}e^{ikx}\mathrm{sign}(k)\widehat{\eta^*\circ \zeta}(k),
\]
where the right-hand side is  analytic in $\eps$ with values in Sobolev spaces  since $\eta^*$ is so. The analyticity of $\zeta$ in $\eps$ then follows from the Analytic Implicit Function Theorem. Next, we return to the formula \eqref{z1} in which  $\eta^*\circ \zeta$ is now analytic in $\eps$ with values in any Sobolev space provided $\eps$ is small enough.  Consequently $\widehat{\eta^*\circ \zeta}(k)$ is analytic in $\eps$ for all $k$, and the series in \eqref{z1} converges absolutely and is analytic with values in Sobolev spaces. 
\end{rema} 
  
After the Riemann transformation, the Dirichlet-Neumann operator  takes the form given 
in \eqref{flatten:G} below. The proof of Theorem \ref{theo:flattenG} will make use of the following trace result:
\begin{prop}[\protect{Section  3.2, Chapter IV, \cite{BoyerFabrie}}]\label{prop:trace}
Let $U\subset \Rr^n$ be Lipschitz domain with compact boundary and denote 
\bq
H_{\cnx}(U)=\{u\in L^2(U)^n: \cnx u\in L^2(U)\}.
\eq
 If $u\in H_{\cnx}(U)$ and $w\in H^1(U)$, then 
\bq\label{Stokes}
\int_U u\cdot \na w+\int_U w\cnx u=\langle \gamma_\nu(u), \gamma_0(w)\rangle_{H^{-\mez}(\p U), H^\mez(\p U)}
\eq
where $\gamma_0(w)$ is the trace of $w$ and $\gamma_\nu (u)$ is the trace of $u\cdot \nu$, $\nu$ being the unit outward normal to $\p U$. The trace operator 
\bq\label{normaltrace}
\gamma_\nu:\quad H_{\cnx}(U)\to H^{-\mez}(\p U)
\eq
is continuous.
\end{prop}
\begin{theo}\label{theo:flattenG} 
For fixed  $\alpha\ne 0$, let $\beta=\alpha^2$. There exists $\iota>0$ such that for each $\eps \in (-\iota, \iota)$ there exists  a bounded linear operator $\cG_{\eps, \beta}\in \cL(H^\mez(\T), H^{-\mez}(\T))$   with the following properties.  Let $s\ge\frac12$.  
\begin{itemize}
\item[(i)] 
There exists $0<\eps_0(s)\le \iota$ such that for all $\eps \in (-\eps_0(s), \eps_0(s))$ 
the operator  $\cG_{\eps, \beta}$ is bounded from   $\cL(H^s(\T)$ to $H^{s-1}(\T))$. In addition, we have
\bq\label{cG:even}
(\cG(\bar f(-\cdot)))(x) = \overline{(\cG f)(-x)}
\eq
and 
\bq\label{form:G0}
\cG_{0, \beta}=(|D|^2+\beta)^\mez.
\eq
If $s=1$, the operator $\cG_{\eps, \beta}$ is self-adjoint on $L^2(\T)$.  
\item[(ii)] 
The mapping
\bq\label{mappimg:cG}
 (-\eps_0(s), \eps_0(s))\times (0, \infty)\ni  (\eps,\beta) \mapsto \cG_{\eps, \beta}\in \mathcal{L}(H^{s}(\T), H^{s-1}(\T))
\eq
 is analytic.

  \item[(iii)]  If $s\ge 1$ and $f_0\in H^s(\T)$, we have the identity 
\bq\label{flatten:G}
\cG_{\eps, \beta}(f_0\circ \zeta)(x)=e^{-i\alpha y}\zeta'(x)\big[ G(\eta^*)(f_0e^{i\alpha y})\big]\vert_{(\zeta(x), y)}
\eq
for a.e. $x\in \T$ and for all $y\in \Rr$. 
\end{itemize}
\end{theo}
\begin{proof}
We consider $\alpha>0$ throughout the proof and denote $\T^2_\alpha=\T\times (\Rr/ 2\pi \alpha \Zz)$. Viewing $\eta_*$ as a function on $\T^2_\alpha$ that is independent of $y$, it is known that
\bq\label{cont:DN}
G(\eta^*)\in \cL(H^s(\T^2_\alpha), H^{s-1}(\T^2_\alpha)),\quad \forall\ s \ge \tfrac12. 
\eq
See Theorems 3.8 and 3.12 in \cite{ABZ3}. 
 
{\bf 1.}  In order to calculate  $ G(\eta^*)(f_0e^{i\alpha y})$ for $f_0\in H^\mez(\T)$,  we consider the problem 
\[
\begin{cases}
\Delta_{X, y, Z}\vartheta=0\quad\text{in } \{Z<\eta^*(X)\}\subset \Rr^3,\\
\vartheta(X, y, \eta^*(X))=f_0(X)e^{i\alpha y},\\
\na_{X, Z}\vartheta\to 0\quad\text{as } Z\to -\infty.
\end{cases}
\]
Of course, 
the solution $\vartheta$ has the form $\vartheta(X, y, Z)=\tt(X, Z)e^{i\alpha y}$, where $\tt$ satisfies
\bq\label{system:tt}
\begin{cases}
\Delta_{X,  Z}\tt -\alpha^2\tt=0\quad\text{in } \{Z<\eta^*(X)\} \subset \Rr^2,\\
\tt(X, \eta^*(X))=f_0(X),\\
\na_{X, Z}\tt\to 0\quad\text{as } Z\to -\infty.
\end{cases}
\eq
We use the Riemann mapping $(X, Z)$ in Proposition \ref{prop:Riemann} to flatten the 
domain $\{ Z<\eta_*(X)\}$.  It has the Jacobian 
\[
\cJ=|\p_xX|^2+|\p_z X|^2.
\]  
For an arbitrary function $f\in H^\mez(\T)$,  we will define 
function $\Tt$ by solving 
\bq\label{system:Theta}    \begin{cases}
\Delta_{x, z} \Tt - \beta\cJ\Tt=0\quad\text{in } \{(x, z): z<0\},\\
\Tt(x, 0)=f(x),\\
\na_{x, z}\Tt\in L^2(\T\times \Rr_-).   
\end{cases}\eq
 In order to specify $f$, we will denote $\Theta = \Theta_f$. Then $\theta(X(x,z),Z(x,z))=\Tt _{f_0\circ\zeta} (x,z)$ is defined on the 
lower half-space $\{(x,z):\ z<0\}$.  
We note that by virtue of Theorem \ref{theo:Stokes}, $\eta^*$ 
can be taken as smooth as we wish, and hence so can $\zeta$, provided $\eps$ is small enough. 

To be precise, we now consider  $\eps \in (-\iota, \iota)$ where $\iota>0$ is sufficiently small so that $\cJ$ is bounded with bounded inverse  independent of $\eps$. By classical elliptic regularity, \eqref{system:Theta} has a unique solution $\Tt\in H^1(\T\times \Rr_-)$ for any  $f\in H^\mez(\T)$. Consequently   $\Delta_{x, z}\Tt=\beta \cJ\Tt\in  L^2(\T\times \Rr_-)$. Since both the vector field $\na_{x, z}\Tt$ and its divergence are in $L^2(\T\times \Rr_-)$,  Proposition \ref{prop:trace} implies  the trace of the normal component $\na_{x, z}\Tt(\cdot, 0)\cdot (0, 1)\equiv \p_z\Tt(\cdot, 0)$ 
 makes sense in  $H^{-\mez}(\T)$ and 
\bq\label{dzTt:low}
\begin{aligned}
\| \p_z\Tt(\cdot, 0)\|_{H^{-\mez}(\T)}&\le C\Big(\| \Tt\|_{L^2(\T\times \Rr_-)}+\| \na_{x, z}\Tt\|_{L^2(\T\times \Rr_-)}\big) \\
&\le C(\beta)\| f\|_{H^\mez(\T)},\quad \eps\in (-\iota, \iota).
\end{aligned}
\eq 
Now, if $f\in H^2(\T)$ and  $\iota$ is chosen smaller if necessary, 
 we know that $\Tt\in H^{\frac{5}{2}}(\T\times \Rr_-)$.  
Hence the trace theorem applied to $\p_z\Tt\in H^\tdm(\T\times \Rr_-)$ 
yields $\p_z\Tt(\cdot, 0)\in H^1(\T)$ with 
\bq\label{dzTt:H1}
\| \p_z\Tt(\cdot, 0)\|_{H^1(\T)}\le C(\beta)\| f\|_{H^2(\T)},\quad \eps\in (-\iota, \iota).
\eq
In view of \eqref{dzTt:low} and \eqref{dzTt:H1} together with linear interpolation,  the linear operator $\cG_{\eps, \beta}$ defined by 
\bq\label{def:cG}
\cG_{\eps, \beta} f=\p_z\Tt_f(\cdot, 0) 
\eq
is  bounded  from $H^s(\T)$ to  $H^{s-1}(\T)$ for all $s \in [\mez, 2]$ and $\eps \in (-\iota, \iota)$. 
Using Theorem \ref{theo:Stokes}, elliptic regularity, 
and the trace theorem for the range $s\in (2, \infty)$, we also deduce that 
\bq\label{bounded:cT}
\cG_{\eps, \beta}\in \cL(H^s(\T), H^{s-1}(\T))\quad\forall s\ge \tfrac12,~\forall \eps \in (-\eps_0(s), \eps_0(s))
\eq 
Moreover, $\cG_{\eps, \beta}$ is invertible, as proven for example in Proposition 2.4 in \cite{Nguyen2023}.  For the preceding interpolation argument we have chosen $\eps_0(s)=\eps(2)$ for all $s\in [\mez, 2]$.  In what follows, $\eps_0(s)$ may  shrink from one line to another.
Additionally, if  $\eps=0$, the system \eqref{system:Theta} can be solved explicitly by Fourier series and we obtain \eqref{form:G0}. 

 Next,  we prove \eqref{cG:even}. By Proposition \ref{prop:Riemann} (i), $X(x, z)$ is odd in $x$, hence $J=|\p_xX|^2+|\p_zX|^2$ is even in $x$. We recall that $\eta_*$ is also even. Therefore, $\Tt_{\bar f(-\cdot)}(x,0) -=\overline {\Tt_f(-x,0)}$ by uniqueness of solutions to the elliptic problem \eqref{system:Theta}. It follows that $(\cG(\bar f(-\cdot)))(x) = \overline{(\cG f)(-x)}$ as claimed.  

For the self-adjointness, we fix $\eps$ and $\beta$ and regard 
$\cG=\cG_{\eps, \beta}$ as an unbounded operator on $L^2(\T)$ with domain $H^1(\T)$.  For $f, m\in H^1(\T)$, we apply the Stokes formula \eqref{Stokes} with $u=\na \Tt_f$ and $w= \Tt_m$ to have 
\[
\begin{aligned}
\langle \cG f, m\rangle_{H^{-\mez}, H^\mez}&=\int_{\T\times \Rr_-}\na \Tt_f\cdot \na \Tt_mdxdz+\int_{\T\times \Rr_-}\Tt_m\Delta \Tt_f\\
&=\int_{\T\times \Rr_-}\na \Tt_f\cdot \na \Tt_mdxdz+\int_{\T\times \Rr_-}\beta J  \Tt_m\Tt_f dxdz.
\end{aligned}
\]
Since the right-hand side is symmetric in $\Tt_m$ and $\Tt_f$, we deduce that  
\[
\langle \cG f, m\rangle_{H^{-\mez}, H^\mez}=\langle \cG m, f\rangle_{H^{-\mez}, H^\mez}.
\]
By \eqref{bounded:cT} we have $\cG f$, $\cG m\in L^2(\T)$, whence $(\cG f, m)=(\cG m, f)$. Thus $\cG$ is a symmetric operator.  In order to prove the self-adjointness, let $\ell,n\in L^2(\T)$ satisfy 
$(\cG f,\ell)=(f,n)$ for all $f\in H^1(\T)$.  It is required to prove that $\ell\in H^1(\T)$ and $n=\cG\ell$.   
Indeed,  the invertibility of $\cG$ implies $m:=\cG^{-1}(n)\in H^1(\T)$.  
Then $(\cG f,\ell)=(f,\cG m)=(\cG f,m)$ for all $f\in H^1(\T)$.  
Since $\cG f$ is an arbitrary function in $L^2(\T)$, we deduce that $\ell=m\in H^1(\T)$, 
so that $n=\cG \ell$ as required. 

Thus we have completed the proof of  (i).

{\bf 2.}  Now we shall prove (iii).  
Recalling Proposition  \ref{prop:Riemann}, we have 
$X(x,0=\zeta(x),\ Z(x,0)=\eta^*(\zeta(x))$ and 
\[ 
\frac{\p Z}{\p z}(x,0)=\frac{\p X}{\p x}(x,0)= \zeta'(x), \quad 
\frac{\p X}{\p z}(x,0)= -\frac {\p Z}{\p x}(x,0) = -\p_x\eta^*(\zeta(x))\zeta'(x).    
\] 
Hence for $s>1$, $\eps \in (-\eps_0(s), \eps_0(s))$, and $f_0\in H^s(\T)$,  the trace theorem implies $\p_z\Tt(\cdot, 0)\in H^{s-1}(\T)$ and the chain rule yields 
\bq\label{cT:G}
\begin{aligned}
\cG_{\eps, \beta}f(x)=\p_z\Theta(x, 0)&=\zeta'(x)\Big[\tt_z\big(\zeta(x), y, \eta^*(\zeta(x)\big)-\tt_x\big(\zeta(x), y, \eta^*(\zeta(x)\big)\p_x\eta^*(\zeta(x)) \Big]\\
&=e^{-i\alpha y}\zeta'(\cdot)\big[G(\eta^*)(f_0e^{i\alpha y})\big]\vert_{(\zeta(x), y)}
\end{aligned}
\eq
for a.e. $x\in \Rr$ and for all $y \in \Rr$. This proves \eqref{flatten:G} for $s>1$.

In order to also prove \eqref{flatten:G} for $s=1$, we fix an arbitrary $\eps \in (-\eps_0(2), \eps_0(2))$ 
and let $f_0 \in H^1(\T)$. 
Let $f_{0, n}$ be a sequence in $H^2(\T)$ converging to $f_0$ in $H^1(\T)$.  
Then $f_{0, n}(x)e^{i\alpha y}\to f_0(x)e^{i\alpha y}$ in $H^1(\T^2_\alpha)$, 
where $\T^2_\alpha=\T\times (\Rr/ 2\pi \alpha \Zz)$. 
The continuity of the linear operator in \eqref{cont:DN} implies that $G(\eta^*)(f_{0, n}e^{i\alpha y})$ converges to $G(\eta^*)(f_0e^{i\alpha y})$ in $L^2(\T^2_\alpha)$. 
Consequently we obtain the convergence of the composition 
\[
e^{-i\alpha y}\zeta'(x)\big[G(\eta^*)(f_{0, n}e^{i\alpha y})\big]\vert_{(\zeta(x), y)} \to e^{-i\alpha y}\zeta'(x)\big[G(\eta^*)(f_{0}e^{i\alpha y})\big]\vert_{(\zeta(x), y)}\quad\text{ in}~ L^2(\T^2_\alpha).
\]
On the other hand, \eqref{cT:G} and \eqref{bounded:cT} imply that 
\[
\forall y\in \Rr,\quad e^{-i\alpha y}\zeta'(\cdot)\big[G(\eta^*)(f_{0, n}e^{i\alpha y})\big]\vert_{(\zeta(\cdot), y)}=\cG_{\eps, \beta} f_n \to \cG_{\eps, \beta}f\quad\text{in }~L^2(\T),
\]
where $f_n:=f_{0, n}\circ \zeta\to f$ in $H^1(\T)$.   Therefore 
\[
e^{-i\alpha y}\zeta'(x)\big[ G(\eta^*)(f_0e^{i\alpha y})\big]\vert_{(\zeta(x), y)}=\cG_{\eps, \beta}(f_0\circ \zeta)(x)
\]
for a.e. $x\in \T$ and for all $y\in \Rr$. This concludes the proof of (iii).
\\ \\
{\bf 3.}  In order to prove the analyticity (ii), we fix  $s\ge \mez$ unless stated otherwise. We shall write $J(\eps)\equiv J(x, z; \eps)$ and $\Tt(\eps, \beta)\equiv\Tt(x, z; \eps, \beta)$. Choosing $\eps_0(s)$ sufficiently small, we can ensure that $\cJ(\eps)\in H^{s+100}(\T)$. 
 Sobolev estimates for the elliptic operator $\Delta_{x, z}-\beta \cJ$ yield the uniform-in-$\eps$ bound 
 \bq\label{ub:Tt}
 \| \Tt(\eps, \beta)\|_{H^{s+\mez}(\T\times \Rr_-)}\le C_1(\beta, s)\| f\|_{H^s(\T)},\quad s\ge \tfrac12,
 \eq
 where $C(\beta, s)$ is bounded for $\beta$ in  compact sets of $(0, \infty)$. 
 
{\bf 3a.} We  begin by proving that  that for $f\in H^s(\T)$,  the solution $\Tt(\eps, \beta)$ is 
$C^\infty$  in $(\eps, \beta)\in (-\eps_0(s), \eps_0(s))\times (0, \infty)$ with values in $H^{s+\mez}(\T\times \Rr_-)$. For the differentiability of $\Tt$ in $\eps$, we let $u_{\eps, \beta}$ be the solution of the inhomogeneous  problem 
\bq\label{sys:ueps}
\begin{cases}
\Delta_{x, z} u_{\eps, \beta}-\beta \cJ(\eps)u_{\eps, \beta}
=\beta [\frac{d}{d\eps}\cJ(\eps)] \ 
\Tt(\eps, \beta)\quad\text{in } \T\times \Rr_-,\\
u_{\eps, \beta}(\cdot, 0)=0,
\end{cases}
\eq
which is obtained by formally differentiating \eqref{system:Theta} in $\eps$.  Using  \eqref{ub:Tt} to bound the right-hand side and then applying elliptic regularity, we find
\bq\label{ub:u}
\| u_{\eps, \beta}\|_{H^{s+\frac{5}{2}}(\T\times \Rr_-)}\le C_2(\beta, s)\| f\|_{H^s(\T)},\quad s\ge \tfrac12,
\eq
where $C_2(\beta, s)$ is bounded for $\beta$ in  compact sets of $(0, \infty)$. Now we fix $(\eps, \beta)\in (-\eps_0(s), \eps_0(s))\times (0, \infty)$. For small $h$, the expression $v_h:=\frac{\Tt(\eps+h, \beta)-\Tt(\eps, \beta)}{h}-u_{\eps, \beta}$ satisfies 
\[
 \begin{cases}
 \Delta_{x, z} v_h-\beta \cJ(\eps+h)v_h&=F_h\quad\text{in } \T\times \Rr_-,\\
 v_h(\cdot, 0)=0,
 \end{cases}
 \]
 where
 \[
F_h:= \beta\left[\frac{\cJ(\eps+h)-\cJ(\eps)}{h}-\frac{d}{d\eps}\cJ(\eps)\right]\Theta(\eps+h, \beta)+\beta \left[\cJ(\eps+h)-\cJ(\eps)\right]u_{\eps, \beta}.
 \]
 Combining \eqref{ub:Tt} and \eqref{ub:u} implies $ \| F_h\|_{H^{s+\mez}(\T\times \Rr_-)}\to 0$ as $h\to 0$.   Hence $v_h\to 0$ in $H^{s+\frac{5}{2}}(\T\times \Rr_-)$ as $h \to 0$. 
 We have proven that $\frac{\p}{\p\eps}\Tt(\eps, \beta)=u_{\eps, \beta}$ in $H^{s+\frac{1}{2}}(\T\times \Rr_-)$. An analogous argument shows that $\frac{\p}{\p \beta}\Tt(\eps, \beta)$ exists in $H^{s+\mez}(\T\times \Rr_-)$ and is the solution of 
 \bq\label{sys:dbetaTt}
 \begin{cases}
\Delta_{x, z}\frac{\p}{\p \beta}\Tt(\eps, \beta)-\beta\cJ(\eps)\frac{\p}{\p \beta}\Tt(\eps, \beta)=\cJ(\eps)\Tt(\eps, \beta)\quad\text{in } \T\times \Rr_-,\\
\frac{\p}{\p \beta}\Tt(\eps, \beta)\vert_{z=0}=0.
 \end{cases}
 \eq
  On the other hand, repeating the above argument for the problems \eqref{sys:ueps} and \eqref{sys:dbetaTt}, we deduce  that $\Tt$ is twice differentiable in $\eps$ and $\beta$. An induction argument yields that the mapping 
 \[
 (-\eps_0(s), \eps_0(s))\times (0, \infty)\ni (\eps, \beta) \mapsto \Tt(\eps, \beta)\in H^{s+\frac{1}{2}}(\T\times \Rr_-) 
\]
is $C^\infty$. 

{\bf 3b.} Our next task is to prove the analyticity of $\Tt(\eps, \beta)$. Fix arbitrary $(\eps^0, \beta^0)\in (\-\eps_0(s), \eps_0(s))\times (0, \infty)$. If we choose $\eps_0(s)$  small enough, the analyticity in $\eps$ of $X$ and $Z$ (see Remark \ref{rema:XZzeta}) implies that  the mapping 
\[
(-\eps_0(s), \eps_0(s))\ni \eps\mapsto \cJ\in H^{s+10}(\T\times \Rr_-)
\]
is analytic, where $\cJ = |\p_xX|^2+|\p_zX|^2$.   
So for any compact subinterval $I=[-a, a]\subset (-\eps_0(s), \eps_0(s))$, there exists a constant $M=M(I)>0$ such that for all $\eps \in I$ we have 
\[
\cJ(x, z; \eps)=\sum_{j=0}^\infty \eps^j\chi_j(x, z). 
\]
with 
\bq\label{est:gammaj}
\| \chi_j\|_{H^{s+10}(\T\times \Rr_-)}\le M^{j+1}. 
\eq
We recall also that $\inf_{\T\times \Rr_-}\chi_0\ge c_0>0$ for any $\eps\in (-\iota, \iota)$, where $\iota\ge \eps_0(s)$. We shall prove that $\Tt(\eps, \beta)$ can be expanded into a convergent  series 
\bq\label{series:Tt}
\Tt(\eps, \beta)=\sum_{j=0}^\infty\sum_{\ell=0}^\infty \eps^j (\beta-\beta^0)^\ell \Tt_{j, \ell}\quad\text{in } H^{s+\mez}(\T\times \Rr_-),
\eq
provided 
\bq\label{small:epsdelta}
|\eps|<r\quad\text{and}\quad|\beta-\beta^0|<r,
\eq
where $r<\beta^0$ is sufficiently small.  Formally, the coefficients satisfy the following problems.   $\Tt_{0, 0}$ satisfies 
\bq
\begin{cases}
\Delta_{x, z}\Tt_{0, 0}-\beta^0\chi_0\Tt_{0, 0}=0,\\
\Tt_{0, 0}(\cdot, 0)=f(\cdot). 
\end{cases}
\eq
For $j\ge 1$ and $\ell=0$, $\Tt_{j, 0}$ satisfies 
\bq
\begin{cases}
\Delta_{x, z}\Tt_{j, 0}-\beta^0\chi_0\Tt_{j, 0}=\beta^0\sum_{k=0}^{j-1}\chi_{j-k}\Tt_{k, 0},\\
\Tt_{j, 0}(\cdot, 0)=0.
\end{cases}
\eq
Next, for $\ell \ge 1$, we find  $\Delta\Theta_{0,\ell} - \beta\chi_0\Theta_{0,\ell}=0$ 
and $\Theta_{0,\ell}(\cdot,\ell)=0$, so that $\Theta_{0,\ell}\equiv 0$.   
On the other hand, for $\ell\ge0$ and $j\ge 1$, $\Tt_{j, \ell}$ satisfies 
\bq
\begin{cases}
\Delta_{x, z}\Tt_{j, \ell}-\beta^0\chi_0\Tt_{j, \ell}=\sum_{k=0}^j\chi_{j-k}\Tt_{k, \ell-1}+\beta^0\sum_{k=0}^{j-1}\chi_{j-k}\Tt_{k, \ell},\\
\Tt_{j, \ell}(\cdot, 0)=0.
\end{cases}
\eq 
From these elliptic equations we obtain the following estimates.  
Obviously $\Tt_{0, 0}$ satisfies 
\bq\label{est:Tt00}
\| \Tt_{0, 0}\|_{H^{s+\mez}(\T\times \Rr_-)}\le M_0\| f\|_{H^s(\T)} 
\eq 
for some $M_0=M_0(\beta^0, s)$.  
By elliptic estimates and product rules for Sobolev norms, there exists $M_1=M_1(\beta^0, s)$ such that for all $j\ge 1$ we have
\bq\label{est:Tt:3}
\| \Tt_{j, 0}\|_{H^{s+\frac52}(\T\times \Rr_-)}\le M_1\beta^0 \sum_{k=0}^{j-1}\| \chi_{j-k}\|_{H^{s+10}(\T\times \Rr_-)} \| \Tt_{k, 0}\|_{H^{s+\mez}(\T\times \Rr_-)}
\eq
and for all $j\ge 1$ and $\ell\ge 1$ we have
\bq\label{est:Tt:4}
\begin{aligned}
\| \Tt_{j, \ell}\|_{H^{s+\frac52}(\T\times \Rr_-)}&\le M_1\sum_{k=0}^j\|\chi_{j-k}\|_{H^{s+10}(\T\times \Rr_-)}\|\Tt_{k, \ell-1}\|_{H^{s+\mez}(\T\times \Rr_-)}
\\
&\quad+\beta^0M_1\sum_{k=0}^{j-1}\|\chi_{j-k}\|_{H^{s+10}(\T\times \Rr_-)}\|\Tt_{k, \ell}\|_{_{H^{s+\mez}(\T\times \Rr_-)}}.
\end{aligned}
\eq
Let 
\bq\label{def:A}
A >1+M_1+ \beta^0M_1M.
\eq  
We claim that 
\bq\label{induction:Tt}
\| \Tt_{j, \ell}\|_{H^{s+\frac52}}\le \| \Tt_{0, 0}\|_{H^{s+\mez}(\T\times \Rr_-)}(A M)^{j+\ell}\quad\text{if}\quad j+\ell \ge 1.
\eq 
Temporarily taking \eqref{induction:Tt} for granted, we deduce using \eqref{est:Tt00} that the series \eqref{series:Tt} converges absolutely in $H^{s+\mez}(\T\times \Rr_-)$ provided $r<(AM)^{-1}$. We will prove the claim \eqref{induction:Tt} by induction on $\ell$ and $j$.

 We begin with $\ell=0$ and $j = 1$.    Combining 
 \eqref{est:Tt:3}, \eqref{est:Tt00},  \eqref{est:gammaj}, and \eqref{def:A}, we obtain 
\begin{align*}
\| \Tt_{1, 0}\|_{H^{s+\frac52}(\T\times \Rr_-)}&\le M_1\beta^0 \| \chi_1\|_{H^{s+10}(\T\times \Rr_-)} \| \Tt_{0, 0}\|_{H^{s+\mez}(\T\times \Rr_-)}\\
&\le \| \Tt_{0, 0}\|_{H^{s+\mez}(\T\times \Rr_-)} \beta^0M_1 M^2
\le  \| \Tt_{0, 0}\|_{H^{s+\mez}(\T\times \Rr_-)}(AM). 
\end{align*}
Thus \eqref{induction:Tt} holds for $j=1$.  
Now let $\ell=1$ and $j\ge1$ and assume  
by induction that \eqref{induction:Tt} is valid for $\ell=0$ and up to $j-1$, where $j\ge 2$. 
Then it follows from \eqref{est:Tt:3}, \eqref{est:gammaj}, and \eqref{def:A} that 
\begin{align*}
\| \Tt_{j, 0}\|_{H^{s+\frac52}(\T\times \Rr_-)}&\le M_1\beta^0 \sum_{k=0}^{j-1}M^{j-k+1} \| \Tt_{0, 0}\|_{H^{s+\mez}(\T\times \Rr_-)}(A M)^k\\
&= \| \Tt_{0, 0}\|_{H^{s+\mez}(\T\times \Rr_-)} (AM)^j \beta^0M_1M   \frac{1-A^{-j}}{A -1}\\
&\le  \| \Tt_{0, 0}\|_{H^{s+\mez}(\T\times \Rr_-)} (A M)^j.
\end{align*}
This proves \eqref{induction:Tt} for $\ell=0$ and $j\ge 1$. 

Next, we consider $\ell=1$. The case $ j=0,\ \ell=1$ is trivial since $\Tt_{0, 1}=0$.   
Let $\ell=1, \ j\ge1$ and 
assume by induction on $j$ that \eqref{induction:Tt} holds up to $j-1$.  In view of  \eqref{est:Tt:4},  \eqref{est:gammaj}, \eqref{def:A} and the case $\ell=0$, we have
\allowbreak
\begin{align*}
\| \Tt_{j, 1}\|_{H^{s+\frac52}(\T\times \Rr_-)}&\le M_1\sum_{k=0}^j\|\chi_{j-k}\|_{H^{s+10}(\T\times \Rr_-)}\|\Tt_{k, 0}\|_{H^{s+\mez}(\T\times \Rr_-)}
\\
&\quad+\beta^0M_1\sum_{k=0}^{j-1}\|\chi_{j-k}\|_{H^{s+10}(\T\times \Rr_-)}\|\Tt_{k, 1}\|_{_{H^{s+\mez}(\T\times \Rr_-)}}\\
&\le 
M_1\sum_{k=0}^j M^{j-k+1}(AM)^k\|\Tt_{0, 0}\|_{H^{s+\mez}(\T\times \Rr_-)}
\\
&\quad+\beta^0M_1\sum_{k=0}^{j-1}M^{j-k+1}(AM)^{k+1}\|\Tt_{0, 0}\|_{_{H^{s+\mez}(\T\times \Rr_-)}}   
\quad(\text{by induction})\\
&= 
M_1M^{j+1}\|\Tt_{0, 0}\|_{H^{s+\mez}(\T\times \Rr_-)}\frac{A^{j+1}-1}{A-1}+\beta^0M_1M^{j+2}\frac{A^{j+1}-A}{A-1}\\
&\le (AM)^{j+1}\|\Tt_{0, 0}\|_{H^{s+\mez}(\T\times \Rr_-)}(M_1+\beta^0M_1M)\frac{1}{A-1}\\
&\le (AM)^{j+1}\|\Tt_{0, 0}\|_{H^{s+\mez}(\T\times \Rr_-)}.
\end{align*}
This proves \eqref{induction:Tt} for $\ell=1$ and $j\ge 0$. 
The same reasoning allows one to prove by induction on $\ell$ that \eqref{induction:Tt} 
holds for all $\ell \ge 1$.
\\ \\  
{\bf 3c.} Now we make the connection to $\cG_{\eps,\beta}$, defined in \eqref{def:cG}.    
In case $s>1$,  the preceding analyticity of $\Tt(\eps,\beta)$ in $H^{s+\mez}(\T\times \Rr_-)$ 
and the standard continuity from  $H^{s-\mez}(\T\times \Rr_-)$ to $H^{s-1}(\T)$ 
of the trace onto $\{z=0\}$ 
implies that the mapping \eqref{mappimg:cG} is analytic. The trace theorem fails if $s=1$. Nevertheless, it follows from \eqref{induction:Tt}  that,  
omitting $\Tt_{0, 0}$, the series 
\[
\sum_{j\ge 0, \ell\ge 0, j+\ell \ge 1}\eps^j(\beta-\beta^0)^\ell \Tt_{j, \ell}
\]
actually converges absolutely in the stronger topology $H^{s+\frac52}(\T\times \Rr_-)$ to some limit $\Tt_\sharp$. Thus for all $s\ge \mez$ the continuity of the trace $H^{s+\frac32}(\T\times \Rr_-)\to H^{s+1}(\T)$ implies that the mapping 
\bq\label{map:Ttsharp}
(-\eps_0(s), \eps_0(s))\times (0, \infty)\ni (\eps,\beta) \mapsto \left[f\mapsto  \p_z\Tt_\sharp\vert_{z=0}\right]\in \cL(H^s(\T), H^{s+1}(\T))
\eq
is analytic. On the other hand, as proven in {\bf 1.} the mapping $f\mapsto  \p_z\Tt_{0, 0}\vert_{z=0}$ 
 belongs to $\cL(H^s(\T), H^{s-1}(\T))$ for all $s\ge \mez$.   
 Since $\cG_{\eps, \beta}f=\p_z\Tt_{0, 0}\vert_{z=0}+\p_z\Tt_\sharp\vert_{z=0}$ by definition \eqref{def:cG}, 
 it follows from \eqref{map:Ttsharp} that $\cG_{\eps, \beta}$ is analytic in $(\eps, \beta)$ with values in $\cL(H^s(\T), H^{s-1}(\T))$. This completes the proof of (ii).  
\end{proof} 

\subsection{The reformulated instability problem}
Theorem \ref{theo:flattenG} allows us to precisely formulate the instability problem as follows.  
Starting from the Riemann stretch \eqref{def:zeta}, 
it is convenient to first define the two auxiliary operators 
\bq
\zeta_\sharp f(x)=(f\circ \zeta)(x),\quad \zeta_*f(x)=\zeta'(x)\zeta_\sharp f(x).
\eq
Then we define the new unknowns 
\bq\label{def:w12}
w_1(x)=\zeta_*\wt{v_1},\quad w_2(x)=\zeta_\sharp \wt{v_2}.
\eq
We apply $\zeta_*$ to \eqref{eq:v1} and $\zeta_\sharp$ to \eqref{eq:v2}, and use Theorem \ref{theo:flattenG} (iii) to have 
\[
\zeta'(x)\big[ G(\eta^*)(\wt{v_2}e^{i\alpha y})\big]\vert_{(\zeta(x), y)}=e^{i\alpha y}\cG_{\eps, \beta}(\wt{v_2}\circ \zeta)(x),
\]
that is, $\zeta_*G(\eta^*)(v_2)(x, y)=e^{i\alpha y}\cG_{\eps, \beta}(w_2)(x)$. We  thus obtain the equivalent linearized system
\begin{align}\label{eq:w1}
&\p_t w_1=\p_x\big(p(x)w_1\big) 
+ \cG_{\eps, \beta} w_2,\\  \label{eq:w2}
&\p_t w_2=-\frac{g +q(x)}{\zeta'(x)}w_1+p(x)\p_xw_2,
\end{align}
 where  the variable coefficients 
\bq\label{def:pq}
p:=\frac{c_*-\zeta_\sharp V_*}{\zeta'},\quad q:=-p\p_x(\zeta_\sharp B_*)
\eq
depend only on $\eps$.  
By the change of variables 
\bq\label{rescale:g}
 (\psi, c)\to (\sqrt{g} \psi, \sqrt{g}c)
 \eq
 in the basic water wave system \eqref{sys:Stokes}, 
 we hereafter assume without loss of generality that $g=1$.
 
For the spectral analysis, we seek solutions of the linearized system \eqref{eq:w1}--\eqref{eq:w2} 
of the form $w_j(x, t)=e^{\ld t}u_j(x)$, thereby arriving at the  {\it eigenvalue problem}  
\bq\label{eiproblem}
\ld U=\cL_{\eps, \beta} U:= 
\begin{bmatrix} 
\p_x(p(x)\cdot)  & \cG_{\eps, \beta}\\
-\frac{1 +q(x)}{\zeta'(x)}& p(x)\p_x
\end{bmatrix}U,\quad U=
\begin{bmatrix}  u_1 \\ u_2 \end{bmatrix}.
\eq
By virtue of Theorem \ref{theo:flattenG} (i), we have $\cG_{\eps, \beta}\in \cL(H^1(\T), L^2(\T))$ is self-adjoint for $\eps\in (-\eps_0(1), \eps_0(1))$ and $\beta\in (0, \infty)$. Therefore, we regard $\cL_{\eps, \beta}$, defined above,  as a bounded linear operator from  $(H^1(\T))^2$ to $(L^2(\T))^2$.   Moreover, $\cL_{\eps, \beta}$ can be written in Hamiltonian form
 \bq\label{def:Hamiltonian}
\cL_{\eps, \beta}=J\cH_{\eps, \beta},\quad J=\begin{bmatrix} 0& 1\\ -1 & 0\end{bmatrix},\quad \cH_{\eps, \beta}=\begin{bmatrix} \frac{1 +q(x)}{\zeta'(x)} & -p(x)\p_x \\ \p_x(p(x)\cdot) & \cG_{\eps, \beta}\end{bmatrix},
 \eq
 where $\cH_{\eps, \beta}$ is self-adjoint. All entries of $\cH_{\eps, \beta}$ depend on $\eps$ but only the lower right corner depends on $\beta$. For the two-dimensional modulational instability,  $\beta=0$ and it was shown in \cite{NguyenStrauss} that $\cG_{\eps, 0}=|D|$ (see also Proposition \ref{prop:Rj} below).  However, when $\beta \ne 0$, $\cG_{\eps, \beta}$ is no longer a Fourier multiplier but a genuine pseudo-differential operator, which will be expanded up to $O(\eps^3)$ in Section \ref{Sec:R}. 
  \begin{defi}
 The Stokes wave $(\eta^*, \psi^*, c^*)$ is said to be unstable with respect to  transverse perturbations if there exists  a transverse wave number $\alpha$ such that $\cL_{\eps, \beta}$  has an eigenvalue with positive real part. 
 \end{defi}
 
 The expansions in powers of $\eps$ for the Riemann stretch $\zeta$ and the coefficients in \eqref{eq:w1}--\eqref{eq:w2} are given in the next proposition.  
\begin{prop}\label{prop:expandpq}
For any $s>\frac52$, the following functions are analytic in $\eps$ with values in $H^s(\T)$ provided $\eps$ is small enough. They have the expansions:
\begin{align}\label{zeta1}
&\zeta(x)=x+\eps \sin x+\eps^2 \sin(2x)+O(\eps^3),\\
 \label{expand:p}
&p(x)=1 -2\eps\cos x+\eps^2\big(\tfrac32-2\cos(2x)\big)+\eps^3\big(3\cos x-3\cos(3x)\big)+O(\eps^4),\\ \label{expand:q}
&q(x)=-\eps\cos x+\eps^2\big(1-\cos(2x)\big)+\eps^3\big(2\cos x-\tfrac32 \cos(3x)\big)+O(\eps^4),\\ \label{expand:qzeta}
&\frac{1+q(x)}{\zeta'}=1-2\eps \cos x+2\eps^2\big(1-\cos(2x)\big)+\eps^3\big(4\cos x-3\cos(3x)\big)+O(\eps^4).
\end{align}
\end{prop}
The proof of Proposition  
\ref{prop:expandpq} is given in Appendix \ref{appendix:pq}.  
Up to order $O(\eps^3)$ the same expansions were already obtained in \cite{NguyenStrauss}.
Combining Theorem \ref{theo:flattenG} and Proposition \ref{prop:expandpq}, 
we conclude that the mapping
\bq\label{analyticity:cL}
(-\eps_0(1), \eps_0(1))\times (0, \infty)\ni (\eps, \beta)\mapsto \cL_{\eps, \beta} \in \cL\big((H^1(\T))^2, (L^2(T))^2\big)
\eq
is analytic.

\section{Resonance condition and Kato's perturbed basis} \label{Sec:Kato} 
\subsection{Resonance condition}
We begin with the basic case $\eps=0$.  Using the expansions \eqref{expand:p} and \eqref{expand:q} for $p$ and $q$ and the formula \eqref{form:G0} for $\cG(0, \beta)$,  we find in view of \eqref{eiproblem} that
\bq
\cL_{0, \beta}=\begin{bmatrix} 
\p_x  & (|D|^2+\beta)^\mez\\
-1& \p_x
 \end{bmatrix}.
\eq
 The {\it spectrum of} $\cL_{0,\beta}$ consists of the purely imaginary eigenvalues 
\bq\label{ld0pm}
\ld^0_\pm(k, \beta)=i[k\pm (k^2+\beta)^\frac14],\quad k\in \Zz,
\eq
which are the roots of the quadratic characteristic polynomial 
\bq
\Delta_0(\ld; k, \beta)=(\ld-ik)^2+(k^2+\beta)^\mez=[\ld-\ld^0_+(k, \beta)][\ld-\ld^0_-(k, \beta)].
\eq
We learn from \cite{MMMSY} that the resonance condition is a double eigenvalue 
$\ld^0_+(-(m+1), \beta)=\ld^0_-(m, \beta)$.  That is, 
\bq
-(m+1)+((m+1)^2+\beta)^\frac14=m-(m^2+\beta)^\frac14.
\eq
For our purposes, we choose the {\it simplest resonance} $m=1$. We define $\beta_*\approx 2.7275211479$ to be the unique positive solution of 
\bq\label{resonancecondition}
\ld^0_+(-2, \beta_*)=\ld^0_-(1, \beta_*)\iff -2+(\beta_*+4)^\frac14=1-(\beta_*+1)^\frac14. 
\eq
Given $\beta_*$, we also define $\sigma$ as
\bq\label{def:sigma}
i\sigma=\ld^0_+(-2, \beta_*)=\ld^0_-(1, \beta_*).
\eq
We have 
\bq\label{dkld0}
\frac{d}{dk}\left(-i\ld^0_\pm(k, \beta )\right)\ge \mez\quad\forall |k|\ge 1,
\eq
so the functions $\mathbb{Z}\ni k \mapsto -i\ld^0_\pm(k, \beta)$ are strictly increasing  on $(-\infty, -1]$ and $[1, \infty)$. In addition, it can be readily checked that 
\begin{align}\label{ineq:ldp}
 &-i\ld^0_+(k, \beta)(-1)< -i\ld^0_+(0, \beta)< -i\ld^0_+(1, \beta),\\ \label{ineq:ldm}
 &-i\ld^0_-(-1, \beta)(-1)< -i\ld^0_-(1, \beta)< -i\ld^0_-(0, \beta).
\end{align}
Combining \eqref{dkld0}, \eqref{ineq:ldp} and \eqref{ineq:ldm}, we deduce that
\[
\begin{cases}
\ld^0_+(k, \beta_*)=\ld^0_+(-2, \beta_*)\iff k=-2,\\
\ld^0_-(k, \beta_*)=\ld^0_+(1, \beta_*)\iff k=1.
\end{cases}
\]
 Consequently
\bq\label{Delta0}
\Delta_0(i\sigma; k, \beta _*)=0\iff k\in \{1, -2\},
\eq
and thus $i\sigma$ is a {\it double eigenvalue} of $\cL_{0,\beta_*}$. Moreover, it follows from \eqref{dkld0}, \eqref{ineq:ldp}, and  \eqref{ineq:ldm} that $i\sigma$ is {\it isolated} in the spectrum of $\cL_{0, \beta_*}$. The  eigenspace corresponding to  $i\sigma$  is 
\bq
\cU:=N(\cL_{0,\beta_*}-i\sigma\I)=\text{span}\{U_1, U_2\},
\eq
where $\I$ denotes the identity operator and 
\bq      \label{def:U}
U_1=\begin{bmatrix} i(\beta_* +1)^\frac14 e^{ix}\\ e^{ix} \end{bmatrix},\quad U_2=\begin{bmatrix} -i(\beta_* +4)^\frac14 e^{-2ix}\\ e^{-2ix} \end{bmatrix}.
\eq
For ease of notation, we denote 
  \bq\label{def:gammaj}
 \g_j=\Om(j)^\mez=(\beta_* +j^2)^\frac14,
  \eq
  so that 
  \bq\label{sigma:gamma12}
  \sigma=1-\g_1=-2+\g_2.
  \eq 
  With the aid of Mathematica, we can obtain the following numerical values of the various constants above:
\begin{center}
\begin{tabular}{c || r} 
 \hline \hline
 $\beta_*$ & 2.7275211479 \\ 
 \hline
 $\sigma$ & -0.3894887313  \\ 
 \hline
 $\gamma_1$ & 1.3894887313 \\ \hline
 $\gamma_2$ & 1.6105112687 \\
 \hline
 \hline
\end{tabular}
\end{center}

\subsection{Kato's  perturbed basis} 
In order to prove the transverse instability, we will prove that the isolated double eigenvalue $i\sigma$ of $\cL_{0, \beta_*}$ leaves the imaginary axis as  $\cL_{0, \beta_*}$ is perturbed  to  $\cL_{\eps, \beta_*+\delta}$ for well-chosen $(\eps, \delta)$ near $(0, 0)$. This will in turn be achieved by tracking how the eigenspace $\{U_1, U_2\}$ is perturbed accordingly. To that end, following Berti {\it et al} \cite{Berti1}, we use Kato's similarity transformation \cite{Kato} 
 \bq \label{def:Kato}
\cK_{\eps, \delta} := \{1-(P_{\eps, \delta}-P_{0, 0})^2\}^{-\mez} \{P_{\eps, \delta} P_{0, 0} + (1-P_{\eps, \delta})(1-P_{0, 0})\}.   \eq
 Here, $P_{\eps, \delta}:  (L^2(\T))^2\to (H^1(\T))^2$ is the spectral projection 
 \bq\label{def:Peps}
P_{\eps, \delta} = -\frac1{2\pi i} \int_\Gamma (\cL_{\eps, \beta_*+\delta} - \lb)^{-1} d\lb,
  \eq 
where  $\Gamma$ is a closed circle of sufficiently small radius around the isolated double eigenvalue $i\sigma$  such that $\Gamma\subset \rho(\cL_{\eps, \delta})$ for $(\eps, \delta)$ close to $(0, 0)$.  The analyticity \eqref{analyticity:cL} of $\cL_{\eps, \beta}$ implies that $P_{\eps, \delta}$ and $\cK_{\eps, \delta}$ are analytic in $(\eps, \delta)$ near $(0, 0)$. Set 
\bq
\cV_{\eps, \delta}:=\text{Range}(P_{\eps, \delta}).
\eq
We will need the following properties:
\begin{align}\label{prop1}
&\cL_{\eps, \beta_*+\delta}: \cV_{\eps, \delta}\to \cV_{\eps, \delta},\\ \label{prop2}
&\sigma(\cL_{\eps, \beta_*+\delta})\cap \{z\in \mathbb{C}~\text{inside~}\Gamma\}=\sigma(\cL_{\eps, \beta_*+\delta}\vert_{\cV_{\eps, \delta}}),\\ \label{prop3}
&\cV_{\eps, \delta}=\cK_{\eps, \delta}\cV_{0, 0},
\end{align}
whose proof can be found in   Lemma 3.1 in \cite{Berti1}. 
\\
Since $\cV_{0, 0}=\text{span}\{U_1, U_2\}$,  \eqref{prop3} implies
\bq
\cV_{\eps, \delta}=\text{span}\{U^{\eps,\delta}_1, U^{\eps,\delta}_2\},
\eq
where 
\bq\label{Katobasis}
U^{\eps, \delta}_m=\cK_{\eps, \delta} U_m \quad (m=1,2).
\eq
We note that $U^{\eps, \delta}_m$ is analytic in $(\eps, \delta)$ near $(0, 0)$ because $\cK_{\eps, \delta}$ is so. The property \eqref{prop1}  allows us to reduce the spectral analysis of $\cL_{\eps, \beta_*+\delta}$ to that of a $2\times 2$ matrix. To write down the matrix, we  first compute the inner products 
\bq\label{JUU:0}
(JU_1, U_2)=(JU_2, U_1)=0,\quad (JU_1, U_1)=-i4\pi \g_1,\quad (JU_2, U_2)=i4\pi \g_2.
\eq
By Lemma 3.2 in \cite{Berti1}, $\cK_{\eps, \delta}$ is symplectic, \emph{i.e.} $\cK_{\eps, \delta}^*J\cK_{\eps, \delta}=J$. It follows that $(J\cK_{\eps, \delta}U, \cK_{\eps, \delta} V)=(JU, V)$ and hence \eqref{JUU:0} yields
\bq\label{JUU}
(JU_1^{\eps, \delta}, U_2^{\eps, \delta})=(JU_2^{\eps, \delta}, U_1^{\eps, \delta})=0,\quad (JU_1^{\eps, \delta}, U_1^{\eps, \delta})=-i4\pi \g_1,\quad (JU_2^{\eps, \delta}, U_2^{\eps, \delta})=i4\pi \g_2.
\eq
\begin{lemm}
Using our notation  $\beta=\beta_*+\delta$, the {$2\times 2$ matrix}  
that represents the linear operator $\cL_{\eps,  \beta}=J\cH_{\eps,  \beta}:\cV_{\eps, \delta}\to \cV_{\eps, \delta}$ (by \eqref{prop1}) with respect to the basis $\{U_1^{\eps, \delta}, U_2^{\eps, \delta}\}$ is 
\bq\label{Meps}
 \begin{bmatrix} -\frac{i}{4\pi \g_1}(\cH_{\eps, \beta} U_1^{\eps, \delta}, U_1^{\eps, \delta}) &  \frac{i}{4\pi \g_2}(\cH_{\eps, \beta} U_1^{\eps, \delta}, U_2^{\eps, \delta})   \\
- \frac{i}{4\pi \g_1}(\cH_{\eps, \beta} U_2^{\eps, \delta}, U_1^{\eps, \delta}) & \frac{i}{4\pi \g_2}(\cH_{\eps, \beta} U_2^{\eps, \delta}, U_2^{\eps, \delta}) 
\end{bmatrix}. 
\eq 
\end{lemm}
\begin{proof}
 For any $U\in \cV_{\eps, \delta}$, we write $U=a_1U^{\eps, \delta}_1+a_2U_2^{\eps, \delta}$. Then 
\[
(JU, U_k^{\eps, \delta})=a_k(JU_k^{\eps, \delta}, U_k^{\eps, \delta})=i(-1)^ka_k4\pi \g_k  \quad\text{ for } k=1,2 
\]
in view of  \eqref{JUU}. Consequently, 
\[
U=-i\sum_{k=1}^2 \frac{(-1)^k}{4\pi \g_k}(JU, U_k^{\eps, \delta})U^{\eps, \delta}_k.
\]
For any $U\in \cV_{\eps, \delta}=R(P_{\eps, \delta})$, we have $\cL_{\eps, \beta} U\in \cV_{\eps, \delta}$ since $\cL_{\eps, \beta} P_{\eps, \delta}=P_{\eps, \delta} \cL_{\eps, \beta}$. Substituting $U = \cL_{\eps, \beta} U^{\eps, \delta}_j$, we have 
\[
\cL_{\eps, \beta} U^{\eps, \delta}_j=-i\sum_{k=1}^2 \frac{(-1)^k}{4\pi \g_k}(J\cL_{\eps, \delta} U_j^\eps, U_k^{\eps, \delta})U^{\eps, \delta}_k=i\sum_{k=1}^2 \frac{(-1)^k}{4\pi \g_k}(\cH_\eps U_j^{\eps, \delta}, U_k^{\eps, \delta})U^{\eps, \delta}_k.
\]
This concludes the proof of \eqref{Meps}.
\end{proof}

We normalize the unperturbed and perturbed basis vectors  for later convenience, yielding
\bq\label{UtoV}
V_m=\frac{1}{\sqrt{\gamma_m}}U_m,\quad V^{\eps, \delta}_m=\frac{1}{\sqrt{\gamma_m}}U^{\eps, \delta}_m,
\eq
respectively. With respect to this normalized basis, the $2\times 2$ matrix  that represents the Hamiltonian operator 
$\cL_{\eps, \beta_*+\delta}=J\cH_{\eps, \beta_*+\delta}:\cV_{\eps, \delta}\to \cV_{\eps, \delta}$ becomes
\bq\label{matrixL}
\textrm{L}_{\eps,\delta}=\begin{bmatrix} -\frac{i}{4\pi }(\cH_{\eps, \beta_*+\delta} V_1^{\eps, \delta}, V_1^{\eps, \delta}) &  \frac{i}{4\pi }(\cH_{\eps, \beta_*+\delta} V_1^{\eps, \delta}, V_2^{\eps, \delta})   \\
- \frac{i}{4\pi}(\cH_{\eps, \beta_*+\delta} V_2^{\eps, \delta}, V_1^{\eps, \delta}) & \frac{i}{4\pi }(\cH_{\eps, \beta_*+\delta} V_2^{\eps, \delta}, V_2^{\eps, \delta}) 
\end{bmatrix}. 
\eq 
We also define the {\it reversal operator} as  
\bq 
R \begin{bmatrix} v_1(x) \\ v_2(x) \end{bmatrix} 
=  \begin{bmatrix} -\bar v_1(-x) \\ \bar v_2(-x) \end{bmatrix}   . \eq 
Clearly $(Rv,Rw) = \overline{(v,w)}$.  
In addition, 
\bq    \label{reverse H}
\cH_{\eps,\beta} R = R \cH_{\eps,\beta}.  
\eq 
Indeed, we recall from \eqref{cG:even} that
that $(\cG(\bar f(-\cdot)))(x) = \overline{(\cG f)(-x)}$ and from \cite{NguyenStrauss} that the functions $p, q, \zeta'$ are real and even.  
Therefore, acting on the vector $\begin{bmatrix} v_1 \\  v_2 \end{bmatrix}$, 
a simple calculation shows that both sides of \eqref{reverse H} are equal to 
\[
\begin{bmatrix}  
-\frac{1+q}{\zeta'}\bar v_1(-x) +  p(\p_x\bar v_2)(-x)   \\  
	 (\p_x (p\bar v_1))(-x) + (\cG \bar v_2)(-x)   \end{bmatrix}  . 
	\]
Since $RJ=-JR$, we also have $R\cL_{\eps, \beta} = - \cL_{\eps, \beta} R$.  According to \cite{Berti1}, the operator $\cL_{\eps, \beta}$ is termed  {\it reversible}. 
We also verify directly that $RV_j=V_j$ for $j=1,2$.  
Also $R\cK_{\eps, \beta} = \cK_{\eps, \beta} R$ by  Lemma 3.2 (i) in \cite{Berti1}, so that $RV_j^{\eps, \delta}=V_j^{\eps, \delta}$ for $j=1,2$,  in view of \eqref{Katobasis} and \eqref{UtoV}. It follows that 
\bq 
(\cH_{\eps, \beta_*+\delta} V_j^{\eps, \delta},V_k^\eps) = \overline{ (R\cH_{\eps, \beta_*+\delta}V_j^{\eps, \delta} , RV_k^{\eps, \delta})} = \overline{(\cH_{\eps, \beta_*+\delta} V_j^{\eps, \delta},V_k^{\eps, \delta})}  
\quad
 \text{ is real for } j,k=1,2.  \eq
 Thus the entries of $\textrm{L}_{\eps, \delta}$ are purely imaginary and 
\bq\label{L:reversediagonal}
(\textrm{L}_{\eps, \delta})_{12}=-(\textrm{L}_{\eps, \delta})_{21}.
\eq

\section{Expansions of $\cL_{\eps,\beta}$ up to third order}\label{Sec:R}
 We  recall the Hamiltonian operator \eqref{def:Hamiltonian}
 \[
 \cH_{\eps, \beta}=\begin{bmatrix} \frac{1 +q(x)}{\zeta'(x)} & -p(x)\p_x \\ \p_x(p(x)\cdot) & \cG_{\eps, \beta}\end{bmatrix}.
 \]
 By Proposition \ref{prop:expandpq}, the variable coefficients $p(x)$ and $\frac{1 +q(x)}{\zeta'(x)}$ are analytic in $\eps$. By Theorem \ref{theo:flattenG}, $\cG_{\eps, \beta}$ is analytic in $(\eps, \beta)\in (-\eps_0(1), \eps_0(1))\times (0, \infty)$. 
 In particular, for fixed $\beta>0$ an expansion 
 \bq\label{expandcG}
 \cG_{\eps, \beta}=\sum_{j=0}^\infty \eps^j R_j(\beta)
 \eq
 is valid for $|\eps|<\eps_0(1)$.  
 Using \eqref{expandcG}  in conjunction with the expansions in Proposition  \ref{prop:expandpq}, we find the first few terms of the expansion to be 
 \bq    \label{expandHepsbeta}
\cH_{\eps, \beta}=\sum_{j=0}^3\eps^j\cH^{j}+O(\eps^4),
\eq
			where 
\bq
\begin{aligned} \label{Hexpansion:eps}
&\cH^{0}=\begin{bmatrix} 
1 &-\p_x \\ \p_x & R_0
\end{bmatrix}, \quad R_0=\Om(D):=(|D|^2+\beta)^\mez,\\
&\cH^{1}=\begin{bmatrix} 
-2 \cos x& 2 \cos x\p_x\\ -2\p_x(\cos x\cdot)  &   R_1
\end{bmatrix},\\
& \cH^{2}=\begin{bmatrix} 
2(1-\cos(2x))& -[\tdm-2\cos(2x)]\p_x\\ \p_x\{[\tdm-2\cos(2x)]\cdot)\} & R_2
\end{bmatrix},\\
&\cH^{3}=\begin{bmatrix}
-\big(-4\cos x+3\cos(3x)\big)& -\big(3\cos x-3\cos(3x)\big)\p_x\\ \p_x\left\{\big[3\cos x-3\cos(3x)\big]\cdot\right\} & R_3 
\end{bmatrix}.
\end{aligned}
\eq
\begin{rema}
 { Acting on the basis \eqref{def:U}, we have 
  $\cH^{0} U_1=-i\sigma \begin{bmatrix}1\\ -i\g_1 \end{bmatrix} e^{ix}$ 
  and $\cH^{0} U_2=-i\sigma \begin{bmatrix}1\\ i\g_2 \end{bmatrix} e^{-2ix}$. } 
\end{rema}
 In the following proposition, we show that $R_1,R_2,$ and $R_3$ are Fourier multipliers, and we explicitly compute their coefficients.
\begin{prop}\label{prop:Rj}
The Fourier multipliers $R_j$ for $1 \leq j \leq 3$ take the form
\begin{align}\label{form:R1}
&\wh{R_1f}(k)={C_k^-}\wh{f}(k-1)+{C^+_k}\wh{f}(k+1),\\\label{form:R2}
&\wh{R_2f}(k)=B^-_k\wh{f}(k-2)+B^0_k\wh{f}(k)+B^+_k\wh{f}(k+2),\\ \label{form:R3}
&\wh{R_3f}(k)= D^{-3}_k\wh{f}(k-3)+ D^{-1}_k\wh{f}(k-1) + D^{1}_k\wh{f}(k+1)+D^{3}_k\wh{f}(k+3), 
\end{align}
where the coefficients $C_k$, $B_k$, and $D_k$ are explicit functions of $\beta$ and $k$ to be derived below. In case $\beta=0$,  $R_j\equiv 0$ for all $j\ge 1$.
\end{prop}
\begin{proof}
We begin with the expansion
\[
\eta^*(\zeta(x))=\eps \cos x+\eps^2(\cos(2x)-\mez)+\eps^3\big(\tdm \cos(3x)-\cos x \big)+O(\eps^4),
\]
from \eqref{etazeta}. Inserting it into the Fourier expansion of the Riemann stretch \eqref{z1}, we find that 
\begin{align*}
&\p_x X=1+\eps e^z\cos x+2\eps^2e^{2z}\cos(2x)+\eps^3\left\{\frac92 e^{3z}\cos(3x)-e^z\cos x \right\}+O(\eps^4),\\
 &\p_zX=\eps e^z\sin x+2\eps^2e^{2z}\sin(2x)+\eps^3\left\{\frac92 e^{3z}\sin(3x)-e^z\sin x\right\}+O(\eps^4).
\end{align*}
It follows that the Jacobian is 
\[
\cJ(x, z)=1+2\eps e^z \cos x+\eps^2e^{2z}(1+4\cos(2x))+\eps^3\left\{e^{3z}[9\cos (3x)+4\cos x)]-2e^z\cos x\right\}+O(\eps^4).
\]
 Given these expansions, we return to 
\bq\label{sys:Tt:100}
\begin{cases}
\Delta_{x, z}\Tt-\beta\cJ\Tt=0\quad\text{in } \{(x, z): z<0\},\\
\Tt(x, 0)=f(x), \quad     
\na_{x, z}\Tt\to 0\quad\text{as } z\to -\infty,
\end{cases}
\eq
as defined in \eqref{system:Theta}.  We recall from the proof of  Theorem \ref{theo:flattenG} that $\cG_{\eps, \beta}f=\p_z\Tt(\cdot, 0)$ and $\Tt$ is analytic in $(\varepsilon, \beta)$. In particular, for fixed $\beta>0$,  we can expand $\Tt=\Tt^0+\eps \Tt^1+\eps^2\Tt^2+O(\eps^3)$, where $\Tt^0$, $\Tt^1$, $\Tt^2$, and $\Tt^3$ respectively satisfy
\bq\label{sys:Tt0}
\begin{cases}
\Delta_{x, z}\Tt^0-\beta\Tt^0=0\quad\text{in } \{(x, z): z<0\},\\
\Tt^0(x, 0)=f(x),\\
\p_z\Tt^0\to 0\quad\text{as } z\to -\infty,
\end{cases}
\eq
\bq\label{sys:Tt1}
\begin{cases}
\Delta_{x, z}\Tt^1-\beta\Tt^1={2\beta} e^z\cos x\Tt^0\quad\text{in } \{(x, z): z<0\},\\
\Tt^1(x, 0)=0,\\
\p_z\Tt^1\to 0\quad\text{as } z\to -\infty,
\end{cases}
\eq
\bq\label{sys:Tt2}
\begin{cases}
\Delta_{x, z}\Tt^2-\beta\Tt^2={2\beta} e^z\cos x\Tt^1+\beta e^{2z}(1{+4}\cos(2x))\Tt^0\quad\text{in } \{(x, z): z<0\},\\
\Tt^2(x, 0)=0,\\
\p_z\Tt^2\to 0\quad\text{as } z\to -\infty,
\end{cases}
\eq
\bq\label{sys:Tt3}
\begin{cases}
\begin{aligned}\Delta_{x, z}\Tt^3-\beta\Tt^3&=2\beta e^z\cos x\Tt^2+\beta e^{2z}(1+4\cos(2x))\Tt^1\\
&\quad+\beta e^{3z}[{9\cos(3x)+4\cos x}]\Tt^0{+2}\beta e^z\cos x\Tt^0\quad\text{in } \{(x, z): z<0\},
\end{aligned}\\
\Tt^3(x, 0)=0,\\
\p_z\Tt^3\to 0\quad\text{as } z\to -\infty.
\end{cases}
\eq
Comparing with the expansion \eqref{expandcG}, we find that $R_jf=\p_z\Tt^j(\cdot, 0)$. We  remark that when $\beta=0$, it is a classical calculation for \eqref{sys:Tt:100} that $\p_z\Tt(\cdot, 0)=|D|f$, which is the Fourier multiplier obtained in \cite{NguyenStrauss}.

We consider the $\Theta^j$, one at a time.  For $j=0$, 
the solution of \eqref{sys:Tt0} is given by
\bq\label{Tt0}
\wh{\Tt^0}(k, z)=\wh{f}(k)e^{\Om(k)z},\quad \Om(k):=(k^2+\beta)^\mez,
\eq
where $\wh{\Tt^0}(k, z)$ denotes the Fourier coefficient of $\Tt^0(x, z)$ with respect to $x$. Consequently
\bq
R_0f=\p_z\Tt(\cdot, 0)=\Om(D)f.
\eq

Next, we consider the equation for $\Theta^1$. Taking the Fourier transform in $x$ and using \eqref{sys:Tt1} and \eqref{Tt0}, we obtain the following equation for $g(k, z):=\wh{\Tt^1}(k, z)$: 
 \bq\label{ODE:g}
 \begin{aligned}
\p_z^2g-(k^2+\beta)g&={\beta }e^z(\wh{\Tt^0}(k-1, z)+ \wh{\Tt^0}(k+1, z))\\
&={\beta }e^z\wh{f}(k-1)e^{\Om(k-1)z}{+\beta} e^z\wh{f}(k+1)e^{\Om(k+1)z}.
\end{aligned}
\eq
Now the general solution of a generic, inhomogeneous ODE of the form
\[
\p_z^2P-(k^2+\beta)P=F(k, z)
\]
 is 
 \[
 P(k, z)=C_1e^{\Om(k)z}+C_2e^{-\Om(k)z}+P_*(k, z),
\]
where  $P_*$ is the particular solution
\bq\label{G*}
P_*(k, z)=\frac{1}{2\Om(k)}e^{\Om(k)z}\int e^{-\Om(k)z}F(k, z)dz-\frac{1}{2\Om(k)}e^{-\Om(k)z}\int e^{\Om(k)z}F(k, z)dz.
\eq
Using \eqref{G*} we find that the particular solution of \eqref{ODE:g} that vanishes as $z\to -\infty$ together with its derivatives, is
\begin{align*}
g_*(k, z)&=\beta\wh{f}(k-1)e^{[\Om(k-1)+1]z}\{[\Om(k-1)+\Om(k)+1][\Om(k-1)-\Om(k)+1]\}^{-1}\\
&\quad+\beta\wh{f}(k+1)e^{[\Om(k+1)+1]z}\{[\Om(k+1)+\Om(k)+1][\Om(k+1)-\Om(k)+1]\}^{-1}. 
\end{align*}
Since $\p_z g\to 0$ as $z\to -\infty$ and $g(k, 0)=0$, we deduce
\bq\label{Tt1}
\begin{aligned}
\wh{\Tt^1}(k, z)&=g(k, z)=-g_*(k, 0)e^{\Om(k)z}+g_*(k, z),\\
&={\beta}\wh{f}(k-1)\left\{e^{[\Om(k-1)+1]z}-e^{\Om(k)z}\right\}A_k^-+\beta\wh{f}(k+1)\left\{e^{[\Om(k+1)+1]z}-e^{\Om(k)z}\right\}A^+_k,
\end{aligned}
\eq
where we have denoted 
\bq 
A^\pm_k=\{[\Om(k\pm 1)+\Om(k)+1][\Om(k\pm 1)-\Om(k)+1]\}^{-1}. \label{ak1}
\eq
It follows that
\[
\p_zg(k, 0)= \widehat{R_1f}(k) = {C^-_k}\wh{f}(k-1){+ C^+_k}\wh{f}(k+1),
\]
where
\bq  \label{coeff:C}
C^\pm_k=\beta[\Om(k\pm 1)+\Om(k)+1]^{-1}.
\eq
Note  that for $k\in \Zz$, $\Om(k\pm1)-\Om(k)+1\ne 0$, so the coefficients $C_k$ are well-defined. 

Next we consider the equation for $\Theta^2$.  Defining $h(k, z)=\wh{\Tt^2}(k, z)$, we find from \eqref{sys:Tt2} that
 \bq\label{ODE:h}
 \begin{aligned}
\p_z^2h-(k^2+\beta)h=\wh{H}(k, z),~~\textrm{with} ~~ H(x, z)={2\beta} e^z\cos x\Tt^1{+\beta} e^{2z}(1+4\cos(2x))\Tt^0.
\end{aligned}
\eq
Using \eqref{Tt0} and \eqref{Tt1} we compute
\bq
\begin{aligned}
\wh{H}(k, z)&=-\beta e^z\{\wh{\Tt^1}(k-1)+\wh{\Tt^1}(k+1)\}+\beta e^{2z}\{\wh{\Tt^0}(k){+2}\wh{\Tt^0}(k-2){+2}\wh{\Tt^0}(k+2)\}\\
&=\beta^2e^z\left\{\wh{f}(k-2)(e^{[\Om(k-2)+1]z}-e^{\Om(k-1)z})A^-_{k-1} +\wh{f}(k)(e^{[\Om(k)+1]z}-e^{\Om(k-1)z})A^+_{k-1}\right.\\
&\quad\quad\left.+\wh{f}(k)(e^{[\Om(k)+1]z}-e^{\Om(k+1)z})A^-_{k+1}+\wh{f}(k+2)(e^{[\Om(k+2)+1]z}-e^{\Om(k+1)z})A^+_{k+1}\right\}\\
&\quad+ \beta e^{2z}\left\{\wh{f}(k)e^{\Om(k)z}{+2}\wh{f}(k-2)e^{\Om(k-2)z}{+2}\wh{f}(k+2)e^{\Om(k+2)z}\right\}.
\end{aligned}
\eq
Thus in view of \eqref{G*}, the particular solution of \eqref{ODE:h} is $h_*=h_*^--h_*^+$, where
\bq
\begin{aligned}
h_*^\pm(k, z)&=\frac{\beta^2}{2\Om(k)}\left\{\wh{f}(k-2)A^-_{k-1}\left[\frac{e^{[\Om(k-2)+2]z}}{\Om(k-2)\pm\Om(k)+2}-\frac{e^{[1+\Om(k-1)]z}}{1+\Om(k-1)\pm\Om(k)}\right] \right.\\
&\quad +\left.\wh{f}(k)A^+_{k-1}\left[\frac{e^{[\Om(k)+2]z}}{2+\Om(k)\pm\Om(k)}-\frac{e^{[1+\Om(k-1)]z}}{1+\Om(k-1)\pm\Om(k)}\right]\right.\\
&\quad\left.+\wh{f}(k)A^-_{k+1}\left[\frac{e^{[\Om(k)+2]z}}{2}-\frac{e^{[1+\Om(k+1)]z}}{1+\Om(k+1)-\Om(k)}\right]  \right.\\
&\quad +\left.\wh{f}(k+2)A^+_{k+1}\left[\frac{e^{[\Om(k+2)+2]z}}{\Om(k+2)\pm\Om(k)+2}-\frac{e^{[1+\Om(k+1)]z}}{1+\Om(k+1)\pm\Om(k)}\right]\right\}\\
&+\frac{\beta}{2\Om(k)}\left\{\wh{f}(k)\frac{e^{[\Om(k)+2]z}}{2+\Om(k)\pm\Om(k)}{+2}\wh{f}(k-2)\frac{e^{[\Om(k-2)+2]z}}{\Om(k-2)\pm \Om(k)+2}\right.\\
&\quad\quad\left.{+2}\wh{f}(k+2)\frac{e^{[\Om(k+2)+2]z}}{\Om(k+2)\pm\Om(k)+2}\right\}.
\end{aligned}
\eq
Combining terms yields 
\bq
\begin{aligned}
h_*(k, z)&=\beta^2\wh{f}(k-2)\left\{e^{[\Om(k-2)+2]z}A^-_{k, 2}(A^-_{k-1}{+\frac{2}{\beta}})-e^{[\Om(k-1)+1]z}A^-_{k-1}A^-_k \right\}\\
&\quad +\beta^2\wh{f}(k+2)\left\{e^{[\Om(k+2)+2]z}A^+_{k, 2}(A^+_{k+1}{+\frac{2}{\beta}})-e^{[\Om(k+1)+1]z}A^+_kA^+_{k+1} \right\}\\
&\quad+\beta^2\wh{f}(k)\left\{e^{[\Om(k)+2]z}\frac{A^+_{k-1}+A^-_{k+1}+\beta^{-1}}{4[\Om(k)+1]}\right.\\
&\quad\quad\left.-e^{[\Om(k-1)+1]z}A^+_{k-1}A^-_k-e^{[\Om(k+1)+1]z}A^-_{k+1}A^+_k\right\},
\end{aligned}
\eq
where we have denoted
\bq
A^\pm_{k, 2}=\{[\Om(k\pm 2)+\Om(k)+2][\Om(k\pm 2)-\Om(k)+2]\}^{-1}. \label{ak2}
\eq
Clearly, $h_*\to 0$ as $z\to -\infty$. Consequently,
\bq
\begin{aligned}
h(k, z)&=\wh{\Tt^2}(k, z)=-h_*(k, 0)e^{\Om(k)z}+h_*(k, z)\\
&=\beta^2\wh{f}(k-2)\left\{[e^{[\Om(k-2)+2]z}-e^{\Om(k)z}]A^-_{k, 2}(A^-_{k-1}{+\frac{2}{\beta}})-[e^{[\Om(k-1)+1]z}-e^{\Om(k)z}]A^-_{k-1}A^-_k \right\}\\
&\quad +\beta^2\wh{f}(k+2)\left\{[e^{[\Om(k+2)+2]z}-e^{\Om(k)z}]A^+_{k, 2}(A^+_{k+1}{+\frac{2}{\beta}})-[e^{[\Om(k+1)+1]z}-e^{\Om(k)z}]A^+_kA^+_{k+1} \right\}\\
&\quad+\beta^2\wh{f}(k)\left\{[e^{[\Om(k)+2]z}-e^{\Om(k)z}]\frac{A^+_{k-1}+A^-_{k+1}+\beta^{-1}}{4[\Om(k)+1]}\right.\\
&\quad\quad\left.-[e^{[\Om(k-1)+1]z}-e^{\Om(k)z}]A^+_{k-1}A^-_k-[e^{[\Om(k+1)+1]z}-e^{\Om(k)z}]A^-_{k+1}A^+_k\right\}
\end{aligned}
\eq
and 
\bq
\begin{aligned}
\p_zh(k, 0)
&=\beta^2\wh{f}(k-2)\left\{[\Om(k-2)+2-\Om(k)]A^-_{k, 2}(A^-_{k-1}{+\frac{2}{\beta}})-[\Om(k-1)+1-\Om(k)]A^-_{k-1}A^-_k \right\}\\
&\quad +\beta^2\wh{f}(k+2)\left\{[\Om(k+2)+2-\Om(k)]A^+_{k, 2}(A^+_{k+1}{+\frac{2}{\beta}})-[\Om(k+1)+1-\Om(k)]A^+_kA^+_{k+1} \right\}\\
&\quad+\beta^2\wh{f}(k)\left\{[\Om(k)+2-\Om(k)]\frac{A^+_{k-1}+A^-_{k+1}+\beta^{-1}}{4[\Om(k)+1]}\right.\\
&\quad\quad\left.-[\Om(k-1)+1-\Om(k)]A^+_{k-1}A^-_k-[\Om(k+1)+1-\Om(k)]A^-_{k+1}A^+_k\right\}.
\end{aligned}
\eq
Setting
\bq.  \label{coeff:B}
\begin{aligned}
B^-_k&=\beta\left([\Om(k-2)+2-\Om(k)]A^-_{k, 2}(\beta A^-_{k-1}{+2})-\beta[\Om(k-1)+1-\Om(k)]A^-_{k-1}A^-_k\right) ,\\
B^+_k&= \beta\left([\Om(k+2)+2-\Om(k)]A^+_{k, 2}(\beta A^+_{k+1}{+2})-\beta[\Om(k+1)+1-\Om(k)]A^+_kA^+_{k+1}\right),\\
B^0_k&=\beta\Big(\frac{\beta(A^+_{k-1}+A^-_{k+1})+1}{2[\Om(k)+1]}-\beta[\Om(k-1)+1-\Om(k)]A^+_{k-1}A^-_k\\
&\quad-\beta[\Om(k+1)+1-\Om(k)]A^-_{k+1}A^+_k\Big),
\end{aligned}
\eq
we obtain
\bq
\widehat{R_2f}(k) =\p_zh(k, 0)= \wh{f}(k-2)B^-_k+\wh{f}(k)B^0_k+\wh{f}(k+2)B^+_k,
\eq
as desired. 
\begin{rema}
Substituting \eqref{ak1} and \eqref{ak2} into \eqref{coeff:B}, we find equivalent expressions of the $B$ coefficients that depend only on $\beta$ and the dispersion relation $\Omega$:
\bq  \label{coeff:B}
\begin{aligned}
B^-_k&=\frac{\beta  \left(2-\frac{\beta }{(\Omega (k-2)+\Omega (k-1)+1) (\Omega (k-1)+\Omega (k)+1)}\right)}{\Omega (k-2)+\Omega (k)+2},\\
B^+_k&= \frac{\beta  \left(2-\frac{\beta }{(\Omega (k)+\Omega (k+1)+1) (\Omega (k+1)+\Omega (k+2)+1)}\right)}{\Omega (k)+\Omega (k+2)+2},\\
B^0_k&=\frac{\beta  \left(\beta  \left(-\frac{1}{(\Omega (k)+\Omega (k+1)+1)^2}-\frac{1}{(\Omega (k-1)+\Omega (k)+1)^2}\right)+1\right)}{2 (\Omega (k)+1)}.
\end{aligned}
\eq
\end{rema}
Now we consider the equation for $\Theta^3$.  
From \eqref{sys:Tt3} we find that $m(k, z):=\wh{\Tt^3}(k, z)$ obeys
\bq\label{eq:m}
\begin{aligned}
&\p_z^2m(k, z)-(k^2+\beta)m(k, z)=\wh{M}(k, z),  \quad \text{ where }\\
& M(x,z) ={2\beta} e^z\cos x\Tt^2+\beta e^{2z}(1{+4}\cos(2x))\Tt^1\\
&\qquad\qquad\qquad+\beta e^{3z}[{4\cos x+9\cos(3x)}]\Tt^0{-2}\beta e^z\cos x\Tt^0.
\end{aligned}
\eq
Taking the Fourier transform of $M$ with respect to $x$, we compute
\begin{align}
\wh{M}(k, z)&={\beta} e^z[\wh{\Tt^2}(k-1, z)+\wh{\Tt^2}(k+1, z)]+\beta e^{2z}[\wh{\Tt^1}(k, z){+2}\wh{\Tt^1}(k-2, z){+2}\wh{\Tt^1}(k+2, z)] \nonumber \\
&\quad{+\frac{9}{2}} \beta e^{3z}[\wh{\Tt^0}(k-3, z)+\wh{\Tt^0}(k+3, z)]+\beta({2e^{3z}-e^z}) [\wh{\Tt^0}(k-1, z)+\wh{\Tt^0}(k+1, z)], \label{Meqn}
\end{align}
which implies $m(k, z)$, and thus $\p_zm(k,0)$, is a linear combination of $\wh f(k\pm 1)$ and $\wh f(k\pm 3)$. In Sections 5 and 6, we will require only the coefficients of $\wh{f}(k\pm 3)$. For this reason, the coefficients  of $\wh{f}(k\pm 1)$ are omitted in the calculations below and instead replaced by placeholder variables $E_M^{\pm} = E_M^{\pm}(z,k,\beta)$, respectively. Explicit formulas for $E_M^{\pm}$ can be found in the companion Mathematica file. 
\allowdisplaybreaks
\begin{align*}
&\wh{M}(k, z)\\
&={\wh{f}(k-3)}\beta^3 e^z\left\{[e^{[\Om(k-3)+2]z}-e^{\Om(k-1)z}]A^-_{k-1, 2}(A^-_{k-2}{+\frac{2}{\beta}})-[e^{[\Om(k-2)+1]z}-e^{\Om(k-1)z}]A^-_{k-2}A^-_{k-1} \right\}\\
&\quad {+}\wh{f}(k+3)\beta^3 e^z\left\{[e^{[\Om(k+3)+2]z}-e^{\Om(k+1)z}]A^+_{k+1, 2}(A^+_{k+2}{+\frac{2}{\beta}})-[e^{[\Om(k+2)+1]z}-e^{\Om(k+1)z}]A^+_{k+1}A^+_{k+2} \right\}\\
&\quad {+2}\wh{f}(k-3)\beta^2e^{2z}\left\{e^{[\Om(k-3)+1]z}-e^{\Om(k-2)z}\right\}A_{k-2}^-{+2}\wh{f}(k+3)\beta^2e^{2z}\left\{e^{[\Om(k+3)+1]z}-e^{\Om(k+2)z}\right\}A^+_{k+2}\\
&\quad{+\frac92} \beta e^{3z}[\wh{f}(k-3)e^{\Om(k-3)z}+\wh{f}(k+3)e^{\Om(k+3)z}] + E_M^{+}(z, k, \beta)\wh{f}(k+ 1)+E_M^{-}(z, k, \beta) \wh{f}(k- 1).\\\\
&={\wh{f}(k-3)} \left\{\beta^3[e^{[\Om(k-3)+3]z}-e^{[\Om(k-1)+1]z}]A^-_{k-1, 2}(A^-_{k-2}{+\frac{2}{\beta}})-\beta^3[e^{[\Om(k-2)+2]z}-e^{[\Om(k-1)+1]z}]A^-_{k-2}A^-_{k-1} \right\}\\
&\quad {+}\wh{f}(k+3) \left\{\beta^3[e^{[\Om(k+3)+3]z}-e^{[\Om(k+1)+1]z}]A^+_{k+1, 2}(A^+_{k+2}{+\frac{2}{\beta}})-\beta^3[e^{[\Om(k+2)+2]z}-e^{[\Om(k+1)+1]z}]A^+_{k+1}A^+_{k+2} \right\}\\
&\quad +{2}\wh{f}(k-3)\beta^2\left\{e^{[\Om(k-3)+3]z}-e^{[\Om(k-2)+2]z}\right\}A_{k-2}^-{+2}\wh{f}(k+3)\beta^2\left\{e^{[\Om(k+3)+3]z}-e^{[\Om(k+2)+2]z}\right\}A^+_{k+2}\\
&\quad{\frac92} \beta [\wh{f}(k-3)e^{[\Om(k-3)+3]z}+\wh{f}(k+3)e^{[\Om(k+3)+3]z}] + E_M^{+}(z, k, \beta)\wh{f}(k+ 1)+E_M^{-}(z, k, \beta) \wh{f}(k- 1).\\
&=\wh{f}(k-3) \left\{{\beta^3}[e^{[\Om(k-3)+3]z}-e^{[\Om(k-1)+1]z}]A^-_{k-1, 2}(A^-_{k-2}{+\frac{2}{\beta}})\right.\\
&\quad\left.{-\beta^3}[e^{[\Om(k-2)+2]z}-e^{[\Om(k-1)+1]z}]A^-_{k-2}A^-_{k-1}+{2\beta^2}[e^{[\Om(k-3)+3]z}-e^{[\Om(k-2)+2]z}]A_{k-2}^-{+\frac92} \beta e^{[\Om(k-3)+3]z}\right\}\\
&\quad  +\wh{f}(k+3) \left\{-\beta^3[e^{[\Om(k+3)+3]z}-e^{[\Om(k+1)+1]z}]A^+_{k+1, 2}(A^+_{k+2}{+\frac{2}{\beta}})\right.\\
&\quad\left. {-\beta^3}[e^{[\Om(k+2)+2]z}-e^{[\Om(k+1)+1]z}]A^+_{k+1}A^+_{k+2}+{2\beta^2}[e^{[\Om(k+3)+3]z}-e^{[\Om(k+2)+2]z}]A^+_{k+2}{+\frac92}\beta e^{[\Om(k+3)+3]z}\right\}\\
&\quad+E_M^{+}(z, k, \beta)\wh{f}(k+ 1)+E_M^{-}(z, k, \beta) \wh{f}(k- 1).
\end{align*}
Using  \eqref{G*} we find  that the particular solution of \eqref{eq:m} is $m_*=m_*^--m^+_*$, where
\bq
\begin{aligned}
m^\pm_*&=\wh{f}(k-3)\frac{1}{2\Om(k)} \left\{ {\beta^3}\left[\frac{e^{[\Om(k-3)+3]z}}{\Om(k-3)+3\pm\Om(k)}-\frac{e^{[\Om(k-1)+1]z}}{\Om(k-1)+1\pm\Om(k)}\right]A^-_{k-1, 2}(A^-_{k-2} {+\frac{2}{\beta}})\right.\\
&\quad\left. {-}\beta^3\left[\frac{e^{[\Om(k-2)+2]z}}{\Om(k-2)+2\pm\Om(k)}-\frac{e^{[\Om(k-1)+1]z}}{\Om(k-1)+1\pm\Om(k)}\right]A^-_{k-2}A^-_{k-1}\right.\\
&\quad\left.+ {2}\beta^2\left[\frac{e^{[\Om(k-3)+3]z}}{\Om(k-3)+3\pm\Om(k)}-\frac{e^{[\Om(k-2)+2]z}}{\Om(k-2)+2\pm\Om(k)}\right]A_{k-2}^- {+\frac92} \beta\frac{e^{[\Om(k-3)+3]z}}{\Om(k-3)+3\pm\Om(k)}\right\}\\
&\quad  +\wh{f}(k+3)\frac{1}{2\Om(k)} \left\{ {\beta^3}\left[\frac{e^{[\Om(k+3)+3]z}}{\Om(k+3)+3\pm\Om(k)}-\frac{e^{[\Om(k+1)+1]z}}{\Om(k+1)+1\pm\Om(k)}\right]A^+_{k+1, 2}(A^+_{k+2} {+\frac{2}{\beta}})\right.\\
&\quad\left.  {-\beta^3}\left[\frac{e^{[\Om(k+2)+2]z}}{\Om(k+2)+2\pm\Om(k)}-\frac{e^{[\Om(k+1)+1]z}}{\Om(k+1)+1\pm\Om(k)}\right]A^+_{k+1}A^+_{k+2}\right.\\
&\quad\left.+ {2\beta^2}\left[\frac{e^{[\Om(k+3)+3]z}}{\Om(k+3)+3\pm\Om(k)}-\frac{e^{[\Om(k+2)+2]z}}{\Om(k+2)+2\pm\Om(k)}\right]A^+_{k+2} {+\frac92} \beta\frac{e^{[\Om(k+3)+3]z}}{\Om(k+3)+3\pm\Om(k)}\right\}\\
&\quad+E_{m_*^{\pm}}^{+}(z, k, \beta)\wh{f}(k+ 1)+E_{m_*^{\pm}}^{-}(z, k, \beta) \wh{f}(k- 1).
\end{aligned}
\eq
Here, $E_{m_*^{\pm}}^{+}$ is a placeholder for the coefficient of $\hat{f}(k+1)$ in $m_*^{\pm}$ and similarly for $E_{m_*^{\pm}}^{-}$.
Since $m(k, 0)=0$, the full solution of \eqref{eq:m} is $m=m^--m^+$, where
\bq
\begin{aligned}
m^\pm&=-m_*^\pm(k, 0)e^{\Om(k)z}+m_*^\pm(k, z) \\
&=\wh{f}(k-3)\frac{1}{2\Om(k)} \left\{{\beta^3}\left[\frac{e^{[\Om(k-3)+3]z}-e^{\Om(k)z}}{\Om(k-3)+3\pm\Om(k)}-\frac{e^{[\Om(k-1)+1]z}-e^{\Om(k)z}}{\Om(k-1)+1\pm\Om(k)}\right]A^-_{k-1, 2}(A^-_{k-2}{+\frac{2}{\beta}})\right.\\
&\quad\left.{-}\beta^3\left[\frac{e^{[\Om(k-2)+2]z}-e^{\Om(k)z}}{\Om(k-2)+2\pm\Om(k)}-\frac{e^{[\Om(k-1)+1]z}-e^{\Om(k)z}}{\Om(k-1)+1\pm\Om(k)}\right]A^-_{k-2}A^-_{k-1}\right.\\
&\quad\left.+{2}\beta^2\left[\frac{e^{[\Om(k-3)+3]z}-e^{\Om(k)z}}{\Om(k-3)+3\pm\Om(k)}-\frac{e^{[\Om(k-2)+2]z}-e^{\Om(k)z}}{\Om(k-2)+2\pm\Om(k)}\right]A_{k-2}^-{+\frac92} \beta\frac{e^{[\Om(k-3)+3]z}-e^{\Om(k)z}}{\Om(k-3)+3\pm\Om(k)}\right\}\\
&\quad  +\wh{f}(k+3)\frac{1}{2\Om(k)} \left\{{\beta^3}\left[\frac{e^{[\Om(k+3)+3]z}-e^{\Om(k)z}}{\Om(k+3)+3\pm\Om(k)}-\frac{e^{[\Om(k+1)+1]z}-e^{\Om(k)z}}{\Om(k+1)+1\pm\Om(k)}\right]A^+_{k+1, 2}(A^+_{k+2}{+\frac{2}{\beta}})\right.\\
&\quad\left. {-}\beta^3\left[\frac{e^{[\Om(k+2)+2]z}-e^{\Om(k)z}}{\Om(k+2)+2\pm\Om(k)}-\frac{e^{[\Om(k+1)+1]z}-e^{\Om(k)z}}{\Om(k+1)+1\pm\Om(k)}\right]A^+_{k+1}A^+_{k+2}\right.\\
&\quad\left.+{2}\beta^2\left[\frac{e^{[\Om(k+3)+3]z}-e^{\Om(k)z}}{\Om(k+3)+3\pm\Om(k)}-\frac{e^{[\Om(k+2)+2]z}-e^{\Om(k)z}}{\Om(k+2)+2\pm\Om(k)}\right]A^+_{k+2}{+\frac92} \beta\frac{e^{[\Om(k+3)+3]z}-e^{\Om(k)z}}{\Om(k+3)+3\pm\Om(k)}\right\}\\
&\quad+E_{m^\pm}^{+}(z, k, \beta)\wh{f}(k+ 1)+E_{m^\pm}^{-}(z, k, \beta) \wh{f}(k- 1).
\end{aligned}
\eq
Yet again, $E_{m^{\pm}}^{+}$ is a placeholder for the coefficient of $\hat{f}(k+1)$ in $m^{\pm}$ and similarly for $E_{m^{\pm}}^{-}$. Differentiating $m$ in $z$ and setting $z=0$ yields

\bq
\begin{aligned} \label{coeff:D}
\wh{R_3f}(k)&=\p_zm(k, 0)\\
&=\beta\wh{f}(k-3)\left\{{\beta}\left[\frac{1}{\Om(k-3)+3+\Om(k)}-\frac{1}{\Om(k-1)+1+\Om(k)}\right]A^-_{k-1, 2}(\beta A^-_{k-2}{+2})\right.\\
&\quad\left.{-}\beta^2\left[\frac{1}{\Om(k-2)+2+\Om(k)}-\frac{1}{\Om(k-1)+1+\Om(k)}\right]A^-_{k-2}A^-_{k-1}\right.\\
&\quad\left.+{2}\beta\left[\frac{1}{\Om(k-3)+3+\Om(k)}-\frac{1}{\Om(k-2)+2+\Om(k)}\right]A_{k-2}^-{+\frac92} \frac{1}{\Om(k-3)+3+\Om(k)}\right\}\\
&\quad  +\beta\wh{f}(k+3) \left\{{\beta}\left[\frac{1}{\Om(k+3)+3+\Om(k)}-\frac{1}{\Om(k+1)+1+\Om(k)}\right]A^+_{k+1, 2}(\beta A^+_{k+2}{+2})\right.\\
&\quad\left. {-}\beta^2\left[\frac{1}{\Om(k+2)+2+\Om(k)}-\frac{1}{\Om(k+1)+1+\Om(k)}\right]A^+_{k+1}A^+_{k+2}\right.\\
&\quad\left.+{2}\beta\left[\frac{1}{\Om(k+3)+3+\Om(k)}-\frac{1}{\Om(k+2)+2+\Om(k)}\right]A^+_{k+2}{+\frac92} \frac{1}{\Om(k+3)+3+\Om(k)}\right\}\\
&\quad+ D^{-1}_k\wh{f}(k- 1)+ D^{1}_k\wh{f}(k+1).\\
&=D^{-3}_k\wh{f}(k-3)+D^3_k\wh{f}(k+3)+ D^{-1}_k\wh{f}(k- 1)+ D^{1}_k\wh{f}(k+1),
\end{aligned}
\eq
where $D_{k}^{\pm 1}$ are placeholders for the coefficients of $\hat{f}(k\pm 1)$.
\begin{rema}
Similar to the previous remark, one can substitute \eqref{ak1} and \eqref{ak2} into \eqref{coeff:D} to find equivalent expressions of the $D$ coefficients that depend only on $\beta$ and the dispersion relation $\Omega$. For $D_k^{\pm 3}$, we have
\bq
\label{coeff:Dv2:1}
\begin{aligned}
D^{-3}_k &= \beta  \biggr\{2 \beta ^2 \Big[\Omega (k-3)+\Omega (k-2)+\Omega (k-1)+\Omega (k)+4\Big]-4 \beta  \Big[\Omega (k-2)+\Omega (k-1)+1\Big] \\ &\quad \cdot \Big[\Omega (k-2)^2+\big(\Omega (k)+3\big) \Omega (k-2)+3 \Omega (k)+\Omega (k-1) \big(\Omega (k-1)+\Omega (k)+3\big) \\ &\quad +\Omega (k-3) \big(\Omega (k-2)+\Omega (k-1)+2 \Omega (k)+3\big)+4\Big]+9 \Big[\Omega (k-3)+\Omega (k-2)+1\Big] \\ &\quad \cdot \Big[\Omega (k-3)+\Omega (k-1)+2\Big] \Big[\Omega (k-2)+\Omega (k-1)+1\Big] \Big[\Omega (k-2)+\Omega (k)+2\Big] \\ &\quad \cdot \Big[\Omega (k-1)+\Omega (k)+1\Big]\biggr\} \biggr/ \biggr\{  2 \Big[\Omega (k-3)+\Omega (k-2)+1\Big]\Big[\Omega (k-3)+\Omega (k-1)+2\Big] \\ &\quad \cdot \Big[\Omega (k-2)+\Omega (k-1)+1\Big]\Big[\Omega (k-3)+\Omega (k)+3\Big]  \Big[\Omega (k-2)+\Omega (k)+2\Big]\Big[\Omega (k-1)+\Omega (k)+1\Big] \biggr\}, 
\end{aligned}
\eq
and
\bq\label{coeff:Dv2:2}
\begin{aligned}
D^{3}_k &= \beta  \biggr\{2 \beta ^2 \Big[\Omega (k)+\Omega (k+1)+\Omega (k+2)+\Omega (k+3)+4\Big]-4 \beta  \Big[\Omega (k+1)+\Omega (k+2)+1\Big] \\ &\quad \cdot \Big[\Omega (k+1)^2+\big(\Omega (k+3)+3\big) \Omega (k+1)+3 \Omega (k+3)+\Omega (k+2) \big(\Omega (k+2)+\Omega (k+3)+3\big)+\Omega (k) \\ &\quad \cdot \big(\Omega (k+1)+\Omega (k+2)+2 \Omega (k+3)+3\big)+4\Big]+9 Big[\Omega (k)+\Omega (k+1)+1\Big] \Big[\Omega (k)+\Omega (k+2)+2\Big] \\ &\quad \cdot \Big[\Omega (k+1)+\Omega (k+2)+1\Big] \Big[\Omega (k+1)+\Omega (k+3)+2\Big] \Big[\Omega (k+2)+\Omega (k+3)+1\Big]\biggr\} \biggr/ \biggr\{ 2 \Big[\Omega (k) \\ &\quad+\Omega (k+1)+1\Big] \Big[\Omega (k)+\Omega (k+2)+2\Big] \Big[\Omega (k+1)+\Omega (k+2)+1\Big] \Big[\Omega (k)+\Omega (k+3)+3\Big] \\ &\quad \cdot \Big[\Omega (k+1)+\Omega (k+3)+2\Big] \Big[\Omega (k+2)+\Omega (k+3)+1\Big]\biggr\}.
\end{aligned}
\eq
The explicit expressions of $D_k^{\pm 1}$ are significantly more cumbersome and can be found in the companion Mathematica file.
\end{rema}This completes the proof of Proposition \ref{prop:Rj}.
\end{proof}
With 
the operators $R_j=R_j(\beta)$ in hand for $j \leq 3$, 
we return to the $\varepsilon$-expansion of the Hamiltonian operator \eqref{expandHepsbeta} and substitute
\begin{align*}
\beta = \beta_* + \delta,
\end{align*}
where $\beta_*$ is the resonant value $\approx 2.73$ and $\delta$ is a small deviation from $\beta_*$. By virtue of Theorem \ref{theo:flattenG}  (ii), the operators $R_j$ depend analytically on $\beta \in (0, \infty)$, 
so that  we can expand  $R_j$ to third order in $\delta$ as
 \bq\label{def:Sjl}
  R_j=\sum_{\ell=0}^3\delta^\ell R_{j, \ell}+O(\delta^4).
 \eq
 Because $\beta$ only appears in the lower right corner of $\cH^j$, for $j\in \{0, 1, 2, 3\}$ we denote   
 \bq
 \begin{aligned}
&\cH^{j ,0}=\cH^j\vert_{\delta=0},\\
& \cH^{j, \ell}=R_{j, \ell}K,\quad \ell \in \{1, 2, 3\}, \quad K=\begin{bmatrix} 0 & 0\\ 0 &1 \end{bmatrix}.
 \end{aligned}
 \eq
Then $\cH^j$  can be expanded in powers of $\delta$ as
 \bq\label{def:Hjl}
 \cH^{j}=\sum_{\ell=0}^3 \delta^\ell\cH^{j, \ell}+O(\delta^4). 
 \eq
Combining \eqref{Hexpansion:eps} and \eqref{def:Hjl}, we have the full expansion of 
the Hamiltonian $\cH_{\eps, \beta_*+\delta}$  in powers of both $\eps$ and $\delta$:
 \bq\label{expand:H}
 \cH_{\eps, \beta_*+\delta}=\sum_{j=0}^3\sum_{\ell=0}^{3}\eps^j\delta^\ell \cH^{j, \ell}+O(|\eps|^4+|\delta|^4).
 \eq
 We are now poised to expand the spectral data of the $2\times 2$ matrix $\textrm{L}_{\varepsilon,\delta}$ as power series in both $\varepsilon$ and $\delta$.

\section{Third-order expansions  of the matrix $\textrm{L}_{\eps, \delta}$}\label{section:expansions}
\subsection{Expansion of the eigenvectors $U^{\eps, \delta}_j$} 
We recall the basis $\{U_1,U_2\}$ from \eqref{def:U} and the projections $P_{\eps,\delta}$ from \eqref{def:Peps}.  
Since $P_{0, 0}U_j=U_j$ for $j=1,2$, we have 
\begin{align*}
U^{\eps, \delta}_j&= \{1-(P_{\eps, \delta}-P_{0, 0})^2\}^{-\mez} P_{\eps, \delta} U_j.
\end{align*}
Denoting 
\bq
P^{m, n}=\p_\eps^m\p_\delta^nP_{\eps, \delta}\vert_{(\eps, \delta)=(0, 0)} ,
\eq
we expand 
\bq\label{expand:P}
\begin{aligned}
P_{\eps, \delta}&=P_{0, 0}+\eps P^{1, 0}+\delta P^{0, 1}+\mez \eps^2 P^{2, 0}+\mez \delta^2 P^{0, 2}+\eps \delta P^{1, 1}\\
&\qquad+\frac16 \eps^3 P^{3, 0}+\mez \eps^2\delta P^{2, 1}+\mez \eps \delta^2 P^{1, 2}+\frac16\delta^3P^{0, 3}+O((\eps+\delta)^4).
\end{aligned}
\eq
This series and the ones that follow converge for small $(\eps,\delta)$ due to the discussion in Section \ref{Sec:Kato}.  
Using the elementary Taylor expansion $(1-x^2)^{-\mez}=1+\mez x^2+O(x^4)$, we find 
\bq\label{expand:Kato:1}
\begin{aligned}
\{I-(P_{\eps, \delta}-P_{0, 0})^2\}^{-\mez}&=I +\mez \eps^2 P^{1, 0}P^{1, 0}+\mez \delta^2 P^{0, 1}P^{0, 1}+\mez \eps \delta(P^{1, 0}P^{0, 1}+P^{0, 1}P^{1, 0})\\
&\quad+\frac14\eps^3(P^{1, 0}P^{2, 0}+P^{2, 0}P^{1, 0})+\frac14\delta^3(P^{0, 1}P^{0, 2}+P^{0, 2}P^{0, 1})\\
&\quad+ \eps^2 \delta\left[\mez(P^{1, 0}P^{1, 1}+P^{1, 1}P^{1, 0})+\frac14 (P^{0, 1}P^{2, 0}+P^{2, 0}P^{0, 1})\right]\\
&\quad+ \eps \delta^2\left[\mez(P^{0, 1}P^{1, 1}+P^{1, 1}P^{0, 1})+\frac14 (P^{1, 0}P^{0, 2}+P^{0, 2}P^{1, 0})\right]\\
&\quad+O((\eps+\delta)^4).
\end{aligned}
\eq
Combining the expansions \eqref{expand:P} and \eqref{expand:Kato:1}, we obtain the expansion of $U^{\eps, \delta}_j$ as 
\bq\label{expand:U}
U^{\eps, \delta}_j=U_j+\sum_{m+n=1}^3\eps^m\delta^nU^{(m, n)}_j+O((\eps+\delta)^4),
\eq
where the coefficients $U_j^{(m, n)}$ are given by
 \allowdisplaybreaks
\begin{align}\label{expand:U:start}
&U^{(1, 0)}_j=P^{1, 0}U_j,\quad U^{(0, 1)}_j=P^{0, 1}U_j,\\
&U^{(2, 0)}_j=\mez (P^{2, 0}+P^{1, 0}P^{1, 0})U_j,\quad U^{(0, 2)}_j=\mez (P^{0, 2}+P^{0, 1}P^{0, 1})U_j,\\
&U^{(1, 1)}_j=(P^{1, 1}+\mez P^{0, 1}P^{1, 0}+\mez P^{1, 0}P^{0, 1})U_j,\\
&U_j^{(3,0)} = \frac16\Big(P^{3,0} + \frac32\left(P^{2,0}P^{1,0}+P^{1,0}P^{2,0} \right)+3P^{1,0}P^{1,0}P^{1,0} \Big)U_j, \\
&U_j^{(0,3)}= \frac16\Big(P^{0,3} + \frac32\left(P^{0,1}P^{0,2}+P^{0,2}P^{0,1} \right)+3P^{0,1}P^{0,1}P^{0,1} \Big)U_j,\\
&U_j^{(2,1)} = \frac12\Big(P^{2,1} + \frac12\left(P^{0,1}P^{2,0} +2P^{1,0}P^{1,1}+2P^{1,1}P^{1,0}+P^{2,0}P^{0,1}\right) \nonumber \\
&\hspace{1cm} +\left(P^{1,0}P^{0,1}+P^{0,1}P^{1,0} \right)P^{1,0} + P^{1,0}P^{1,0}P^{0,1}\Big)U_j,
\\ \label{expand:U:end}
&U_j^{(1,2)} = \frac12\Big(P^{1,2} + \frac12\left(P^{1,0}P^{0,2} +2P^{1,1}P^{0,1}+2P^{0,1}P^{1,1}+P^{0,2}P^{1,0}\right) \nonumber \\
&\hspace{1cm} +\left(P^{1,0}P^{0,1}+P^{0,1}P^{1,0} \right)P^{0,1} + P^{0,1}P^{0,1}P^{1,0}\Big)U_j.
\end{align}
To calculate $P^{m, n}$ we use Neumann series and \eqref{expand:H}, yielding
\[
 \begin{aligned}
& (\cL_{\eps, \beta_*+\delta} - \lb)^{-1}-(\cL_{0, \beta_*} - \lb)^{-1}\\
&=\sum_{k=1}^\infty(-1)^k(\cL_{0, \beta_*} - \lb)^{-1}\left[(\cL_{\eps, \beta_*+\delta}- \cL_{0, \beta_*})(\cL_{0, \beta_*} - \lb)^{-1}\right]^k\\
 &=\sum_{k=1}^\infty(-1)^k(\cL_{0, \beta_*} - \lb)^{-1}\left[\left(\sum_{0\le j, \ell \le 3; j+\ell\ge 1}\eps^j\delta^\ell J\cH^{j, \ell}+O((\eps+\delta)^4)\right)(\cL_{0, \beta_*} - \lb)^{-1}\right]^k.
 \end{aligned}
 \]
 { Of course, due to the analyticity, the infinite series converges 
 and the remainder $O((\eps+\delta)^4)$ is indeed finite in an $(\eps,\delta)$ neighborhood of $(0,0)$ .}  
Denoting 
\[
S_\ld=(\cL_{0, \beta_*} - \lb)^{-1},
\]
we deduce from the preceding series that 
 \allowdisplaybreaks
\begin{align}\label{dL:start}
\p_\eps (\cL_{\eps, \beta_*+\delta} - \lb)^{-1}\vert_{(\eps, \delta)=(0, 0)}&=-S_\ld J\cH^{1, 0}S_\ld,\\
\p_\delta (\cL_{\eps, \beta_*+\delta} - \lb)^{-1}\vert_{(\eps, \delta)=(0, 0)}&=-S_\ld J\cH^{0, 1}S_\ld,\\
\p^2_\eps (\cL_{\eps, \beta_*+\delta} - \lb)^{-1}\vert_{(\eps, \delta)=(0, 0)}&=S_\ld\Big(-2 J\cH^{2, 0}+2J\cH^{1, 0}S_\ld J\cH^{1, 0}\Big)S_\ld,\\
\p^2_\delta (\cL_{\eps, \beta_*+\delta} - \lb)^{-1}\vert_{(\eps, \delta)=(0, 0)}&=S_\ld\Big(-2 J\cH^{0, 2}+2J\cH^{0, 1}S_\ld J\cH^{0, 1}\Big)S_\ld,\\
\p_\eps\p_\delta(\cL_{\eps, \beta_*+\delta} - \lb)^{-1}\vert_{(\eps, \delta)=(0, 0)}&=S_\ld\Big(-J\cH^{1, 1}+J\cH^{1, 0}S_\ld J\cH^{0, 1}+ J\cH^{0, 1}S_\ld J\cH^{1, 0}\Big)S_\ld,\\
\p^3_{\eps}(\cL_{\eps, \beta_*+\delta} - \lb)^{-1}\vert_{(\eps, \delta)=(0, 0)}&=6S_\lambda\Big[-J H^{3,0} + \big(JH^{1,0}S_\lambda JH^{2,0} + JH^{2,0}S_\lambda JH^{1,0}\big) \nonumber\\
&\hspace{1cm}- JH^{1,0}S_\lambda JH^{1,0}S_\lambda JH^{1,0}\Big]S_\lambda,   \\
\p^3_{\delta}(\cL_{\eps, \beta_*+\delta} - \lb)^{-1}\vert_{(\eps, \delta)=(0, 0)}&=6S_\lambda\Big[-J H^{0,3} + \big(JH^{0,1}S_\lambda JH^{0,2} + JH^{0, 2}S_\lambda JH^{0, 1}\big) \nonumber\\
&\hspace{1cm}- JH^{0,1}S_\lambda JH^{0, 1}S_\lambda JH^{0, 1}\Big]S_\lambda, \\
\p^2_{\delta}\p_\delta(\cL_{\eps, \beta_*+\delta} - \lb)^{-1}\vert_{(\eps, \delta)=(0, 0)}&=2S_\lambda\Big[-J H^{2,1} +JH^{1,0}S_\lambda JH^{1,1} + JH^{0,1}S_\lambda JH^{2,0} \nonumber \\
&\hspace{1cm} + JH^{2,0}S_\lambda JH^{0,1} + JH^{1,1}S_\lambda JH^{1,0}  \nonumber
\\
&\hspace{1cm}  -JH^{1,0}S_\lambda \big(JH^{1,0}S_\lambda JH^{0,1}+ JH^{0,1}S_\lambda JH^{1,0}\big)\Big]S_\lambda,\\ \label{dL:end}
\p_{\eps}\p^2_\delta(\cL_{\eps, \beta_*+\delta} - \lb)^{-1}\vert_{(\eps, \delta)=(0, 0)}&=2S_\lambda\Big[-J H^{1, 2} +JH^{0, 1}S_\lambda JH^{1,1} + JH^{1, 0}S_\lambda JH^{0, 2} \nonumber \\
&\hspace{1cm} + JH^{0, 2}S_\lambda JH^{1, 0} + JH^{1,1}S_\lambda JH^{0, 1}  \nonumber
\\
&\hspace{1cm}  -JH^{0, 1}S_\lambda \big(JH^{0, 1}S_\lambda JH^{1, 0}+ JH^{0,1}S_\lambda JH^{0, 1}\big)\Big]S_\lambda.
\end{align}
Since $S_\ld U_j=-(\ld-i\sigma)^{-1}U_j$, we obtain from \eqref{def:Peps} and \eqref{dL:start}-\eqref{dL:end} the following formulas for $P^{m, n}U_j$. 
 \allowdisplaybreaks
\begin{align}\label{Pmn:start}
P^{1, 0}U_j&= \frac1{2\pi i} \int_\Gamma -S_\ld J\cH^{1, 0}U_j \frac{d\lb}{\ld-i\sigma},\\
 P^{0, 1}U_j&= \frac1{2\pi i} \int_\Gamma -S_\ld J\cH^{0, 1}U_j \frac{d\lb}{\ld-i\sigma},\\
P^{2, 0}U_j&=\frac1{2\pi i} \int_\Gamma S_\ld \Big(-2J\cH^{2, 0}+2J\cH^{1, 0} J\cH^{1, 0}\Big)U_j\frac{d\lb}{\ld-i\sigma},\\
P^{0, 2}U_j&=\frac1{2\pi i} \int_\Gamma S_\ld \Big(-2J\cH^{0, 2}+2 J\cH^{0, 1}S_\ld J\cH^{0, 1}\Big)U_j\frac{d\lb}{\ld-i\sigma},\\
P^{1, 1}U_j&=\frac1{2\pi i} \int_\Gamma S_\ld\Big(- J\cH^{1, 1}+ J\cH^{1, 0}S_\ld J\cH^{0, 1}+ J\cH^{0, 1}S_\ld J\cH^{1, 0}\Big)U_j\frac{d\lb}{\ld-i\sigma},\\
P^{3,0}U_j &= \frac{1}{2\pi i}\int_\Gamma 6S_\lambda\Big[-J H^{3,0} + \big(JH^{1,0}S_\lambda JH^{2,0} + JH^{2,0}S_\lambda JH^{1,0}\big) \nonumber\\
&\hspace{1cm}- JH^{1,0}S_\lambda JH^{1,0}S_\lambda JH^{1,0}\Big]U_j \frac{d\lb}{\ld-i\sigma}, 
\\
P^{0,3}U_j &= \frac{1}{2\pi i}\int_\Gamma 6S_\lambda\Big[-J H^{0,3} + \big(JH^{0,1}S_\lambda JH^{0,2} + JH^{0, 2}S_\lambda JH^{0, 1}\big) \nonumber\\
&\hspace{1cm}- JH^{0,1}S_\lambda JH^{0, 1}S_\lambda JH^{0, 1}\Big]U_j \frac{d\lb}{\ld-i\sigma},\\
P^{2,1}U_j &= \frac{1}{2\pi i}\int_\Gamma 2S_\lambda\Big[-J H^{2,1} +JH^{1,0}S_\lambda JH^{1,1} + JH^{0,1}S_\lambda JH^{2,0} \nonumber \\
&\hspace{1cm} + JH^{2,0}S_\lambda JH^{0,1} + JH^{1,1}S_\lambda JH^{1,0}  \nonumber
\\
&\hspace{1cm}  -JH^{1,0}S_\lambda \big(JH^{1,0}S_\lambda JH^{0,1}+ JH^{0,1}S_\lambda JH^{1,0}\big)\Big]U_j \frac{d\lb}{\ld-i\sigma},\\ \label{Pmn:end}
P^{1,2}U_j &= \frac{1}{2\pi i}\int_\Gamma 2S_\lambda\Big[-J H^{1, 2} +JH^{0, 1}S_\lambda JH^{1,1} + JH^{1, 0}S_\lambda JH^{0, 2} \nonumber \\
&\hspace{1cm} + JH^{0, 2}S_\lambda JH^{1, 0} + JH^{1,1}S_\lambda JH^{0, 1}  \nonumber
\\
&\hspace{1cm}  -JH^{0, 1}S_\lambda \big(JH^{0, 1}S_\lambda JH^{1, 0}+ JH^{0,1}S_\lambda JH^{0, 1}\big)\Big]U_j \frac{d\lb}{\ld-i\sigma}.
\end{align}
To summarize, the third-order expansions of the eigenvectors $U_j^{\eps, \delta}$ are given by \eqref{expand:U}, where the coefficients $U_j^{(m, n)}$ are expressed in terms of $P^{m, n}$ as in \eqref{expand:U:start}-\eqref{expand:U:end} which are in turn calculated by the contour integrals in \eqref{Pmn:start}-\eqref{Pmn:end}.

\subsection{Expansions of the matrix $\textrm{L}_{\eps, \delta}$}
We recall that   $\textrm{L}_{\eps, \delta}$ given by \eqref{matrixL}  is the matrix representation of the linearized operator $\mathcal{L}_{\eps, \beta_*+\delta}$ with respect to the basis $\{V_j^{\eps, \delta}: j=1, 2\}$, where  $V_j^{\eps, \delta}=\frac{1}{\sqrt{\gamma_j}}U_j^{\eps, \delta}$. The expansion  \eqref{expand:U} of $U_j^{\eps, \delta}$ implies  
\bq\label{expand:V}
V^{\eps, \delta}_j=V_j+\sum_{m+n=1}^3\eps^m\delta^nV_j^{(m, n)}+O((\eps+\delta)^4)
\eq
with 
\[
V_j=\frac{1}{\sqrt{\gamma_j}}U_j,\quad V_j^{(m, n)}=\frac{1}{\sqrt{\gamma_j}}U_j^{(m, n)}.
\]

Combining \eqref{expand:V} with the expansion \eqref{expand:H} for $\cH_{\eps, \beta_*+ \delta}$, we  expand  the inner products appearing in $L_{\eps, \delta}$ \eqref{matrixL} as
\bq
\left(\mathcal{H}_{\varepsilon,\beta_*+\delta}V_j^{\varepsilon,\delta},V_k^{\varepsilon,\delta} \right)=(\cH^{0, 0}V_j, V_k)+\sum_{m+n=1}^3 \left(\mathcal{H}_{\varepsilon,\beta_*+\delta}V_j^{\varepsilon,\delta},V_k^{\varepsilon,\delta} \right)_{m,n}\eps^m\delta^n+O((\eps+\delta)^4),
\eq
where
 \allowdisplaybreaks
\begin{align}\label{innerproduct:start}
 \left(\mathcal{H}_{\varepsilon,\beta_*+\delta}V_j^{\varepsilon,\delta},V_k^{\varepsilon,\delta} \right)_{1, 0}&=(\cH^{0, 0}V_j, V^{(1, 0)}_k)+(\cH^{1, 0}V_j, V_k)+(\cH^{0, 0}V_j^{(1, 0)}, V_k),\\
 \left(\mathcal{H}_{\varepsilon,\beta_*+\delta}V_i^{\varepsilon,\delta},V_{j}^{\varepsilon,\delta} \right)_{0, 1}&=(\cH^{0, 0}V_j, V^{(0, 1)}_k)+(\cH^{0, 1}V_j, V_k)+(\cH^{0, 0}V_j^{(0, 1)}, V_k),\\
 \left(\mathcal{H}_{\varepsilon,\beta_*+\delta}V_i^{\varepsilon,\delta},V_{j}^{\varepsilon,\delta} \right)_{2, 0}&=(\cH^{0, 0}V_j, V^{(2, 0)}_k)+(\cH^{2, 0}V_j, V_k) +(\cH^{1, 0}V^{(1, 0)}_j, V_k)\nonumber\\
&\hspace{1cm} +(\cH^{0, 0}V^{(2, 0)}_j, V_k)+(\cH^{1, 0}V_j, V^{(1, 0)}_k)+(\cH^{0, 0}V_j^{(1, 0)}, V_k^{(1, 0)}),\\
 \left(\mathcal{H}_{\varepsilon,\beta_*+\delta}V_i^{\varepsilon,\delta},V_{j}^{\varepsilon,\delta} \right)_{0, 2}&=(\cH^{0, 0}V_j, V^{(0, 2)}_k)+(\cH^{0, 2}V_j, V_k) +(\cH^{0, 1}V^{(0, 1)}_j, V_k)\nonumber\\
&\hspace{1cm}+(\cH^{0, 0}V^{(0, 2)}_j, V_k)+(\cH^{0, 1}V_j, V^{(0, 1)}_k)+(\cH^{0, 0}V_j^{(0, 1)}, V_k^{(0, 1)}),\\
 \left(\mathcal{H}_{\varepsilon,\beta_*+\delta}V_i^{\varepsilon,\delta},V_{j}^{\varepsilon,\delta} \right)_{1, 1}&=(\cH^{1, 1}V_j, V_k)+(\cH^{1, 0}V_j^{(0, 1)}, V_k)+(\cH^{0, 1}V_j^{(1, 0)}, V_k)\nonumber\\
&\hspace{1cm} +(\cH^{0, 0}V_j, V_k^{(1, 1)}) +(\cH^{0, 0}V_j^{(1, 1)}, V_k) +(\cH^{1, 0}V_j, V_k^{(0, 1)})\nonumber\\
&\quad+(\cH^{0, 0}V_j^{(1, 0)}, V_k^{(0, 1)})+(\cH^{0, 1}V_j, V_k^{(1, 0)})+(\cH^{0, 0}V^{(0, 1)}_j, V_k^{(1, 0)}),\\
\left(\mathcal{H}_{\varepsilon,\beta_*+\delta} V^{\varepsilon,\delta}_j,V^{\varepsilon,\delta}_k \right)_{3,0} &=\Big(\mathcal{H}^{3,0} V_j + \mathcal{H}^{2,0}V_j^{(1,0)} + \mathcal{H}^{1,0}V_j^{(2,0)} + \mathcal{H}^{0,0}V_j^{(3,0)},V_k\Big) \nonumber \\
&\hspace{1cm}+\Big(\mathcal{H}^{2,0}V_j + \mathcal{H}^{1,0}V_j^{(1,0)} + \mathcal{H}^{0,0}V_j^{(2,0)},V_k^{(1,0)}   \Big) \nonumber \\
&\hspace{1cm}+ \Big(\mathcal{H}^{1,0}V_j + \mathcal{H}^{0,0}V_j^{(1,0)},V_k^{(2,0)}   \Big) + \Big(\mathcal{H}^{0,0}V_j,V_k^{(3,0)}  \Big), \\
\left(\mathcal{H}_{\varepsilon,\beta_*+\delta} V^{\varepsilon,\delta}_j,V^{\varepsilon,\delta}_k \right)_{0,3} &=\Big(\mathcal{H}^{0,3} V_j + \mathcal{H}^{0,2}V_j^{(0,1)} + \mathcal{H}^{0,1}V_j^{(0,2)} + \mathcal{H}^{0,0}V_j^{(0,3)},V_k\Big) \nonumber \\
&\hspace{1cm}+\Big(\mathcal{H}^{0,2}V_j + \mathcal{H}^{0,1}V_j^{(0,1)} + \mathcal{H}^{0,0}V_j^{(0,2)},V_k^{(0,1)}   \Big) \nonumber \\
&\hspace{1cm}+ \Big(\mathcal{H}^{0,1}V_j + \mathcal{H}^{0,0}V_j^{(0,1)},V_k^{(0,2)}   \Big) + \Big(\mathcal{H}^{0,0}V_j,V_k^{(0,3)}  \Big),\\
\left(\mathcal{H}_{\varepsilon,\beta_*+\delta} V^{\varepsilon,\delta}_j,V^{\varepsilon,\delta}_k \right)_{2,1} &= \Big(\mathcal{H}^{2,1}V_j + \mathcal{H}^{2,0}V_j^{(0,1)} + \mathcal{H}^{1,1}V_j^{(1,0)} + \mathcal{H}^{1,0}V_j^{(1,1)} + \mathcal{H}^{0,1}V_j^{(2,0)} + \mathcal{H}^{0,0}V_j^{(2,1)}, V_k \Big) \nonumber \\
&\hspace{1cm}+\Big(\mathcal{H}^{2,0}V_j + \mathcal{H}^{1,0}V_j^{(1,0)} + \mathcal{H}^{0,0}V_j^{(2,0)},V_k^{(0,1)} \Big) \nonumber \\
&\hspace{1cm}+\Big(\mathcal{H}^{1,1}V_j + \mathcal{H}^{1,0}V_j^{(0,1)} + \mathcal{H}^{0,1}V_j^{(1,0)} + \mathcal{H}^{0,0}V_j^{(1,1)},V_{n}^{(1,0)} \Big) \nonumber \\
&\hspace{1cm}+ \Big(\mathcal{H}^{1,0}V_j + \mathcal{H}^{0,0}V_j^{(1,0)},V_k^{(1,1)}  \Big) + \Big(\mathcal{H}^{0,1}V_j + \mathcal{H}^{0,0}V_j^{(0,1)},V_k^{(2,0)} \Big) \nonumber \\
&\hspace{1cm}+ \Big(\mathcal{H}^{0,0}V_j,V_k^{(2,1)} \Big), \\ \label{innerproduct:end}
\left(\mathcal{H}_{\varepsilon,\beta_*+\delta} V^{\varepsilon,\delta}_j,V^{\varepsilon,\delta}_k \right)_{1,2} &= \Big(\mathcal{H}^{1,2}V_j + \mathcal{H}^{1,1}V_j^{(0,1)} + \mathcal{H}^{1,0}V_j^{(0,2)} + \mathcal{H}^{0,2}V_j^{(1,0)} + \mathcal{H}^{0,1}V_j^{(1,1)} + \mathcal{H}^{0,0}V_j^{(1,2)}, V_k \Big) \nonumber \\
&\hspace{1cm}+\Big(\mathcal{H}^{0,2}V_j + \mathcal{H}^{0,1}V_j^{(0,1)} + \mathcal{H}^{0,0}V_j^{(0,2)},V_k^{(1,0)} \Big) \nonumber \\
&\hspace{1cm}+\Big(\mathcal{H}^{1,1}V_j + \mathcal{H}^{1,0}V_j^{(0,1)} + \mathcal{H}^{0,1}V_j^{(1,0)} + \mathcal{H}^{0,0}V_j^{(1,1)},V_{n}^{(0,1)} \Big) \nonumber \\
&\hspace{1cm}+ \Big(\mathcal{H}^{0,1}V_j + \mathcal{H}^{0,0}V_j^{(0,1)},V_k^{(1,1)}  \Big) + \Big(\mathcal{H}^{1,0}V_j + \mathcal{H}^{0,0}V_j^{(1,0)},V_k^{(0,2)} \Big) \nonumber \\
&\hspace{1cm}+ \Big(\mathcal{H}^{0,0}V_j,V_k^{(1,2)} \Big).
\end{align}

\section{Proof of Theorem \ref{theo:main}}
\begin{lemm} \label{freqProp}
The purely imaginary matrix $\textrm{L}_{\eps, \delta}$ can be written as  
\begin{align}
\textrm{L}_{\varepsilon,\delta} = 
\begin{pmatrix} i\sigma & 0 \\ 0 & i\sigma \end{pmatrix} + i\begin{pmatrix} A &B \\ -B& C \end{pmatrix}, 
\end{align}
where $i\sigma$ is the repeated eigenvalue of the unperturbed operator $\mathcal{L}_{0,\beta_*}$ and
where $A$, $B$, and $C$ are real analytic functions of $(\eps, \delta)$ in a neighborhood of $(0,0)$ that have the expansions 
\bq\label{expand:ABC}
\begin{aligned}
A &= a_{0,1}\delta + a_{2,0}\varepsilon^2 + a_{0,2}\delta^2 + a_{2,1}\varepsilon^2\delta +a_{0,3}\delta^3 +a_{4,0}\eps^4   + O(\delta(|\eps|^3+|\delta|^3)),\\
B&=b_{3,0}\varepsilon^3 + O(|\eps|^4+|\delta|^4),\\
C &=c_{0,1}\delta + c_{2,0}\varepsilon^2 + c_{0,2}\delta^2 + c_{2,1}\varepsilon^2\delta +c_{0,3}\delta^3 +c_{4,0}\eps^4 +  O(\delta (|\eps|^3+|\delta|^3)).
\end{aligned}
\eq
\end{lemm} 
\begin{proof} 
The form \eqref{freqProp} follows from the fact that 
 the reverse diagonal satisfies \eqref{L:reversediagonal}, that is, $\left(\textrm{L}_{\varepsilon,\delta}\right)_{12}=-\left(\textrm{L}_{\varepsilon,\delta}\right)_{21}$. Moreover, $A$, $B$, and $C$  admit expansions in $\varepsilon$ and $\delta$ around $(0,0)$ from the joint analyticity of $\textrm{L}_{\varepsilon,\delta}$ around $(0,0)$, and the coefficients $a_{i,j}$, $b_{i,j}$, and $c_{i,j}$ of these expansions are real-valued since $\textrm{L}_{\eps, \delta}$ is purely imaginary. To prove Theorem \ref{theo:main}, $A$, $B$, and $C$ must be expanded to third order in $\varepsilon$ and $\delta$ and also include terms proportional to $\varepsilon^4$ in $A$ and $C$. Later, we will see that these fourth-order terms in $A$ and $C$ drop from our calculations. The main purpose of the present lemma is to show that, up to third order in $\varepsilon$ and $\delta$,  
 the coefficients not listed explicitly in the expansions of $A$, $B$, and $C$ above vanish identically.
 
To be specific, we will prove that all the terms with odd powers of $\eps$ in $A$ and $C$ vanish, as well as   
all the terms in $B$ up to order 3  except for $b_{3,0}$. 
These properties come from 
substituting the expansions \eqref{innerproduct:start}-\eqref{innerproduct:end} of the inner products $(\mathcal{H}_{\varepsilon,\beta_*+\delta}V_j^{\varepsilon,\delta},V_k^{\varepsilon,\delta})$ into the matrix $L_{\eps, \delta}$ \eqref{matrixL}.  
 Note that all eigenfunction coeffiecients $V_j^{(k,\ell)}$ (for $j \in \{1,2\}$ and $k, \ell \geq 0$) 
 defined by \eqref{expand:V}
 take the form of a finite Fourier series that span a small number of wave numbers. In particular, the zeroth-order corrections $V_1$ and $V_2$ span the single wave numbers $\{1\}$ and $\{-2\}$, respectively. In view of the definitions of the operators $R_{k, \ell}$ and $\cH^{k, \ell}$ given in \eqref{def:Sjl} and \eqref{def:Hjl} and of the operators $R_k$ given in \eqref{form:R1}-\eqref{form:R3},  we deduce that $\cH^{0, \ell}$ acts on $V_j^{k,\ell}$ in such a way that preserves its spanning set of wave numbers, while $\cH^{1, \ell}$ modulates (shifts) these wave numbers by $\pm1$, $\cH^{2, \ell}$ modulates the wave numbers by another $\pm1$, and $\cH^{3, \ell}$ modulates the wave numbers by yet another $\pm1$. Using the projection operators \eqref{Pmn:start}-\eqref{Pmn:end} and the formulas for the eigenfunction corrections \eqref{expand:U:start}-\eqref{expand:U:end}, we calculate the spanning set of wave numbers for each eigenfunction correction $V_j^{(k,\ell)}$ up to third order in $\varepsilon$ and $\delta$. These calculations are tedious, but straightforward. The full details of the calculation are provided in our Mathematica file. We summarize our results in the table below.
\begin{center}
\begin{tabular}{ p{2.5cm}|p{2.5cm}||p{2.5cm}|p{2.5cm}  }
 \multicolumn{2}{c}{$V_1^{(k,\ell)}$ Wave Numbers} & \multicolumn{2}{c}{$V_2^{(k,\ell)}$ Wave Numbers}  \\ 
 \hline \hline
 $V_1$ & $\{1\}$ & $V_2$ & $\{-2\}$\\
 \hline
  $V_1^{(1,0)}$ & $\{0,2\}$ & $V_2^{(1,0)}$ & $\{-3,-1\}$\\
 \hline
   $V_1^{(0,1)}$ & $\{1\}$ & $V_2^{(0,1)}$ & $\{-2\}$\\
   \hline
     $V_1^{(2,0)}$ & $\{-1,1,3\}$ & $V_2^{(2,0)}$ & $\{-4,-2,0\}$\\
 \hline
   $V_1^{(1,1)}$ & $\{0,2\}$ & $V_2^{(1,1)}$ & $\{-3,-1\}$\\
 \hline
   $V_1^{(0,2)}$ & $\{1\}$ & $V_2^{(0,2)}$ & $\{-2\}$\\
   \hline
     $V_1^{(3,0)}$ & $\{-2,0,2,4\}$ & $V_2^{(3,0)}$ & $\{-5,-3,-1,1\}$\\
 \hline
   $V_1^{(2,1)}$ & $\{-1,1,3\}$ & $V_2^{(2,1)}$ & $\{-4,-2,0\}$\\
 \hline
   $V_1^{(1,2)}$ & $\{0,2\}$ & $V_2^{(1,2)}$ & $\{-3,-1\}$\\
 \hline
   $V_1^{(0,3)}$ & $\{1\}$ & $V_2^{(0,3)}$ & $\{-2\}$\\
 \hline
 \hline
\end{tabular}
\end{center}

Given the table of wave numbers above as well as the modulational effect of the operators 
$\mathcal{H}^{k,\ell}$, we can deduce the wave numbers present in each term of the inner-product expansions defined in \eqref{innerproduct:start}.  
As mentioned above, $\mathcal{H}^{(0,\ell)}$ induces no change in the wave numbers, 
$\mathcal{H}^{(1,\ell)}$ modulates the wave numbers by $\pm 1$, 
$\mathcal{H}^{(2,\ell)}$ modulates the wave numbers by $\pm 2$, and 
$\mathcal{H}^{(3,\ell)}$ changes the wave numbers by $\pm3, \pm1$.  
For the terms in the off-diagonal $B$, $\mathcal{H}^{k,\ell}$ must contribute a shift of wave number by 3 
in order to obtain a non-zero contribution.  
For the terms in the diagonal entries $A$ and $C$, $\mathcal{H}^{k,\ell}$ must contribute a shift of wave number by 0 or 2  
in order to obtain a non-zero contribution. 

As an example, let us examine the $\eps$ term of the inner-product $\left(\mathcal{H}_{\varepsilon,\beta_*+\delta}V_1^{\varepsilon,\delta},V_1^{\varepsilon,\delta} \right)$, which is given explicitly by  
\[
(\mathcal{H}_{\varepsilon,\beta_*+\delta}V_1^{\varepsilon,\delta},V_1^{\varepsilon,\delta})_{1, 0}=(\cH^{0, 0}V_1, V^{(1, 0)}_1)+(\cH^{1, 0}V_1, V_1)+(\cH^{0, 0}V_1^{(1, 0)}, V_1).
\]
This term is proportional to the coefficient $a_{1,0}$ in the expansion of the $A$ entry of $\textrm{M}_{\varepsilon,\delta}$. The {\it first} inner-product on the right-hand side above must vanish. Indeed, $\mathcal{H}^{0,0}V_1$ consists of only the wave number $\{1\}$, while $V_1^{(1,0)}$ consists of wave numbers $\{0,2\}$. As these two sets have no intersection, the inner-product must vanish. The {\it second} inner-product also vanishes, as $\mathcal{H}^{(1,0)}V_1$ has wave numbers $\{0,2\}$, while $V_1$ has wave number $\{1\}$. Finally, the {\it third} inner-product vanishes for similar reasons as the first inner-product. Thus, we conclude that the $\varepsilon$ term of the inner-product $\left(\mathcal{H}_{\varepsilon,\beta_*+\delta}V_1^{\varepsilon,\delta},V_1^{\varepsilon,\delta} \right)$ vanishes identically and, as a consequence, $a_{1,0} = 0$. 
A similar analysis holds for the other coefficients. 
\end{proof} 

By definition, the coefficients  $a_{i,j}$, $b_{i,j}$, and $c_{i,j}$ depend only on the  parameter $\beta_*$. Using the formulas derived in Section \ref{section:expansions} and the aid of Mathematica's symbolic computation, we obtain {\it explicit} expressions of these coefficients as functions of $\beta_*$. Since these expressions are very cumbersome, we do not display them here, but if we approximate $\beta_* \approx 2.7275211479$, then we can numerically evaluate the coefficients to find 
\begin{center}
\begin{tabular}{c || r} 
 \hline \hline
 $a_{0,1}$ & -0.0931912038 \\ 
 \hline
 $a_{2,0}$ & -0.4972909772  \\ 
 \hline
 $a_{0,2}$ & 0.0093753194 \\ \hline
 $a_{2,1}$ & -0.0081152843 \\
 \hline
 $a_{0,3}$ & -0.0014671778 \\
 \hline
 $b_{3,0}$ & -0.4947603203 \\
 \hline
  $c_{0,1}$ & 0.0598478709 \\
 \hline
 $c_{2,0}$ & 1.08625864892 \\
 \hline
  $c_{0,2}$ & -0.0033359912\\
 \hline
 $c_{2,1}$ & -0.0002576496 \\
 \hline
  $c_{0,3}$& 0.0002892588  \\  
 \hline
 \hline
\end{tabular}
\end{center}
 We do not include the numerical values of $a_{4,0}$ and $c_{4,0}$, as both coefficients will eventually drop from our calculations. The list above suggests that $a_{0, 1}\ne c_{0, 1}$ and $b_{3, 0}\ne 0$. These facts will be crucial to Theorem \ref{theo:main}, so we prove both in the following lemma. 
\begin{lemm}\label{lemm:b30}
$a_{0,1} \neq c_{0,1}$ and $b_{3, 0}\ne 0$. 
\end{lemm}
\begin{proof}
 To see that $a_{0,1} \neq c_{0,1}$, consider the product $a_{0,1}c_{0,1}$. A direct calculation given in the companion Mathematica file shows
\begin{align}
a_{0,1} c_{0,1} = \frac{\Omega'(-2)\Omega'(1)}{4(2+\sigma)(\sigma -1)},\quad \Omega(k)=(k^2+\beta)^\mez. 
\end{align}
Because $k\Omega'(k) >0$ for all $k \in \mathbb{R}\setminus\{0\}$ and $-2 < \sigma < 1$, it follows that this product is negative. Necessarily, $a_{0,1} \neq c_{0,1}$. 

As for the coefficient $b_{3,0}$, it is an explicit but cumbersome function of $\beta_*$. It is convenient in what follows to rewrite $b_{3,0}$ in terms of $\gamma_1=(1+\beta_*)^\frac{1}{4}$. Doing so and using the resonance condition
\begin{align}
\left(4+\beta_* \right)^{1/4} = 3 - (1 + \beta_*)^{1/4}, \label{res_cond_1}
\end{align}
together with Mathematica's powerful symbolic algebra calculator, we can express $b_{3,0}$ more compactly as
\begin{align}
b_{3,0} = -\frac{(1+\gamma_1^2)\left(p(\gamma_1) + q(\gamma_1)\sqrt{\gamma_1^4 - 1} \right)}{r(\gamma_1)}, \label{b30_expression}
\end{align}
where
\begin{align}
p(\gamma_1) &= 66632 - 283193 \gamma_1 + 552058 \gamma_1^2 - 791360 \gamma_1^3 + 956648 \gamma_1^4 -  941661 \gamma_1^5 + 714646 \gamma_1^6 - 392544 \gamma_1^7 \nonumber \\ &\hspace{1.5cm} + 145056 \gamma_1^8 - 30331 \gamma_1^9 + 622 \gamma_1^{10} + 1440 \gamma_1^{11} - 336 \gamma_1^{12} + 17 \gamma_1^{13} + 2 \gamma_1^{14},\\
q(\gamma_1) &= -6656 - 102903 \gamma_1 + 356580 \gamma_1^2 - 545119 \gamma_1^3 + 508794 \gamma_1^4 -  312190 \gamma_1^5 + 126944 \gamma_1^6 - 32062 \gamma_1^7 \nonumber \\ &\hspace{1.5cm} + 3812 \gamma_1^8 + 213 \gamma_1^9 - 
 100 \gamma_1^{10} - 3 \gamma_1^{11} + 2 \gamma_1^{12}, \\
r(\gamma_1) &= 64 \sqrt{-\left(\left(\gamma_1-3\right) \gamma_1\right)} \left(\left(\gamma_1-3\right) \gamma_1+5\right) \left(\left(\gamma_1-3\right) \gamma_1+6\right) \left(\sqrt{\gamma_1^4-1}+\left(\gamma_1-6\right) \gamma_1+11\right) \nonumber \\ &\quad\cdot \left(\gamma_1 \left(\sqrt{\gamma_1^4-1}+\left(\gamma_1-1\right) \gamma_1+1\right)-1\right){}^2.
\end{align}
It follows from \eqref{res_cond_1}  that $\gamma_1$ is the only positive real solution of $2\gamma_1^3 - 9\gamma_1^2 + 18\gamma_1 - 13 = 0$. A quick application of the intermediate value theorem shows that $1< \gamma_1 <2$. In particular, all factors in $r_1(\gamma_1)$ are positive, hence  $r(\gamma_1)>0$ and \eqref{b30_expression} is well-defined. 

Arguing by contradiction, let us suppose that $b_{3, 0}=0$. Then $p(\gamma_1) + q(\gamma_1)\sqrt{\gamma_1^4 - 1}=0$, so that  the polynomial 
\[
g(\xi) = p(\xi)^2 -q(\xi)^2\left(\xi^4 - 1 \right)
\]
must have a zero at $\gamma_1$. According to the argument principle, we must have
\begin{align*}
\int_{\partial B} \frac{g'(\xi)}{g(\xi)} d\xi \neq 0,
\end{align*}
where $B$ is any open ball of sufficiently small radius that contains only the root of $g$ at $\gamma_1$. On the other hand, one can have Mathematica compute the Laurent expansion of $g'(\xi)/g(\xi)$ about $\xi = \gamma_1$. It is important to emphasize that this calculation is purely algebraic and does not rely on any numerical computations: we use Mathematica only to avoid tedious algebra. The Laurent expansion of $g'(\xi)/g(\xi)$ has no term of order $1/(\xi - \gamma_1)$. Thus, the residue theorem implies
\begin{align}
\int_{\partial B} \frac{g'(\xi)}{g(\xi)} d\xi = 2\pi i \underset{\xi = \gamma_1}{\textrm{Res}}\left(\frac{g'(\xi)}{g(\xi)} \right) = 0.
\end{align}
This contradiction proves that $b_{3,0} \neq 0$.
\end{proof}

With the appropriate lemmas in place, we are now ready to prove Theorem \ref{theo:main}. 

\large{ \textbf{ Proof of Theorem \ref{theo:main}}}
\vspace{.5cm}\\
The characteristic polynomial of $\textrm{L}_{\eps, \delta} - i\sigma I$ is 
\bq\label{charL} 
\det (\textrm{L}_{\eps, \delta}- i\sigma I -\lambda I) = \ld^2-i(A+C)\ld -AC-B^2 ,
\eq
whose discriminant is the real-valued function
  \bq\label{def:Deltaed}
  \Delta(\varepsilon,\delta) =-(A-C)^2+4B^2.
  \eq
  The eigenvalues of $\textrm{L}_{\eps, \delta}$ are 
\bq\label{ld:ABCDelta}
\lambda_\pm = i\Big(\sigma +\frac{1}2 (A+C) \Big) \pm \frac12\sqrt{\Delta(\eps,\delta)},  
\eq 
where $ A+C$ is the trace of $\text{L}_{\varepsilon,\delta}-i\sigma I$ and so is real-valued. Since $A$, $B$, and $C$ are real analytic in $(\eps, \delta)$ in a neighborhood of $(0, 0)$, so are $A+C$ and 
$\Delta(\eps, \delta)$. Moreover, it follows from the expansions of $A$, $B$, and $C$ that $\Delta(\varepsilon,\delta) = \cO(\delta^2)$ and $A+C = \cO(\delta)$ as $(\varepsilon,\delta) \rightarrow (0,0)$. We will complete the proof of Theorem \ref{theo:main} by {\it showing that the characteristic polynomial \eqref{charL} has a root with positive real part}, that is, $\Delta(\eps,\delta)>0$ for suitably small $\eps$ and $\delta$.

 If we expand $\Delta(\eps, \delta)$ about $(\eps, \delta)=(0, 0)$ using \eqref{expand:ABC} only up to the third-order order terms, the remainder in $(A-C)^2$ will be of order $O(|\eps|^5+|\delta|^5)$, which is lower than the main term $4b_{3,0}^2\varepsilon^6$ coming from $B^2$. This may suggest that one needs to further expand $A$ and $C$ up to order  $O(|\eps|^6+|\delta|^6)$, which would warrant many more heavy calculations. Fortunately, a formal application of dominant balance to $\Delta(\varepsilon,\delta)$ suggests that $\delta =O(\eps^2)$ is the correct distinguished limit that will achieve $\Delta(\varepsilon,\delta) > 0$ for small $\varepsilon$. As a result, it will be sufficient in this work to expand $A$ and $C$ just beyond third order in $\varepsilon$ and $\delta$ to include terms proportional to $\varepsilon^4$, as written in Lemma \ref{freqProp}. All remaining fourth-order terms in the expansions of $A$ and $C$ become proportional to $\varepsilon^7$ in the expansion of $\Delta(\varepsilon,\delta)$ once $\delta = O(\varepsilon^2)$, which is higher order than the crucial term $4b_{3,0}^2\varepsilon^6$. 
 
Taking inspiration from dominant balance, we renormalize $\delta$ as follows:
\bq\label{delta:eps}
\delta =\eps^2 \ka
\eq
so that expansions \eqref{expand:ABC} become
\bq\label{expand:ABC:2}
\begin{aligned}
A &= a_{0,1} \eps^2 \ka + a_{2,0}\varepsilon^2 + a_{0,2}\eps^4 \ka^2 + a_{2,1}\varepsilon^4\ka+a_{4, 0}\eps^4+ O(\eps^5),\\
B&=b_{3,0}\varepsilon^3 +O(\eps^4),\\
C &=c_{0,1}\eps^2 \ka + c_{2,0}\varepsilon^2 + c_{0,2}\eps^4 \ka^2 + c_{2,1}\varepsilon^4\ka +c_{4, 0}\eps^4+ O(\eps^5),
\end{aligned}
\eq
where the remainders depend on $\ka$ to be chosen.  We now define
\begin{align*}
\Delta'(\varepsilon,\kappa) = \Delta(\varepsilon,\varepsilon^2\kappa).
\end{align*}
Substituting the expansions \eqref{expand:ABC:2} into the determinant $-(A-C)^2+4B^2$ and rearranging terms in increasing powers of $\varepsilon$, we obtain the following expansion for $\Delta'(\varepsilon,\kappa)$: 
\begin{align} \label{Delta':order4}
 \Delta'(\varepsilon,\ka) &=-\left[(a_{0, 1}-c_{0, 1})\ka+(a_{2, 0}-c_{2, 0})\right]^2\eps^4\nonumber\\
 &\hspace{1cm}-2(a_{4, 0}-c_{4, 0})\left[(a_{0, 1}-c_{0, 1})\ka +(a_{2, 0}-c_{2, 0})\right]\eps^6 +(E(\ka)+ 4b_{3,0}^2)\eps^6+ O(\eps^7),
 \end{align}
 where
 \bq
 \begin{aligned}
 E(\ka)&=   -2(a_{0,1}-c_{0,1})(a_{0,2} - c_{0,2})\ka^3 -2\big((a_{0,2}-c_{0,2})(a_{2,0}-c_{2,0}) + (a_{0,1}-c_{0,1})(a_{2,1}-c_{2,1}) \big)\ka^2\\
 &\hspace{1cm}- 2(a_{2,0}-c_{2,0})(a_{2,1}-c_{2,1})\ka\\
   &=-2\ka\left[(a_{0,1}-c_{0,1})\ka +(a_{2,0}-c_{2,0})\right] \left[(a_{0,2}-c_{0,2})\ka+(a_{2,1}-c_{2,1})\right].
 \end{aligned}
 \eq
 This expansion shows that $\Delta'(\eps, \ka)=O(\eps^4)$ as $\varepsilon \rightarrow 0$. Therefore, if we define
 \bq\label{def:Delta''}
 \Delta''(\eps, \ka):=\eps^{-4}\Delta'(\eps, \ka)\equiv \eps^{-4}\Delta(\eps, \eps^2\ka),
 \eq
 then \eqref{Delta':order4} implies 
\begin{align} \label{Delta'':0}
 \Delta''(\varepsilon,\ka) &=-\left[(a_{0, 1}-c_{0, 1})\ka+(a_{2, 0}-c_{2, 0})\right]^2\nonumber\\
 &\hspace{1cm}-2(a_{4, 0}-c_{4, 0})\left[(a_{0, 1}-c_{0, 1})\ka +(a_{2, 0}-c_{2, 0})\right]\eps^2 +E(\ka)\eps^2+ 4b_{3,0}^2\eps^2+O(\eps^3).
 \end{align}
 Since $\Delta(\eps, \eps^2\ka)=\eps^4 \Delta''(\eps, \ka)$, we deduce from \eqref{Delta'':0} that in order for $\Delta(\eps, \eps^2\ka)$ to be positive for all $\eps$ small, it is necessary that $(a_{0, 1}-c_{0, 1})\ka+(a_{2, 0}-c_{2, 0})$ vanishes. Thus we seek $\ka$ of the form 
 \[
 \ka =-\frac{a_{2, 0}-c_{2, 0}}{a_{0, 1}-c_{0, 1}}+\eps \tt=:\ka_0+\eps\tt  
 \]
 with $\tt$ to be determined. In light of Lemma \ref{lemm:b30}, $\kappa_0$ is well-defined. Now, we note that $(a_{0, 1}-c_{0, 1})\ka +(a_{2, 0}-c_{2, 0})$ is the common factor in all the main terms in \eqref{Delta'':0}, except for $4b_{3,0}^2\eps^2$. 
In particular, the coefficients $a_{4, 0}$ and $c_{4, 0}$ do not contribute, and we obtain that
 \bq
 \Delta'''(\eps, \tt):=\eps^{-2}\Delta''(\eps, \ka_0+\eps \tt)
 \eq
satisfies
 \bq\label{def:Delta3}
  \Delta'''(\eps, \tt):=-(a_{0, 1}-c_{0, 1})^2\tt^2+ 4b_{3,0}^2+r(\eps, \tt)
 \eq
 with a remainder $r(\eps, \tt)=O(\eps)$. Therefore, for sufficiently small $\eps$, we have $\Delta'''(\eps, \tt)>0$ provided 
\bq\label{cd:tt}
 |\tt|< \kappa_1, \quad \textrm{where} \quad \kappa_1 := \frac{2|b_{3,0}|}{|a_{0, 1}-c_{0, 1}|},
\eq
which is well-defined by Lemma ~\ref{lemm:b30}. Thus, we have shown that  $\Delta(\eps, \delta)>0$ if  $\eps \ne 0$ is  sufficiently small and 
\bq\label{expand:delta}
\delta =\eps^2(\ka_0+\eps \tt)=\eps^2 \ka_0+\eps^3 \tt \quad \textrm{with} \quad |\theta| < \kappa_1.
\eq
Returning to the expression for the eigenvalues \eqref{ld:ABCDelta}, we find that $\text{Re}\lambda_+=\mez \sqrt{\Delta(\eps, \delta)} > 0$. Therefore, small-amplitude Stokes waves are unstable with respect to transverse perturbations with wave numbers $\alpha = \sqrt{\beta_*+\delta}$ for any  $\delta \in (\varepsilon^2\kappa_0-\varepsilon^3\ka_1, \varepsilon^2\kappa_0 +\varepsilon^3\ka_1)$, provided $\varepsilon$ is sufficiently small. Finally, since \begin{equation} \Delta(\eps, \delta)=\eps^6\Delta'''(\eps, \tt)=\eps^6\left[4b_{3, 0}^2-(a_{0, 1}-c_{0, 1})^2\tt^2\right]+O(\eps^7), \label{Delta_exp}
\end{equation} we have $\text{Re}\lambda_+ = \cO(\varepsilon^3)$ as $\varepsilon \rightarrow 0$, completing the proof.  \qed 
 \begin{coro}\label{cor:ellipse}
 For sufficiently small $\eps$, the curve $(-\ka_1, \ka_1)\ni \tt\mapsto \ld_+(\eps, \tt)$ (resp. $(-\ka_1, \ka_1)\ni \tt\mapsto \ld_-(\eps, \tt)$) is within $O(\eps^4)$  distance to the entire right (resp. left) half of the ellipse 
\begin{align}\label{form:ellipse}
\frac{\lambda_r^2}{\left(b_{3,0}\varepsilon^3\right)^2} + \frac{\left(\lambda_i -\sigma - \left(\frac{a_{0,1}c_{2,0}-a_{2,0}c_{0,1}}{(a_{0,1}-c_{0,1})} \right) \varepsilon^2\right)^2}{\left(\frac{b_{3,0}(a_{0,1}+c_{0,1})}{a_{0,1}-c_{0,1}}\varepsilon^3\right)^2} =1
\end{align}
in the complex plane. 
\end{coro} 
\begin{proof}
We begin with the expansion of $\Delta(\varepsilon,\delta)$ given in \eqref{Delta_exp} as well as the expansion 
\bq\label{expand:A+C}
\begin{aligned}
A+C &=(a_{0, 1}+c_{0, 1})\delta+(a_{2, 0}+c_{2, 0})\eps^2+O(\eps^4)\\
&= \frac{2(a_{0,1}c_{2,0}-a_{2,0}c_{0,1})}{a_{0,1}-c_{0,1}}\varepsilon^2 +(a_{0,1}+c_{0,1}) \theta\eps^3 +O(\eps^4),
\end{aligned}
\eq
where we have used \eqref{expand:delta} to replace $\delta$. Substituting both of these expansions into the formula for the unstable eigenvalues given by \eqref{ld:ABCDelta}, we obtain the expansion 
\bq\label{finalexpansion:ldpm}
\begin{aligned}
\ld_\pm\equiv \ld_\pm(\eps, \tt)
&= i\left[\sigma+\frac{a_{0,1}c_{2,0}-a_{2,0}c_{0,1}}{(a_{0,1}-c_{0,1})}\varepsilon^2 +\frac{a_{0, 1}+c_{0, 1}}{2}\tt\eps^3\right]\\
&\qquad \pm\mez \left[4b_{3, 0}^2-(a_{0, 1}-c_{0, 1})^2\tt^2\right]^\mez|\eps|^3+O(\eps^4).
\end{aligned}
\eq 
For sufficiently small $\eps$ and for $\tt$ satisfying \eqref{cd:tt}, it follows from \eqref{finalexpansion:ldpm} that $\ld_+(\eps, \tt)$ is within $O(\eps^4)$ distance to $(\lambda_r, \lambda_i)\in \Rr^2$, where
\begin{align*}
\lambda_r&=\mez \left[4b_{3, 0}^2-(a_{0, 1}-c_{0, 1})^2\tt^2\right]^\mez|\eps|^3,\\
\lambda_i&=\sigma+\frac{a_{0,1}c_{2,0}-a_{2,0}c_{0,1}}{2(a_{0,1}-c_{0,1})}\varepsilon^2 +\frac{a_{0, 1}+c_{0, 1}}{2}\tt\eps^3.
\end{align*}
By eliminating $\theta$ from the preceding expressions, we  find that $(\lambda_r, \lambda_i)$ lies on the ellipse \eqref{form:ellipse}. We note that, as $\tt$ varies in the interval $(-\ka_1, \ka_1)$, $(\lambda_r, \lambda_i)$ traces the entire right half of the ellipse \eqref{form:ellipse}. Lastly, we note that if the coefficients above are numerically evaluated, we obtain the ellipse \eqref{ellipsenum} given in the introduction.
\end{proof}


\appendix
\section{Stokes expansion}\label{appendix:Stokes}
We recall the water wave system with zero Bernoulli constant:
  \begin{align}\label{ww:eq1}
& \Delta_{x, y} \phi =0\quad \text{in } \Omega, \\ \label{ww:eq2}
&-c\p_x\phi + g\eta + \tfrac12 |\na_{x, y}\phi|^2 =0\quad \text{ on } \{y=\eta(x)\}, \\ \label{ww:eq3}
 &\p_y\phi + (c-\p_x\phi)\p_x\eta=0\quad  \text{ on } \{y=\eta(x)\}, \\ \label{ww:eq4}
 &\na_{x, y}\phi \to 0 \text{ as } y\to -\infty. 
\end{align}
Using superscripts we Taylor-expand the unknowns,  
\begin{align*}
&\eta = \eps\eta^1 + \eps^2\eta^2 +\eps^3\eta^3+\eps^4\eta^4 \dots,\\
&\phi = \eps\phi^1 + \eps^2\phi^2 + \eps^3\phi^3+\eps^4\phi^4+\dots,\\
&c = c^0+\eps c^1 +\eps^2c^2  +\eps^3c^3+\dots,
\end{align*}
  and reserve subscripts for derivatives. Here $\phi^j(x, y): \Rr\times \Rr_-\to \Rr$. It was found in \cite{NguyenStrauss} that 
  \bq\label{Stokes:3}
  \begin{aligned}
  &\eta^1=\cos x,\quad \eta^2=\mez \cos(2x),\quad\eta^3=\frac18\cos x+\frac38 \cos(3x),\\
  &\phi^1=\sqrt{g}e^y\sin(x),\quad\phi^2=\phi^3=0,\\
  &c^0=\sqrt{g},\quad c^1=0,\quad c^2=\frac{\sqrt{g}}{2}.
  \end{aligned}
  \eq
  Our goal is to find the next coefficients $c^3$, $\eta^4$ and $\phi^4$. The coefficient of $\eps^4$ in \eqref{ww:eq2} is
  \[
  \begin{aligned}
  &-c^0(\phi_x)^4-c^2(\phi_x)^2-c^3(\phi_x)^1+g\eta^4+\mez |(\na_{x, y}\phi)^2|+(\na_{x,y}\phi)^1\cdot (\na_{x,y}\phi)^3\\
  &=-c^0\left\{\phi^4_x+\phi^3_{xy}\eta^1+\phi^2_{xy}\eta^2+\phi^1_{xy}\eta^3+\mez\phi^2_{xyy}\eta^1\eta^1+\phi^1_{xyy}\eta^1\eta^2+\frac16\phi^1_{xyyy}\eta^1\eta^1\eta^1\right\}\\
  &\quad-c^2(\phi^2_x+\phi^1_{xy}\eta^1)-c^3\phi^1_x+g\eta^4+\mez[\phi^2_x+\phi^1_{xy}\eta^1]^2+\mez[\phi^2_y+\phi^1_{yy}\eta^1]^2\\
  &\quad +\phi^1_x\phi^1_{xy}\eta^2+\mez \phi^1_x\phi^1_{xyy}\eta^1\eta^1+\phi^1_y\phi^1_{yy}\eta^2+\mez\phi^1_y\phi^1_{yyy}\eta^1\eta^1.
  \end{aligned}
  \]
  Substituting \eqref{Stokes:3}  then equating this to $0$, we obtain
  \bq\label{Stokes:4:eq1}
  \phi^4_x+c^0\left\{-\frac{1}{6}\cos(2x)+\frac13\cos(4x)\right\}+c^3\cos x-\sqrt{g}\eta^4=0
    \eq
    at $y=0$. Next we calculate the coefficient of $\eps^4$ in \eqref{ww:eq3}
  \bq\label{Stokes:4:eq2}
  \begin{aligned}
&(\phi_y\phi)^4+c^0\eta_x^4+c^2\eta_x^2+c^3\eta^1_x-(\phi_x)^1\eta_x^3-(\phi_x)^2\eta_x^2-(\phi_x)^3\eta_x^1\\
&=\phi^4_y+\phi^3_{yy}\eta^1+\phi^2_{yy}\eta^2+\phi^1_{yy}\eta^3+\mez \phi^2_{yyy}\eta^1\eta^1+\phi^1_{yyy}\eta^1\eta^2+\frac16\phi^1_{yyyy}\eta^1\eta^1\eta^1\\ 
&\quad +c^0\eta^4_x+c^2\eta^2_x+c^3\eta^1_x-\phi^1_x\eta^3_x-\eta^2_x(\phi^2_x+\phi^1_{xy}\eta^1)-\eta^1_x\left\{\phi^3_x+\phi^2_{xy}\eta^1+\phi^1_{xy}\eta^2+\mez\phi^1_{xyy}\eta^1\eta^1\right\}\\
&=\phi_y^4+c^0\left\{\frac23\sin(2x)+\frac43\sin(4x)\right\}+c^0\eta_x^4-c^3\sin x,\quad y=0.
\end{aligned}
  \eq
  Differentiating \eqref{Stokes:4:eq1} in $x$ then adding to \eqref{Stokes:4:eq2} we obtain the boundary condition for $\phi$ 
  \bq
  \phi^4_{xx}+\phi^4_y+\sqrt{g}\sin(2x)-2c^3\sin x=0,\quad y=0.
  \eq
  In addition, we have in view of \eqref{ww:eq1} and \eqref{ww:eq4} that $\Delta_{x, y} \phi^4=0$ in $\{y<0\}$ and $\na_{x, y}\phi\to 0$ as $y\to -\infty$. Choosing $c^3=0$ and seeking the solution of the form $\phi(x, y)=a\sin(2x)e^{by}$ with $b<0$, we find
  \bq\label{phi4}
  \phi^4(x, y)=\frac{\sqrt{g}}{2}\sin(2x)e^{2y}.
  \eq
  Inserting \eqref{phi4} in \eqref{Stokes:4:eq1} gives
  \bq
  \eta^4(x)=\sqrt{g}\left\{\frac{5}{6}\cos(2x)+\frac13\cos(4x)\right\}.
  \eq
  Finally, using $\psi(x)=\phi(x, \eta(x))$ we obtain
  \bq
  \begin{aligned}
  \psi^4(x)&=\eps\sqrt{g}\sin x+\eps^2\frac{\sqrt{g}}{2}\sin(2x)+\eps^3\frac{\sqrt{g}}{4}\big(3\sin x\cos(2x)+\sin x \big)+\\
  &\quad+\eps^4\sqrt{g}\left\{\frac{5}{12}\sin(2x)+\frac{1}{3}\sin(4x)\right\}+O(\eps^5).
  \end{aligned}
  \eq
\section{Proof of Proposition  \ref{prop:expandpq}}\label{appendix:pq}
We recall from Remark \ref{rema:XZzeta} that $\zeta$ is analytic in $\eps$ with values in Sobolev spaces. From \eqref{def:pq} and \eqref{def:B*V*} we have
\[
p=\frac{c_*-\zeta_\sharp V_*}{\zeta'},\quad q=-p\p_x(\zeta_\sharp B_*),
\]
where
\[
B^*=\frac{G(\eta^*)\psi^*+\p_x\psi^*\p_x\eta^*}{1+|\p_x\eta^*|^2},\qquad V^*=\p_x\psi^*-B^*\p_x\eta^*.
\] 
By Theorem  1.2 in \cite{Berti-DN}, for any $s>\frac52$, the mapping 
\bq
H^s(\T)\ni\eta\mapsto G(\eta)\in \cL(H^s, H^{s-1})
\eq
is analytic on any bounded set. Combining this with the analyticity in $\eps$ of $\eta^*$ and $\psi^*$ and the boundedness of the Dirichlet-Neumann operator, we deduce that $B^*$ and $V^*$ are analytic in $\eps$ with values in Sobolev spaces. In conjunction with analyticity in $\eps$ of $c_*$ and $\zeta$, this yields analyticity in $\eps$ of $p$ and $q$ and $\frac{1+q}{\zeta'}$. The remainder of this proof is devoted to expansions in powers of $\eps$ for these analytic functions. 
\subsection{Shape-derivative}
We view $G(\eta)\psi$ as a linear operator with respect to $\psi$. Let $G'(\eta)\ol\eta\psi$ denote the shape-derivative of $G(\eta)\psi$ with  respect to $\eta$ evaluated at $\ol\eta$. We recall from Theorem \ref{theo:shapederi} that the shape-derivative is given by 
\bq\label{shapederi}
 G'(\eta)\ol{\eta}\psi=-G(\eta)(\ol\eta B(\eta)\psi)-\p_x(\ol\eta V(\eta)\psi),
\eq
where 
\bq\label{BV}
B(\eta)\psi=\frac{G(\eta)\psi+\p_x\psi\p_x\eta}{1+|\p_x\eta|^2},\qquad V(\eta)\psi=\p_x\psi-\p_x\eta B(\eta)\psi.
\eq
Again, we view $B(\eta)\psi$ and $V(\eta)\psi$ as linear operators with respect to $\psi$. Next, to calculate the second derivative $G''(\eta)\psi$ we reapply \eqref{shapederi} and obtain
\bq
G''(\eta)[\ol{\eta}, \wt \eta]\psi=-G'(\eta)\wt \eta(\ol\eta B(\eta, \psi))-G(\eta)(\ol \eta B'(\eta)\wt \eta\psi)+\p_x\{\ol\eta \p_x\eta B'(\eta)\wt\eta\psi+\ol\eta \p_x\wt \eta B(\eta)\psi\},
\eq
where
\bq\label{dB}
B'(\eta)\wt\eta\psi=\frac{1}{1+|\p_x\eta|^2}\{G'(\eta)\wt \eta\psi+\p_x\wt\eta\p_x\psi \}-\frac{2\p_x\eta\p_x\wt\eta}{(1+|\p_x\eta|^2)^2}\{G(\eta)\psi+\p_x\eta\p_x\psi\}.
\eq
Now  substituting $\eta=0$ yields
\bq\label{G'}
\begin{aligned}
& G(0)=|D|,\quad B(0)=|D|,\quad V(0)=\p_x,\\
& G'(0)\ol\eta\psi=-|D|(\ol\eta |D|\psi)-\p_x(\ol\eta \p_x\psi),
\end{aligned}
\eq
whence
\bq\label{B'(0)}
B'(0)\wt\eta\psi=G'(0)\wt \eta\psi+\p_x\wt\eta\p_x\psi=-|D|(\wt\eta |D|\psi)-\p_x(\wt\eta \p_x\psi)+\p_x\wt\eta\p_x\psi
\eq
and
\bq\label{G''}
\begin{aligned}
G''(0)[\ol{\eta}, \wt \eta]\psi&=-G'(0)\wt \eta(\ol\eta|D|\psi)-|D|(\ol \eta B'(0)\wt \eta\psi)+\p_x\{\ol\eta |D|\psi\p_x\wt \eta\}\\
&=|D|\{\wt\eta |D|(\ol\eta|D|\psi)\}+\p_x\{\wt\eta \p_x(\ol\eta|D|\psi)\}\\
&\quad-|D|\left\{\ol\eta \left[-|D|(\wt\eta |D|\psi)-\p_x(\wt\eta \p_x\psi)+\p_x\wt\eta\p_x\psi\right]\right\}+\p_x\{\ol\eta |D|\psi\p_x\wt \eta\}\\
&=|D|\{\wt\eta |D|(\ol\eta|D|\psi)\}+\p_x\{\wt\eta \p_x(\ol\eta|D|\psi)\}\\
&\quad+|D|\left\{\ol \eta |D|(\wt\eta |D|\psi)\right\}+|D|\left\{\ol\eta\p_x(\wt\eta \p_x\psi)\right\}-|D|\left\{\ol\eta\p_x\wt\eta\p_x\psi\right\}+\p_x\{\ol\eta |D|\psi\p_x\wt \eta\}.
\end{aligned}
\eq
Simplifying, we arrive at 
\bq
G''(0)[\ol{\eta}, \wt \eta]\psi=|D|\{\wt\eta |D|(\ol\eta|D|\psi)\}+|D|\left\{\ol \eta |D|(\wt\eta |D|\psi)\right\}+|D|\left\{\ol\eta\wt\eta\p^2_x\psi\right\}+\p^2_x\{\ol\eta\wt \eta |D|\psi\}.
\eq
\subsection{Expansion of $p$} We use \eqref{B'(0)} to expand $B^*$ up to $\eps^2$:
\bq
\begin{aligned}
B^*(x):&=B(\eta^*)\psi^*=B(0)\psi^*+B'(0)\eta^*\psi^*+O(\eps^2)\\
&=|D|\psi^*-|D|(\eta^* |D|\psi^*)-\p_x(\eta^* \p_x\psi^*)+\p_x\eta^*\p_x\psi^*+O(\eps^2)\\
&=|D|(\eps \sin x+\mez\eps^2\sin(2x))-|D|\big((\eps \cos x)(\eps \sin x)\big)-\p_x\big((\eps \cos x)(\eps \cos x)\big)\\
&\qquad-(\eps \sin x)(\eps \cos x)+O(\eps^3)\\
&=(\eps \sin x+\eps^2\sin(2x))-\eps^2\big\{\sin(2x) -\sin(2x)+\mez\sin(2x)\big\}+O(\eps^3)\\
&=\eps \sin x+\mez \eps^2 \sin(2x)+O(\eps^3).
\end{aligned}
\eq 
This implies the expansion for $V^*$ up to $\eps^3$:
\bq
\begin{aligned}
V^*(x)&=\p_x\psi^*-\p_x\eta^*B^*\\
&=\eps \cos x+\eps^2\cos (2 x)+\frac{1}{4} \eps^3\big\{3\cos x\cos(2x)-6\sin x\sin(2x)+\cos x\big\}\\
&\quad-\{-\eps \sin x- \eps^2\sin(2x)\}\{\eps \sin x+\mez \eps^2 \sin(2x)\}+O_\eps(\eps^4)\\
&=\eps\cos x+\frac{\eps^2}{2}(1+\cos(2x))+\frac{\eps^3}{8}(5\cos x+3\cos(3x))+O(\eps^4).
\end{aligned}
\eq
Regarding $\zeta$,  we write 
\[
\zeta=x+\eps\zeta^1+\eps^2\zeta^2+\eps^3\zeta^3+O(\eps^4),\quad \eta^*=\eps\eta^1+\eps^2\eta^2+\eps^3\eta^3+O(\eps^4),
\]
so that by Taylor expansion,
\bq\label{compose:zeta}
\begin{aligned}
\eta^*(\zeta(x))&=\eps \eta^1(x)+\eps^2(\zeta^1\p_x\eta^1(x)+\eta^2(x))\\
&\quad+\eps^3\left\{\zeta^2(x)\p_x\eta^1(x)+\mez(\zeta^1(x))^2\p^2_x\eta^1(x)+\zeta^1(x)\p_x\eta^2(x)+\eta^3(x)\right\}+O(\eps^4).
\end{aligned}
\eq
It follows from \eqref{z1} that 
\bq
\zeta(x)=X(x, 0)=x-\frac{i}{2\pi}\sum_{k\ne 0}e^{ikx}\mathrm{sign}(k)\widehat{\eta^*\circ \zeta}(k).
\eq
We equate the coefficients of $\eps^j$ on both sides to determine $\zeta^j$. It was  already found in \cite{NguyenStrauss} that $\zeta^1(x)=\sin x$ and $\zeta^2(x)=\sin(2x)$. At the order $\eps^3$ we have
\begin{align*}
\zeta^3(x)=-\frac{i}{2\pi}\sum_{k\ne 0}e^{ikx}\mathrm{sign}(k)\widehat{g}(k),
\end{align*}
where
\begin{align*}
g(x)&=\zeta^2(x)\p_x\eta^1(x)+\mez(\zeta^1(x))^2\p^2_x\eta^1(x)+\zeta^1(x)\p_x\eta^2(x)+\eta^3(x)=-\cos x+\tdm\cos(3x).
\end{align*}
Direct calculations give $\zeta^3(x)=-\sin x+\tdm \sin(3x)$ and thus
\bq\label{expand:zeta}
\zeta(x)=x+\eps\sin x+\eps^2\sin(2x)+\eps^3\big(-\sin x+\tdm \sin(3x)\big)+O(\eps^4).
\eq
Inserting \eqref{expand:zeta} back in \eqref{compose:zeta} gives
\bq\label{etazeta}
\eta^*(\zeta(x))=\eps \cos x+\eps^2\big(\cos(2x)-\mez\big)+\eps^3\big(\tdm \cos(3x)-\cos x \big)+O(\eps^4).
\eq
Using \eqref{compose:zeta} with $V^*$ in place of $\eta^*$ we obtain
\bq\label{zetaV}
\begin{aligned}
V^*(\zeta(x))&=\eps V^1(x)+\eps^2(\zeta^1\p_xV^1(x)+V^2(x))\\
&\quad+\eps^3\left\{\zeta^2(x)\p_xV^1(x)+\mez(\zeta^1(x))^2\p^2_xV^1(x)+\zeta^1(x)\p_xV^2(x)+V^3(x)\right\}+O(\eps^4)\\
&=\eps \cos x+\eps^2\cos(2x)+\eps^3\left\{\sin(2x)(-\sin x)+ \mez\sin^2x(-\cos x)+\sin x(-\sin(2x))\right.\\
&\quad\left.+\frac{5}{8}\cos x+\frac{3}{8}\cos(3x)\right\}+O(\eps^4)\\
&=\eps \cos x+\eps^2\cos(2x)+\eps^3\big(-\mez \cos x+\tdm \cos(3x)\big)+O(\eps^4).
\end{aligned}
\eq
 Combining \eqref{expand:zeta} and \eqref{zetaV} yields the expansion for $p$
\bq\label{expand:p:proof}
p(x)=\frac{c^*-V^*(\zeta(x))}{\zeta'(x)}=1-2\eps \cos x+\eps^2\big(\tdm-2\cos(2x)\big)+\eps^3\big(3\cos x-3\cos(3x)\big)+O(\eps^4).
\eq
\subsection{Expansion of $q$ and $\frac{1+q}{\zeta'}$}
Using \eqref{G'} and \eqref{G''} we expand 
\bq
\begin{aligned}
G(\eta^*)\psi^*&=G(0)\psi^*+G'(0)\eta^*\psi^*+\mez G''(0)[\eta^*, \eta^*]\psi^* + O(\eps^4)\\
&=|D|\psi^*-|D|(\eta^* |D|\psi^*)-\p_x(\eta^* \p_x\psi^*)\\
&\quad +\mez |D|\left\{2\eta^* |D|(\eta^*|D|\psi^*)+\eta^*\eta^*\p^2_x\psi^*\right\}
+\mez\p^2_x\{\eta^*\eta^* |D|\psi^*\} + O(\eps^4)\\
&=\eps \sin x+\eps^2\sin(2x)+\eps^3\big(\frac58\sin x+\frac98\sin(3x)\big)+O(\eps^4).
\end{aligned}
\eq
Consequently,
\bq
B^*(x)=\frac{G(\eta^*)\psi^*+\p_x\eta^*\p_x\psi^*}{1+|\p_x\eta^*|^2}=\eps \sin x+\mez \eps^2\sin(2x)+\eps^3\big(\frac38\sin(3x)-\frac18\sin(x)\big)+O(\eps^4).
\eq
Then using \eqref{compose:zeta} with $\eta^*$ replaced by $B^*$ and recalling \eqref{expand:zeta}, we obtain
\bq\label{expand:zetaB:3}
B^*(\zeta(x))=\eps \sin x+\eps^2\sin(2x)+\eps^3\big(\tdm \sin(3x)-\mez \sin x\big)+O(\eps^4).
\eq
Combining \eqref{expand:p:proof} and \eqref{expand:zetaB:3} yields
\bq\label{expand:q:proof}
\begin{aligned}
q(x)=-p(x)\p_x(\zeta_\sharp B^*(x))=-\eps \cos x+\eps^2(1-\cos(2x))+\eps^3\big(2\cos x-\tdm \cos(3x)\big)+O(\eps^4).
\end{aligned}
\eq
Finally, from \eqref{expand:zeta} and \eqref{expand:q:proof}, we deduce
\bq
\frac{1+q(x)}{\zeta'(x)}=1-2\eps \cos x+2\eps^2(1-\cos(2x))+\eps^3\big(4\cos x-3\cos(3x)\big)+O(\eps^4).
\eq 

\vspace{.1in}
{\noindent{\bf{Acknowledgment.}} 
The work of HQN was partially supported by NSF grant DMS-2205710. We thank Bernard Deconinck and Massimiliano Berti for enlightening discussions over several years.
}


\end{document}